\documentclass[11pt]{amsart}
\usepackage{amsfonts,latexsym,amsthm,amssymb,graphicx,enumitem}
\usepackage[all]{xy}
\usepackage{hyperref}
\usepackage[usenames]{color} 

\usepackage{marginnote}

\DeclareFontFamily{OT1}{rsfs}{}
\DeclareFontShape{OT1}{rsfs}{n}{it}{<-> rsfs10}{}
\DeclareMathAlphabet{\mathscr}{OT1}{rsfs}{n}{it}

\newtheorem{theorem}{Theorem}[section]
\newtheorem{lemma}[theorem]{Lemma}
\newtheorem{corollary}[theorem]{Corollary}
\newtheorem{proposition}[theorem]{Proposition}

{\theoremstyle{definition} \newtheorem{definition}[theorem]{Definition}}
{\theoremstyle{remark} \newtheorem{remark}[theorem]{Remark}
\newtheorem{example}[theorem]{Example}}

\numberwithin{equation}{section}

\newenvironment{itemized}{\begin{itemize}[itemsep=0ex, wide]}{\end{itemize}}

\newcommand{\Pbo}{{\mathbf {P}}}

\newcommand{\Abb}{{\mathbb{A}}}
\newcommand{\Cbb}{{\mathbb{C}}}

\newcommand{\Pbb}{{\mathbb{P}}}
\newcommand{\Qbb}{{\mathbb{Q}}}

\newcommand{\Zbb}{{\mathbb{Z}}}

\newcommand{\Sym}{\text{\rm Sym}}

\newcommand{\cA}{{\mathscr A}}

\newcommand{\cD}{{\mathscr D}}
\newcommand{\cE}{{\mathscr E}}
\newcommand{\cF}{{\mathscr F}}
\newcommand{\cI}{{\mathscr I}}

\newcommand{\cL}{{\mathscr L}}
\newcommand{\cM}{{\mathscr M}}
\newcommand{\cN}{{\mathscr N}}
\newcommand{\cO}{{\mathscr O}}
\newcommand{\cP}{{\mathscr P}}

\newcommand{\cS}{{\mathscr S}}
\newcommand{\cX}{{\mathscr X}}

\newcommand{\ux}{\underline x}

\newcommand{\one}{1\hskip-3.5pt1}
\newcommand{\cma}{{c_{\text{Ma}}}}
\newcommand{\csm}{{c_{\text{SM}}}}
\newcommand{\cfu}{{c_{\text{F}}}}
\newcommand{\cfj}{{c_{\text{FJ}}}}
\newcommand{\cvir}{{c_{\text{vir}}}}

\newcommand{\saf}{\,}
\DeclareMathOperator{\capdot}{\mathbin{\cap\hskip-4.75pt.}\,}
\newcommand{\Llambda}{{\reflectbox{\rotatebox[origin=c]{180}{\text{$\mathbb V$}}}}}
\newcommand{\Ggamma}{{\reflectbox{\rotatebox[origin=c]{180}{\text{$\mathbb L$}}}}}

\DeclareMathOperator{\rk}{rk}

\DeclareMathOperator{\codim}{codim}
\DeclareMathOperator{\Spec}{Spec}
\DeclareMathOperator{\Proj}{Proj}
\DeclareMathOperator{\Eu}{Eu}
\DeclareMathOperator{\Sing}{Sing}

\DeclareMathOperator{\Segre}{Segre}
\DeclareMathOperator{\Gr}{Gr}
\DeclareMathOperator{\Ch}{Ch}
\DeclareMathOperator{\Ff}{F}

\newcommand{\qede}{\hfill$\lrcorner$}

\begin{document}

\title{Segre classes and invariants of singular varieties}
\author{Paolo Aluffi}
\address{
Mathematics Department, 
Florida State University,
Tallahassee FL 32306, U.S.A.
}
\email{aluffi@math.fsu.edu}

\maketitle

\begin{abstract}
{\em Segre classes\/} encode essential intersection-theoretic information concerning
vector bundles and embeddings of schemes. In this paper we survey a range of applications
of Segre classes to the definition and study of invariants of singular spaces. We will focus on
several numerical invariants, on  different notions of characteristic classes for singular varieties,
and on classes of L\^e cycles. We precede the main discussion with a review of relevant background 
notions in algebraic geometry and intersection theory.
\end{abstract}


\section*{Contents}


\section{Introduction}\label{intro}
{\em Segre classes\/} are an important ingredient in Fulton-MacPherson intersection
theory: the very definition of intersection product may be given in terms of these classes,
as we will recall below. It is therefore not surprising that important invariants of 
algebraic varieties may be expressed in terms of Segre classes. 
The goal of this paper is to survey several invariants specifically arising in singularity theory 
which may be defined or recast in terms of Segre classes. Many if not all of these invariants
first arose in complex geometry; the fact that they can be expressed in purely algebraic
terms by means of Segre classes extends their definition to arbitrary algebraically
closed fields of characteristic zero. Tools specific to the theory of Segre classes yield
new information on these invariants, or clarify the relations between them. On the whole,
the language of Segre classes offers a powerful perspective in the study of these invariants.

We will begin with a general introduction to Segre classes and their role in intersection 
theory, in~\S\ref{sec:SegreClasses}; a hurried reader can likely skim through this section 
at first and come back to it 
as needed. The survey itself will focus on the following themes:
\begin{itemize}
\item Numerical invariants (\S\ref{sec:numinvs});
\item Characteristic classes (\S\ref{sec:Charcla});
\item L\^e cycles (\S\ref{sec:Le}).
\end{itemize}
A glance at the table of contents will reveal more specifics about the topics we chose.

One central result will be an expression for the Chern-Schwartz-MacPherson class 
of a (possibly singular) subvariety of a fixed ambient nonsingular variety, in terms
of the Segre class of an associated scheme: see the discussion in~\S\ref{ss:CSMgen} 
and especially Theorem~\ref{thm:CSMemb}. For example, the topological Euler
characteristic of a scheme embedded in a nonsingular compact complex variety 
may be computed in terms of this Segre class. 
In the case of hypersurfaces, or more generally local complete intersections, this 
result implies concrete formulas for (generalized) Milnor numbers and classes.
These formulas are explicit enough that they can be effectively implemented in
computer algebra systems such as Macaulay2 for subschemes of
e.g., projective space. Characteristic classes of singular 
varieties are also treated in detail in other contributions to this `Handbook of Geometry 
and Topology of Singularities'; see especially the papers by Jean-Paul Brasselet
\cite{charclahandbook} and by Roberto Callejas-Bedregal, Michelle Morgado, and 
Jos\'e Seade \cite{milnorhandbook}.
The relation between Segre classes and David Massey's L\^e cycles discussed 
in~\S\ref{sec:Le} is the result of joint work with Massey. L\^e cycles are the subject of
Massey's contribution to this Handbook,~\cite{masseyhandbook}.  

The role of Segre classes in singularity theory is certainly more pervasive than this
survey can convey; because of limitations of space (and of our competence) we had to 
make a rather narrow selection, at the price of passing in silence many important 
topics. Among these omissions, we mention:
\begin{itemized}
\item The careful study of multiplicities and Segre numbers by R\"udiger Achilles, Mirella
Manaresi, and collaborators, see e.g.,~\cite{MR3912658};
\item Work on the Buchbaum-Rim multiplicity, particularly by Steven Kleiman and Anders Thorup,
\cite{MR1282823,MR1393259};
\item Work by Terry Gaffney and Robert Gassler on Segre numbers and cycles, \cite{MR1703611},
briefly mentioned in~\S\ref{sec:Le};
\item Seminal work by Ragni Piene on Segre classes and polar varieties, \cite{Pienepolarclasses},
also only briefly mentioned;
\item Alternative uses of Segre classes in defining characteristic classes of singular varieties,
as developed by Kent Johnson \cite{MR463161} and Shoji Yokura \cite{MR857438};
\item Toru Ohmoto's work on Segre-SM classes and higher Thom polynomials~\cite{MR3585782};
\item Equivariant aspects and positivity questions, which have recently come to the fore
in the study of characteristic classes for Schubert varieties, see e.g., \cite{AMSS, posSMc}.\end{itemized}
Each of these topics would deserve a separate review article, and this list is in itself incomplete.

\smallskip

{\em Acknowledgments.} The author thanks the editors of the {\em Handbook of Geometry 
and Topology of Singularities\/} for the invitation to contribute this survey, and the referees
for constructive comments. 

The author acknowledges support from a 
Simons Foundation Collaboration Grant, award number 625561, 
and also thanks Caltech for the hospitality during the writing of this paper.


\section{Segre classes}\label{sec:SegreClasses}

In this section we review the general definition of {\em Segre class\/} used in the rest
of the article, and place it in the context of Fulton-MacPherson intersection theory.
The reader can safely skim through this section, coming back to it as it is 
referenced later in the survey. 
We also introduce a notion that will be frequently used in the rest, that is, the `singularity 
subscheme' of a hypersurface; \S\ref{ss:hyparr} is an extended example revolving around
the Segre class of this subscheme for hyperplane arrangements.

We work over an algebraically closed field $k$; in later considerations, $k$ will be assumed
to have characteristic~$0$. Schemes are assumed to be 
separated of finite type over $k$.
A {\em variety\/} is a reduced irreducible scheme; a 
{\em subvariety\/}\index{subvariety} 
of a scheme is a closed subscheme that is a variety.
By `point' we will mean {\em closed\/} point. 
An effective {\em Cartier divisor\/} (or, slightly abusing language, a 
{\em hypersurface\/})\index{hypersurface}
is a codimension-$1$ subscheme that is locally defined by a nonzero divisor. Cartier
divisors are zero-schemes of sections of line bundles. 
A 
{\em cycle\/}\index{algebraic cycle}
in a scheme is a formal integer linear combination of 
subvarieties. 
Two cycles are 
{\em rationally\index{rational equivalence} 
equivalent\/} if (loosely speaking) they are connected
by families parametrized by $\Pbb^1$. The 
{\em Chow group\/}\index{Chow group} 
of dimension-$k$ cycles 
of a scheme~$X$ modulo rational equivalence is denoted $A_k(X)$; the direct sum
$\oplus_k A_k(X)$ is denoted $A_*(X)$. We recall that a proper morphism $f\colon  X\to Y$
determines a covariant push-forward homomorphism $f_*\colon A_*(X)\to A_*(Y)$ 
preserving dimension, while a flat or l.c.i.~morphism $f$ 
determines a contravariant
pull-back/Gysin\index{Gysin homomorphism}
homomorphism~$f^*$. If $X$ is complete, that is, the structure 
morphism $X\to \Spec k$ is proper, then the push-forward of a class $\alpha$ via 
$A_*(X)\to A_*(\Spec k)=\Zbb$ is the 
{\em degree\/}\index{degree} 
of $\alpha$, denoted $\int\alpha$ 
or $\int_X \alpha$. Intuitively, $\int\alpha$ is the `number of points' in the zero-dimensional 
component of $\alpha$. Vector bundles determine {\em Chern classes,\/} which act as 
operators on the Chow group, and satisfy various compatibilities (such as the
`projection formula') with morphisms. The 
Chern\label{Chern class} 
class $c_i(E)\cap-$ of a vector bundle~$E$
on $X$ defines group homomorphisms 
$A_k(X)\mapsto A_{k-{\rk E}}(X)$. The `total'
Chern class of $E$ is the operator

\[
c(E) = 1+ c_1(E) + \cdots + c_{\rk E}(E)\saf.
\]
For $i>\rk E$, $c_i(E)=0$. If $\cO(D)$ is the line bundle corresponding to a Cartier
divisor $D$, the action of the operator $c_1(\cO(D))$ amounts to `intersecting by $D$':
if $V\subseteq X$ is a variety not contained in $D$, $c_1(\cO(D))\cap [V]$ is the
class of the Cartier divisor obtained by restricting $D$ to $V$; we write
$c_1(\cO(D))\cap \alpha=D\cdot \alpha$.
Every vector bundle $E\to X$ determines an associated projective bundle `of lines', which 
we denote $\pi: \Pbo(E)\to X$. This bundled is endowed with a tautological subbundle
$\cO_E(-1)$ of $\pi^*E$; its dual $\cO_E(1)$, which restricts to the line bundle of a 
hyperplane in each fiber of $\pi$, plays a distinguished role in the theory.

Our reference for these notions is William Fulton's text, \cite{85k:14004}; Chapters~3--5
of the survey~\cite{MR735435} offer an efficient and well-motivated summary.
A reader who is more interested in topological aspects will not miss much by assuming
throughout that $k=\Cbb$ and replacing the Chow group with homology. 
The constructions in intersection theory are compatible with analogous constructions in 
this context, as detailed in Chapter~19 of~\cite{85k:14004}.\smallskip

\subsection{Segre classes of vector bundles, cones, and subschemes}\label{ss:Segredefs}
\index{Segre class}
Let $V\subseteq \Pbb^n$ be any subvariety. The degree of $V$ may be expressed
as the intersection number of $V$ with a general linear subspace of complementary
dimension:
\begin{equation}\label{eq:degdef}
\deg V=\int_{\Pbb^n} H^{n-\dim V}\cdot V\saf,
\end{equation}
where $H=c_1(\cO(1))$ is the hyperplane class in $\Pbb^n$ and, as recalled above,
$\int_{\Pbb^n} \gamma$ 
denotes the degree of the zero-dimensional component of a rational equivalence class
$\gamma\in A_*(\Pbb^n)$. In fact, by definition $\int_{\Pbb^n} \gamma$ denotes the
integer $m$ such that $\pi_*\gamma = m[p]$, where $\pi\colon \Pbb^n \to p=\Spec k$
is the constant map to a point. With this in mind, we can rewrite~\eqref{eq:degdef} as
\begin{equation}\label{eq:toysegre}
(\deg V)[p] = \pi_* \left(\sum_{i\ge 0} c_1(\cO(1))^i \cap [V]\right)\in A_*(p)\saf\colon
\end{equation}
the only nonzero term on the right is obtained for $i=n-\dim V$, for which it equals
$(H^{n-\dim V}\cdot V)[p]$.

The right-hand side of~\eqref{eq:toysegre} may be viewed as the prototype of a
{\em Segre class,\/} for the trivial projective bundle $\pi\colon \Pbb^n \to p$. 
More generally,
let $X$ be a scheme and let $E$ be a vector bundle over $X$. Denote by
$\pi\colon \Pbo(E)\to X$ the projective bundle of lines in $E$, i.e., let
\begin{equation}\label{eq:prob}
\Pbo(E)=\Proj(\Sym^*_{\cO_X}(\cE^\vee))\saf.
\end{equation}
where $\cE^\vee$ is dual of the sheaf $\cE$ of sections of $E$.
Then for every class $G\in A_*(\Pbo(E))$ we may consider the class
\begin{equation}\label{eq:sigdef}
\Segre_E(G):=\pi_* \left(\sum_{i\ge 0} c_1(\cO_E(1))^i \cap G\right)\in A_*(X)\saf;
\end{equation}
this defines a homomorphism $A_*(\Pbo(E))\to A_*(X)$, which we loosely call
a {\em Segre\index{Segre!operator}
operator.\/}
Even if $G$ is pure-dimensional, $\Segre_E(G)$ will in general consist 
of components of several dimensions. As in the simple motivating example
presented above, however, its effect is to encode intersection-theoretic
information on $G$ in terms of a class in $A_*(X)$.

\begin{example}
Let $X=\Pbb^m$ and let $E=k^{n+1}\times X$ be a free bundle. Then 
$\Pbo(E)\cong \Pbb^m\times \Pbb^n$, and the morphism $\pi\colon
\Pbb^m \times \Pbb^n \to \Pbb^m$ is the projection on the first factor.
If $G\in A_m (\Pbb^m\times \Pbb^n)$ is a class of dimension~$m$
(to fix ideas), then
\[
G = \sum_{i= 0}^m g_i H^{n-i} h^i \cap [\Pbb^{m+n}]\saf,
\]
where $h$, $H$ denote the (pull-backs of the) hyperplane classes from
$\Pbb^m$, $\Pbb^n$, respectively, and $g_i\in \Zbb$ are integers.
Then $H=c_1(\cO_E(1))$, hence
\[
\Segre_E(G)=
\pi_* \left(\sum_{i\ge 0} c_1(\cO_E(1))^i \cap G\right)
=\sum_{i=0}^m g_i h^i\cap [\Pbb^m]
\]
recovers the information of the coefficients $g_i$ determining the class $G$.
\qede\end{example}

Applying $\Segre_E$ to classes $G=\pi^*(\gamma)$ obtained as pull-backs of
classes from the base defines the total 
{\em Segre class\/}\index{Segre class!of a vector bundle} 
of $E$ as an
operator on $A_*(X)$:
\begin{equation}\label{eq:segdef}
s(E)\cap \gamma:= \Segre_E(G)=\pi_*\left(\sum_{i\ge 0} c_1(\cO_E(1))^i 
\cap \pi^*(\gamma)\right)\saf.
\end{equation}
It is a fundamental observation that $s(E)$ is {\em inverse to the 
Chern\index{Chern!class}
class\/} operator, in the sense that
\begin{equation}\label{eq:scinv}
c(E)\cap (s(E)\cap \gamma) = \gamma
\end{equation}
for all $\gamma\in A_*(X)$. 
(Since the intersection product is commutative, it follows that $c(E)$, $s(E)$ are
two-sided inverses to each other.)
Indeed, consider the tautological sequence
\[
\xymatrix{
0 \ar[r] & \cO_E(-1) \ar[r] & \pi^*E \ar[r] & Q \ar[r] & 0\saf.
}
\]
By the Whitney formula,
\[
c(\pi^*E)c(\cO_E(-1))^{-1}\cap \pi^*\gamma =  c(Q)\cap \pi^*\gamma\saf;
\]
by the projection formula,
\[
c(E)\cap \pi_*(c(\cO_E(-1))^{-1}\cap \pi^*\gamma)
=\pi_* (c(Q)\cap \pi^*\gamma)\saf.
\]
Since $Q$ has rank $\rk E-1$, that is, equal to the relative dimension of $\pi$,
\[
\pi_* (c(Q)\cap \pi^*\gamma) = m \gamma
\]
for some integer $m$. Restricting to a fiber shows that $m=1$, and~\eqref{eq:scinv}
follows.

In fact, these considerations may be used to {\em define\/} Chern classes of vector bundles:
Chern classes of line bundles may be defined independently in terms of their relation with
Cartier divisors (as mentioned above); once Chern classes of line bundles are available, 
\eqref{eq:segdef} may 
be used to define Segre classes of vector bundles; and then one may define the Chern
class of a vector bundle $E$ as the inverse of its Segre class, and proceed to prove
all standard properties of Chern classes. This is the approach taken 
in~\cite{85k:14004}, Chapters~2 and~3.

Other choices in~\eqref{eq:sigdef} also lead to interesting notions: whenever
a tautological line bundle $\cO(1)$ is defined, one may define a corresponding Segre
class. For example, we could apply the expression in~\eqref{eq:sigdef} to
\[
\Proj(\Sym^*_{\cO_X}(\cF))
\]
to define a 
Segre class\index{Segre class!of a coherent sheaf}
for any coherent sheaf $\cF$; one instance will appear below, in~\S\ref{ss:CF}.
More generally, the definition may be applied to every {\em projective cone.\/} A 
{\em cone\/}\index{cone} 
over $X$ is a scheme
\[
C=\Spec(\cS^*)=\Spec(\oplus_{k\ge 0} \cS^k)
\]
where $\cS^*$ is a sheaf of graded $\cO_X$ algebras and we assume (as is standard) 
that there is a surjection $\cS^0\twoheadrightarrow \cO_X$, $\cS^1$ is coherent, 
and $\cS^*$ is generated by $\cS^1$ over $\cS^0$. It is useful to enlarge cones by 
a trivial factor: with notation as above, we let
\begin{equation}\label{eq:projcom}
C\oplus \one:=\Spec(\cS^*[t])=\Spec(\oplus_{k\ge 0} (\oplus_{i=0}^k \cS^i t^{k-i}))\saf,
\end{equation}
so that $C$ may be viewed as a dense open subset of its 
`projective completion'\index{projective completion}
$\Pbo(C\oplus \one)=\Proj(\cS^*[t])$; in fact, $C$ is naturally identified with the 
complement of $\Pbo(C)=\Proj(\cS^*)$ in $\Pbo(C\oplus \one)$. 
Cones over $X$ are endowed with a natural projection $\pi$ to $X$ and with 
a tautological line bundle $\cO(1)$, so we may defined the 
{\em Segre class\/}\index{Segre class!of a cone}
of $C$ in the style of~\eqref{eq:sigdef}:
\[
s(C):=\pi_* \left(\sum_{i\ge 0} c_1(\cO_{C\oplus \one}(1))^i \cap [\Pbo(C\oplus \one)]
\right)\in A_*(X)\saf.
\]
If $C$ is a subcone of a vector bundle $E$ (as is typically the case), then
\begin{equation}\label{eq:sCSeg}
s(C) = \Segre_{E\oplus \one}([\Pbo(C\oplus \one)])\saf.
\end{equation}
A case of particular interest is the cone associated with sheaf of $\cO_X$ algebras
\[
\oplus_{k\ge 0} \cI^k/\cI^{k+1}
\]
where $\cI$ is the ideal sheaf defining $X$ as a closed subscheme of a scheme~$Y$.
The corresponding cone $\Spec (\oplus_{k\ge 0} \cI^k/\cI^{k+1})$ is the 
{\em normal\index{normal cone}\index{cone!normal} 
cone\/} of $X$ in $Y$, denoted~$C_XY$. 

\begin{definition}
\index{Segre class!of a subscheme}
Let $X\subseteq Y$ be schemes. The {\em Segre class of $X$ in 
$Y$\/} is the Segre class of the normal cone of $X$ in $Y$:
\begin{equation}\label{eq:SegfSeg}
s(X,Y):= s(C_XY)
=\pi_* \left(\sum_{i\ge 0} c_1(\cO(1))^i \cap [\Pbo(C_XY\oplus \one)]
\right)\saf,
\end{equation}
an element of $A_*(X)$.
\qede\end{definition}

\begin{remark}\label{rem:addfac}
The addition of the trivial factor $\one$ is needed to account for the possibility that
e.g., $\Pbo(C)$ may be {\em empty.\/} For instance, this is the case if $X=Y$,
i.e., $\cI=0$: then $C_XX=\Spec(\cO_X)=X$, $\Pbo(C_XX\oplus \one)=X$,
and $s(X,X)=[X]$.

If $X$ does not contain any irreducible component of $Y$, then
\[
s(X,Y) = \pi_*\left(\sum_{i\ge 0} c_1(\cO(1))^i \cap [\Pbo(C_XY)]\right)\saf.
\]
In general, it is easy to check that $s(C)=s(C\oplus \one)$; in particular, the
notation is compatible with the notation $s(E)$ for vector bundles used above.
\qede\end{remark}

\subsection{Properties}\label{ss:Segpro}
---A closed embedding $X\subseteq Y$ is 
{\em regular,\/}\index{regular embedding} 
of codimension $d$, if 
the ideal $\cI$ of $X$ is locally generated by a regular sequence of length $d$. 
In this case, one can verify that
\[
\oplus_{k\ge 0} \cI^k/\cI^{k+1} \cong \Sym_{\cO_X}^*(\cI/\cI^2)\saf,
\]
so that the normal {\em cone\/} $C_XY$ is a rank-$d$ {\em vector bundle,\/}
denoted $N_XY$. From the definitions reviewed in~\S\ref{ss:Segredefs} it is then
clear that the Segre\index{Segre class!of a regular embedding} 
class of $X$ in $Y$
equals the inverse Chern class of its normal bundle: 
\[
s(X,Y) = s(N_XY)\cap [X] = c(N_XY)^{-1}\cap [X]\saf.
\]

\begin{example}\label{ex:Cardiv}
Let $D\subseteq Y$ be an effective Cartier divisor. Then $N_DY$ is the
line bundle $\cO(D)$, so that
\[
s(D,Y) = c(\cO(D))^{-1}\cap [D] = (1+D)^{-1}\cap [D] \saf.
\]
Abusing notation (writing $D$ for $[D]$), we may write
\[
s(D,Y) = \frac{D}{1+D} = D-D^2 +D^3 -\cdots\saf.
\]
For instance, if $H$ is a hyperplane in $\Pbb^n$, then
\[
s(H,\Pbb^n) = H-H^2+\cdots = [\Pbb^{n-1}]-[\Pbb^{n-2}]+[\Pbb^{n-3}]-\cdots +(-1)^{n-1} [\Pbb^0]
\]
viewed as a class on $H=\Pbb^{n-1}$.

More generally, if $X= D_1\cap \cdots \cap D_r$ is a 
complete intersection\index{Segre class!of a complete intersection}
of $r$ Cartier divisors, then $X\subseteq Y$ is a regular embedding, with
$N_XY=\cO(D_1)\oplus \cdots \oplus \cO(D_r)$, and we may write
\[
s(X,Y) = \frac{[X]}{(1+D_1)\cdots (1+D_r)}\in A_*(X)\saf.
\]
Individual components of this Segre class may be written as symmetric
polynomials in the classes $D_1,\dots, D_r$.
\qede\end{example}

---By definition, the blow-up of $Y$ along $X$ is 
$
B\ell_XY:=\Proj(\oplus_k \cI^k)
$;
the {\em exceptional divisor\/} $E$ of this blow-up is the inverse image of $X$,
so it is defined by the ideal
\begin{equation}\label{eq:ideed}
\cI\oplus \cI^2 \oplus \cI^3\oplus \cdots\quad \subseteq \quad
\cO_Y\oplus \cI\oplus \cI^2\oplus \cdots\saf.
\end{equation}
it follows that
\[
E = \Proj(\oplus_k \cI^k/\cI^{k+1}) = \Pbo(C_XY)\overset \pi\to X\saf.
\]
That is, the 
exceptional\index{exceptional divisor!as a projective normal cone} 
divisor is a concrete realization of the projective
normal cone of $X$ in $Y$. Further, \eqref{eq:ideed} shows that the ideal sheaf 
of $E$ in $B\ell_XY$ is the twisting sheaf
~$\cO(1)$. It follows that $c_1(\cO(1))=-E$,
and therefore
\[
\sum_{i\ge 0} c_1(\cO(1))^i \cap [\Pbo(C_XY)] = E-E^2+E^3 -\cdots \in A_*(E)\saf.
\]
If $X$ does not contain irreducible components of $Y$, it follows
(cf.~Remark~\ref{rem:addfac}) that 
\begin{equation}\label{eq:SegfromE}
s(X,Y) = \pi_*(E-E^2+E^3-\cdots)\saf.
\end{equation}
This observation (and various refinements and alternatives) may be used to construct 
algorithms to compute Segre classes; see~\cite{MR1956868}, \cite{EJP}, \cite{MR3385954}, 
\cite{MR3608229}, \cite{MR4011552} for a sample of approaches and applications. 
The algorithms in the recent paper~\cite{MR4011552} by Corey Harris and Martin Helmer
are implemented in the powerful package {\tt SegreClasses} 
(\cite{SegreClassesSource}) available in the standard implementation of Macaulay2 (\cite{M2}) .

The assumption that $X$ does not contain irreducible components of $Y$
is not a serious restriction: as we have noted that $s(C) = s(C\oplus \one)$
for a cone $C$ (cf.~Remark~\ref{rem:addfac}), it follows that
\[
s(X,Y) = s(X,Y\times \Abb^1)\saf,
\]
where on the right we view $X\cong X\times \{0\}$ as a subscheme of 
$Y\times \Abb^1$. Thus,~\eqref{eq:SegfromE} may be used to compute $s(X,Y)$ 
in general, by employing the exceptional divisor $E$ of the blow-up of 
$Y\times \Abb^1$ along $X$.

---The construction of normal cones is functorial with respect to suitable types of
morphisms. This leads to the following useful result.

\begin{proposition}[{\cite[Proposition~4.2]{85k:14004}}]\label{prop:flapro}
Let $Y$, $Y'$ be pure-dimensional schemes, $X\subseteq Y$ a closed subscheme,
and let $f: Y'\to Y$ be a morphism, and $g:f^{-1}(X)\to X$ the restriction. Then
\begin{itemize}
\item If $f$ is flat, then $s(f^{-1}(X),Y') = g^* s(X,Y)$.
\item If $Y$ and $Y'$ are varieties and $f$ is proper and onto, then 
$g_* s(f^{-1}(X), Y') = (\deg f) s(X,Y)$.
\end{itemize}
\end{proposition}

Here, $f$ realizes the field of rational functions on $Y'$ as an extension of the field
of rational functions on $Y$, and $\deg f$ is the degree of this extension if
$\dim Y=\dim Y'$, and $0$ otherwise. In particular, if $Y'$ and $Y$ are varieties and 
$f: Y'\to Y$ is proper, onto, and {\em birational,\/} then
\[
s(X,Y) = g_* (s(f^{-1}(X), Y')) \saf.
\]
This {\em birational invariance\/}\index{Segre class!birational invariance}
 of Segre classes is especially useful.

\begin{example}
We have verified a particular case of this fact already. Indeed, let $X\subsetneq Y$
be a proper subscheme of a variety, and let 
$f:Y'=B\ell_XY\to Y$ be the blow-up of $Y$ 
along~$X$. Then $f^{-1}(X)=E$ is the exceptional divisor, a Cartier divisor of~$B\ell_XY$,
therefore (Example~\ref{ex:Cardiv})
\[
s(f^{-1}(X), Y') = E - E^2 + E^3 - \cdots\saf.
\]
The birational invariance of Segre classes implies that, letting $g=f|_E: E\to X$,
we must have
\[
s(X,Y) = g_*(E-E^2+E^3-\cdots)\saf;
\]
we have verified this above in~\eqref{eq:SegfromE} 
(where $g$ is denoted $\pi$).
\qede\end{example}

---The Segre class $s(X,Y)$ depends crucially on the scheme structure of $X$; in
general, $s(X,Y) \ne s(X_\text{red},Y)$. On the other hand, different scheme
structures may lead to the same Segre class, and this is occasionally useful.
For instance, assume that the ideals $\cI_{X,Y}$ and $\cI_{X',Y}$ of two subschemes
$X$, $X'$ of $Y$ have the same integral closure. Then $s(X,Y)=s(X',Y)$. Indeed,
we may assume $\cI_{X,Y}$ is a reduction of~$\cI_{X',Y}$; then we have
a finite morphism $B\ell_{X'}Y\to B\ell_XY$ preserving the exceptional divisors
(\cite[Proposition~1.44]{MR2153889}), so the equality follows 
from~\eqref{eq:SegfromE} and the projection formula.
See Example~\ref{ex:gmSm} below for a concrete example of this 
phenomenon.\smallskip

{\em Summary (and shortcut):\/} 
A reader who may not be too comfortable with the algebro-geometric language of
$\Proj$ and cones employed so far may use the following as a characterization
(and hence an alternative definition) of 
Segre\index{Segre class}
classes.

Let $Y$ be a variety.
Every closed embedding $X\subseteq Y$ determines a {\em Segre class\/} 
$s(X,Y)\in A_*(X)$. This class is characterized by the following properties:
\begin{itemize}
\item
If $X\subseteq Y$ is a regular embedding, with normal bundle $N_XY$, then
\[
s(X,Y) = c(N_XY)^{-1}\cap [X]\saf;
\]
\item
if $f: Y'\to Y$ is proper, onto, birational morphism of varieties, and 
$g: f^{-1}(X)\to X$ is the restriction of $f$, then
\[
s(X,Y) = g_* s(f^{-1}(X),Y')\saf.
\]
\end{itemize}
Indeed, by blowing up $Y$ along $X$, the second property reduces the 
computation of Segre class to the case of Cartier divisors, which is covered
by the first property.

Unlike this characterization, the definition given in~\S\ref{ss:Segredefs} does not 
require the ambient scheme $Y$ to be a variety. In our applications, this more 
general situation will not be important.
In any case we note that if $Y$ is pure-dimensional, with irreducible components 
$Y_i$ (taken with their reduced structure) one can in fact show 
(\cite[Lemma~4.2]{85k:14004}) that
\begin{equation}\label{eq:segunio}
s(X,Y) = \sum_i m_i s(X\cap Y_i,Y_i)\saf,
\end{equation}
where $m_i$ is the {\em geometric multiplicity\/} of $Y_i$ in $Y$, and the classes
on the right-hand side are implicitly pushed forward to $X$. Each $s(X\cap Y_i,Y_i)$
is the Segre class of a subscheme of a variety, thus it is determined by the
characterization given above.

\subsection{A little intersection theory}
\index{Fulton-MacPherson intersection theory}\index{intersection theory!Fulton-MacPherson}
Segre classes play a key role in Fulton-MacPherson's intersection theory; indeed,
the very definition of intersection product may be expressed in terms of Segre 
classes. By way of motivation for the formula giving an intersection product, 
consider a vector bundle
\[
p: E \to X
\]
on a scheme $X$. Then it may be verified (\cite[Theorem~3.3(a)]{85k:14004}) that 
the pull-back $p^*: A_*(X) \to A_*(E)$ is an {\em isomorphism.\/} 

\begin{remark}
The fact that $p^*$ is {\em surjective\/} may seem counterintuitive, as it implies that 
a vector bundle over $X$ has no nonzero rational equivalence classes of codimension larger
than the dimension of $X$. See~\cite[\S1.9]{85k:14004}, particularly Proposition~1.9
and Example~1.9.2. This fact can be viewed as a generalization of the observation that 
affine space $\Abb^n$ has no nonzero classes of dimension~$<n$. 
\qede\end{remark}

We may therefore
define a 
`Gysin\index{Gysin homomorphism} 
homomorphism' $\sigma^*: A_*(E)\to A_*(X)$, as the inverse of $p^*$.
That fact that for any subvariety~$Z\subseteq X$, 
\[
\sigma^* ([p^{-1}(Z)])= \sigma^*(p^*[Z]) = [Z]
\]
(and linearity) suggests that $\sigma^*(\alpha)$ should be interpreted as the `intersection
class' of $\alpha$ with the zero-section of $E$.
\begin{center}
\includegraphics[scale=.4]{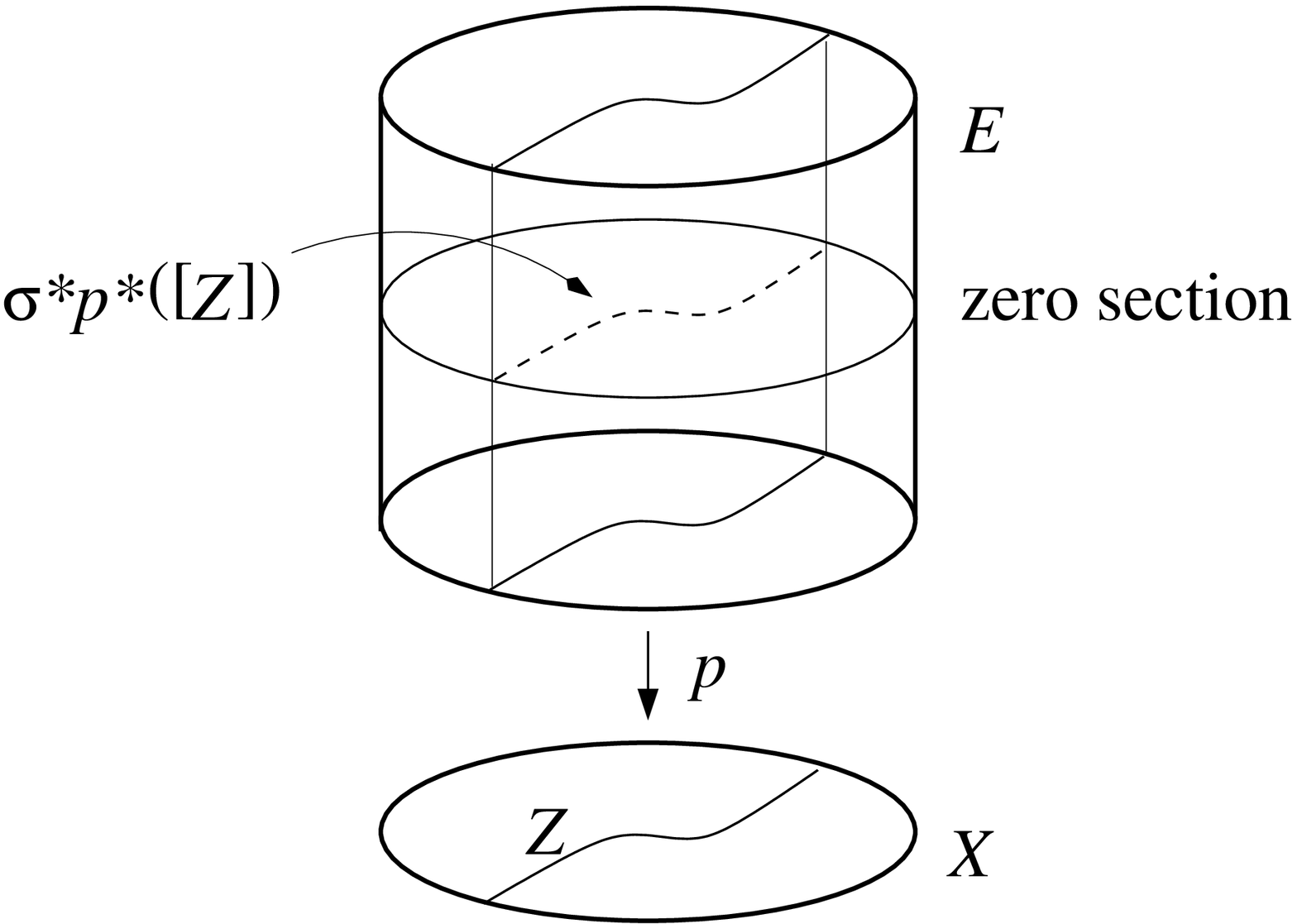}
\end{center}

We can get an explicit expression for $\sigma^*(\alpha)$ in terms of the $\Segre$ 
homomorphism from~\eqref{eq:sigdef}. 
For this, consider
$E$ as a dense open subset of its projective completion $\Pbo(E\oplus \one)$,
and let $\pi: \Pbo(E\oplus \one)\to X$ be the projection. If $\alpha\in A_k(E)$,
then $\alpha=p^*(\sigma^*(\alpha))$ is the restriction to $E$ of 
$\pi^*(\sigma^*(\alpha))$. An expression for $\sigma^*(\alpha)$ may be given
in terms of {\em any\/} class $\overline\alpha\in A_k(\Pbo(E\oplus\one))$ restricting 
to $\alpha$ on $E$.

\begin{lemma}\label{lem:zssc}
Let $\alpha\in A_k(E)$. With notation as above, 
\[
\sigma^*(\alpha) = \left\{c(E)\cap \Segre_{E\oplus\one}(\overline \alpha) \right\}_{k-\rk E}
\]
where $\{\cdots\}_\ell$ is the term of dimension $\ell$ in the class within braces,
and $\overline \alpha$ is any class in $A_k(\Pbo(E\oplus \one))$ restricting
to $\alpha$ on $E$.
\end{lemma}

This statement is equivalent to \cite[Proposition~3.3]{85k:14004}. 
We sketch a verification. As we argued in~\eqref{eq:scinv} (note $c(E\oplus \one)
=c(E)$),
\begin{align*}
\sigma^*(\alpha) &= c(E)\cap (s(E\oplus\one)\cap \sigma^*(\alpha)) \\
&= c(E)\cap \pi_*\left(\sum_{i\ge 0}c_1(\cO(1))^i \cap 
\pi^*(\sigma^*(\alpha)) \right) \\
&= c(E)\cap \Segre_{E\oplus \one}(\pi^*(\sigma^*(\alpha)))\in A_{k-\rk E}(X)\saf.
\end{align*}
Now note that if $\overline\alpha$ is any class in $A_k(\Pbo(E\oplus\one))$
restricting to $\alpha$ on $E$, then
\[
\beta=\overline\alpha-\pi^*(\sigma^*(\alpha))
\]
is supported on the complement $\Pbo(E)$ of $E$ in $\Pbo(E\oplus\one)$.
It follows easily that all components of the class
\[
c(E)\cap \pi_*\left(\sum_{i\ge 0}c_1(\cO(1))^i \cap \beta \right)
\]
have dimension $\ge k-(\rk E-1)$. Thus, the component of dimension $k-\rk E$ of
\[
c(E)\cap \Segre_{E\oplus \one}(\overline\alpha)
\]
agrees with the component of dimension $k-\rk E$ of
\[
c(E)\cap \Segre_{E\oplus \one}(\pi^*(\sigma^*(\alpha)))=\sigma^*(\alpha)
\]
and the statement follows.

A deformation argument reduces to the template of intersecting a class with the
zero-section all intersection situations satisfying
the following requirements. Let $X$ and $V$ be closed subschemes of a scheme $Y$. 
We assume that $V$ is a variety of dimension $m$, and that $X\subseteq V$ is a 
regular embedding of codimension $d$. We have the fiber diagram
\[
\xymatrix{
X\cap V \ar@{^(->}[r] \ar[d]_j & V \ar[d]^i \\
X \ar@{^(->}[r] & Y
}\saf.
\]
The pull-back $i^*\cI_{X,Y}$ of the ideal of $X$ in $Y$ generates the ideal of $X\cap V$
in $V$. This induces a surjection
\[
i^*\Sym^*_{\cO_Y} (\cI_{X,Y}/\cI^2_{X,Y}) =\oplus_{k\ge 0} i^*(\cI_{X,Y}^k/\cI_{X,Y}^{k+1})
\twoheadrightarrow \oplus_{k\ge 0} \cI_{X\cap V,V}^k/\cI_{X\cap V,V}^{k+1}
\]
and consequently realizes $C_{X\cap V}V$ as a closed, $m$-dimensional subscheme
of the pull-back $j^* N_XY$ of the normal bundle of $X$ in $V$. William Fulton and 
Robert MacPherson (cf.~\cite{MR527228}, \cite[Chapter~6]{85k:14004}) define the 
intersection product $X\cdot V\in A_{m-d} (X\cap V)$ to be the intersection of 
$[C_{X\cap V}V]$ with the zero section of the bundle $j^* N_XY$,
defined as above by means of the Gysin morphism:
\[
X\cdot V:= \sigma^*([C_{X\cap V} V])\saf.
\]
As shown in~\cite{85k:14004}, this definition implies all expected properties of an 
intersection product. Applying Lemma~\ref{lem:zssc}, we see that
\begin{align*}
X\cdot V &= \left\{c(j^* N_XY)\cap \Segre_{j^* N_XY\oplus \one}
([\Pbo(C_{X\cap V} V\oplus \one)])\right\}_{m-d} \\
\intertext{since $[\Pbo(C_{X\cap V} V\oplus \one)]$ restricts to $[C_{X\cap V}V]$ on
$j^* N_XY$}
&= \left\{c(j^* N_XY)\cap s(X\cap V,V)\right\}_{m-d}
\end{align*}
(cf.~\eqref{eq:sCSeg} and~\eqref{eq:SegfSeg}). This definition, which we rewrite here
for emphasis:
\begin{equation}\label{eq:FMdef}
X\cdot V := \left\{c(j^* N_XY)\cap s(X\cap V,V)\right\}_{\dim V-\codim_XY}
\end{equation}
is of foundational importance in intersection theory. Note that it assigns an explicit
contribution to $X\cdot V$ to every connected component $Z$ of $X\cap V$:
\begin{equation}\label{eq:FMcont}
\text{contribution of $Z$ to $X\cdot V$:}\quad
\left\{c(N_XY|_Z)\cap s(Z,V)\right\}_{\dim V-\codim_XY}\saf.
\end{equation}
It can be shown that the right-hand side of~\eqref{eq:FMdef} preserves rational equivalence
in the evident sense, so that it defines Gysin homomorphisms $A_k V \to A_{k-d} (X\cap V)$. 
More generally, it defines a homomorphism $A_k Y' \to A_{k-d} (X\times_Y Y')$ for
every morphism $Y'\to Y$. (See~\cite[Chapter~6]{85k:14004}.)

\begin{example}\label{ex:loctcc}
A particular case of~\eqref{eq:FMdef} gives the self-intersection formula of a
regularly embedded subscheme $X$ of $Y$. For this, consider the fiber diagram
\[
\xymatrix{
X \ar[r] \ar[d] & X \ar[d] \\
X \ar[r] & Y
}
\]
and apply~\eqref{eq:FMdef} to obtain
\[
X\cdot X = \{c(N_XY)\cap s(X,X)\}_{\dim X-\codim_XY} = c_d(N_XY)\cap [X]\saf.
\]

For instance, the self-intersection of the zero-section of a vector bundle $E$
on a variety $W$ equals $c_{\rk E}(E)\cap [W]$: indeed, the zero-section is 
regularly embedded, with normal bundle~$E$.

It follows that if $\sigma$ is any section of a vector bundle $E$, then writing
$W$ for the image of the zero-section of $E$, 
\[
\iota_* (W\cdot \sigma(W)) = c_{\rk E}(E)\cap [W]\saf,
\]
where $\iota: Z(\sigma)\to W$ is the embedding of the zero-scheme of $\sigma$. 
Indeed, $\sigma(W)$ is rationally equivalent to the zero-section. 
Again using~\eqref{eq:FMdef}, we can identify the contribution of a union of
connected components $Z$ of $Z(\sigma)$ to $c_{\rk E}(E)\cap [W]$ as
\begin{equation}\label{eq:contrize}
\{ c(E|_Z)\cap s(Z,W) \}_{\dim W-\rk E}\saf,
\end{equation}
`localizing' the top Chern class along the zeros of a section.
(See~\cite[\S14.1]{85k:14004}.)
\qede\end{example}

The requirement that $X$ be {\em regularly\/} embedded in $Y$ is nontrivial. 
It can be bypassed if the ambient scheme $Y$ is a nonsingular variety, say
of dimension $m$. 
Indeed, in this case the diagonal embedding $Y\to Y\times Y$ is regular
with normal bundle $TY$,
and we can interpret the intersection of any two subvarieties $Z,W$ of $Y$ 
as the intersection of the diagonal $\Delta$ with the product $Z\times W$.
The fiber diagram
\[
\xymatrix{
Z\cap W \ar[r] \ar[d]_j & Z\times W \ar[d]^i \\
Y=\Delta \ar[r] & Y\times Y
}
\]
suggests the definition
\begin{equation}\label{eq:nsinde}
[Z]\cdot [W]:= \Delta\cdot (Z\times W)= \{c(j^* TY)\cap s(Z\cap W, Z\times W)\}_{\dim Z+\dim W-m}\saf.
\end{equation}
Note that neither $Z$ nor $W$ need be regularly embedded in $Y$.
This definition passes to rational equivalence and extends by linearity to a 
product $A_*(Y)\times A_*(Y) \to A_*(Y)$
making the Chow group $A_*(Y)$ into a {\em commutative ring.\/} It can be shown
(\cite[Proposition~8.1.1(d)]{85k:14004}) that~\eqref{eq:nsinde} is compatible with 
the previous definition, in the sense that if $Y$ is nonsingular, $Z\subseteq Y$
is a regular embedding, and $W\subseteq Y$ is any subvariety, then 
$[Z]\cdot [W]$ agrees with the definition of $Z\cdot W$ given earlier.

\begin{example}\label{ex:3lines}
Sometimes this intersection product may be used to obtain  information
about a Segre class. For example, consider the three singular quadrics $Q_1,
Q_2,Q_3 \subseteq \Pbb^3$ obtained as unions of two out of three planes in 
general position. For example, $Q_1$ could be defined by the ideal $(x_2 x_3)$,
$Q_2$ by $(x_1 x_3)$, and $Q_3$ by $(x_1 x_2)$. The intersection $J=Q_1\cap Q_2
\cap Q_3$ is the reduced union of three lines through a point. It follows 
(cf.~Example~\ref{ex:nons1}) that 
\begin{equation}\label{eq:3linesc}
\iota_* s(J,\Pbb^3)  = 3[\Pbb^1] + m[\Pbb^0]
\end{equation}
for some integer $m$, where $\iota$ is the embedding of $J$ in $\Pbb^3$.

By B\'ezout's theorem, the intersection product $Q_1\cdot Q_2\cdot Q_3$ 
equals $8$. On the other hand, we may view this intersection product as arising
from the diagram
\[
\xymatrix{
J=Q_1\cap Q_2\cap Q_3 \ar[r] \ar[d]_j & \Pbb^3 \ar[d]^\delta \\
\qquad Q_1\times Q_2\times Q_3 \ar[r] & \Pbb^3\times \Pbb^3\times \Pbb^3
}
\]
where $\delta$ is the diagonal embedding. Using~\eqref{eq:FMdef}, we get 
(omitting an evident pull-back)
\[
\left\{ c(N_{Q_1\times Q_2\times Q_3}\Pbb^3\times\Pbb^3\times\Pbb^3)
\cap s(J,\Pbb^3)\right\}_0=8[\Pbb^0]\saf,
\]
that is, denoting by $H$ the hyperplane class in $\Pbb^3$,
\[
\left\{ (1+2H)^3\cap (3[\Pbb^1] + m[\Pbb^0])\right\} = 8[\Pbb^0]\saf,
\]
which implies $18+m=8$. This determines $m=-10$, and hence
\[
\iota_* s(J,\Pbb^3) = 3[\Pbb^1]-10[\Pbb^0]\saf.
\]
This agrees (as it should) with the result obtained by using the {\tt SegreClasses} package
\cite{SegreClassesSource}: 

{\begin{verbatim}
i1 : load("SegreClasses.m2")

i2 : R=QQ[x0,x1,x2,x3]

i3 : I=ideal(x1*x2,x1*x3,x2*x3)

i4 : segre(I,ideal(0_R))

          3     2
o4 = - 10H  + 3H
          1     1
\end{verbatim}
}

\noindent
(omitting some additional output; and note that the package chooses to call $H_1$
the hyperplane class).

\begin{remark}
We could have chosen the quadrics $Q_1,Q_2,Q_3$ to be the generators of the ideal
of a twisted cubic $C$, and the same argument would show that the push-forward
of $s(C,\Pbb^3)$ also equals $3[\Pbb^1]-10[\Pbb^0]$. In this case, the negative 
coefficient of $[\Pbb^0]$ reflects the fact that the normal bundle to a twisted cubic 
in $\Pbb^3$ is positive. So we could interpret the negative coefficient of $[\Pbb^0]$
in $\iota_* s(J,\Pbb^3)$ as a measure of `positivity' for the normal cone to the scheme
$J$ in $\Pbb^3$.
\qede\end{remark}

The `reverse engineering' technique illustrated above may be used to compute Segre 
classes in broad generality. The approach to the computation of Segre classes in 
projective space developed in~\cite{EJP} is based on an extension of similar methods.
\qede\end{example}

For every class $\overline\alpha\in A_k(\Pbo(E\oplus \one))$, Lemma~\ref{lem:zssc} 
gives an interpretation for the class
\begin{equation}\label{eq:piece}
\left\{c(E)\cap \Segre_{E\oplus\one}(\overline \alpha) \right\}_{k-\rk E}\saf\colon
\end{equation}
this class encodes the class of the restriction of $\overline \alpha$ to $E$.
The other components of the class within braces have an equally compelling
interpretation. If $E$ is a vector bundle of rank $e$ over a scheme $X$, and 
$\pi:\Pbo(E) \to X$ is its projectivization, the Chow group $A_*(\Pbo(E))$ is 
described by a precise structure theorem: for every class $G\in A_k(\Pbo(E))$, 
there exist $e$ unique classes $g_j\in A_j(X)$, $j=k-e+1,\dots, k$ such that
\[
G = \sum_{i=0}^{e-1} c_1(\cO_E(1))^i \cap \pi^*(g_{k-e+1+i})\saf.
\]
(Cf.~\cite[Theorem~3.3(b)]{85k:14004}.)
We call the sum $g_{k-e+1}+\cdots + g_k\in A_*(X)$ the 
{\em shadow\/}\index{shadow} 
of $G$. Note that $G$ may be reconstructed from its shadow and its dimension. 
The following elementary result relates the shadow of $G$ to its Segre class.

\begin{lemma}[{\cite[Lemma~4.2]{MR2097164}}]\label{lem:shad}
With notation as above, the shadow of $G$ is given~by
\[
\sum_{i=0}^e g_{k-e+1+i} = c(E)\cap \Segre_E(G)\saf.
\]
\end{lemma}

With this understood, we see that the class
\[
c(E)\cap \Segre_{E\oplus\one}(\overline \alpha) 
\]
within braces in~\S\ref{eq:piece} is simply the shadow of $\overline\alpha$.
From this point of view, the intersection product $X\cdot V$ is one component
of the shadow of $[C_{X\cap V} V\oplus \one]\in A_*(\Pbo(j^* N_XY\oplus \one))$.
Several classes we will encounter will have natural interpretations as shadows
of classes in suitable projective bundles. 

\subsection{`Residual intersection', and a notation}\label{ss:resanot}
\index{residual intersection}
Let $V$ be a variety, let $X\subseteq V$ be a subscheme, and let $\cL$ be
a line bundle defined on~$X$. We introduce the following notation: if $\alpha$ is
a class in $A_*(X)$, and $\alpha=\oplus_j \alpha^{(j)}$, with 
$\alpha^{(j)}$ of codimension~$j$
{\em in $V$,\/} we let
\begin{equation}\label{eq:defmytens}
\alpha\otimes_V \cL:=\sum_{j\ge 0} s(\cL)^j\cap \alpha^{(j)}
=\sum_{j\ge 0} \frac{\alpha^{(j)}}{c(\cL)^j}\saf.
\end{equation}
This definition was introduced in~\cite{MR96d:14004}. Its notation is motivated by
the following property relating the definition to the ordinary operation of tensor
product: if $E$ is a vector bundle on $X$, or more generally any element in the
K-group of vector bundles on $X$, then for all $\alpha\in A_*(X)$ we have
\begin{equation}\label{eq:notpr1}
(c(E)\cap \alpha) \otimes_M \cL 
= \frac{c(E\otimes \cL)}{c(\cL)^{\rk E}}\cap (\alpha\otimes_M \cL)\saf.
\end{equation}
See~\cite[Proposition~1]{MR96d:14004}; the proof of this fact is elementary.
Equally elementary is the observation that the notation gives an action of
Pic on the Chow group: if $\cL$ and $\cM$ are line bundles on $X$, then
for all $\alpha\in A_*(X)$ we have
\begin{equation}\label{eq:notpr2}
(\alpha\otimes_V \cL)\otimes_V \cM = \alpha\otimes_V (\cL\otimes \cM)\saf.
\end{equation}
See~\cite[Proposition~2]{MR96d:14004}.

The notation introduced above often facilitates computations involving Segre classes. 
One good
example is a formula for the Segre class of a scheme supported on a Cartier divisor, 
along with `residual' (possibly embedded) components. 
Let $D\subseteq V$ be an effective Cartier divisor, and let $R\subseteq V$ a closed 
subscheme. The scheme-theoretic union of $D$ and $R$ is the closed subscheme
$Z\subseteq V$ whose ideal sheaf is the product of the ideal sheaves of $D$ and $R$.
We say that $R$ is the `residual' scheme to $D$ in $Z$. The task is to express
the Segre class of $Z$ in $V$ in terms of the Segre classes of $D$ and of the residual
scheme $R$.

\begin{proposition}\label{prop:ressc}
With notation as above, 
\[
s(Z,V)=s(D,V)+ c(\cO(D))^{-1}\cap (s(R,V)\otimes_V \cO(D))\saf.
\]
\end{proposition}

This is~\cite[Proposition~9.2]{85k:14004}, written using the notation given above;
see~\cite[Proposition~3]{MR96d:14004}. An equivalent alternative formulation is
\begin{equation}\label{eq:ressc}
s(Z,V)=\left([D]+ c(\cO(-D))\cap s(R,V)\right)\otimes_V\cO(D)\saf.
\end{equation}

Along with definition~\eqref{eq:FMdef} and a blow-up construction, 
Proposition~\ref{prop:ressc} may be used to assign a contribution to intersections 
products due to residual schemes, with important applications; 
see~\cite[Chapter~9]{85k:14004}. In this article, the residual formula~\eqref{eq:ressc}
will have applications in the theory of characteristic classes for singular varietes,
cf.~especially~\S\ref{ss:CSMhyp}.

\subsection{Example: hyperplane arrangements}\label{ss:hyparr}
In the rest of this article we will focus on the relation between Segre classes and 
invariants of (possibly) singular spaces. Typically, we will extract information about
a variety $X$ by considering a Segre class of a scheme associated with the 
singular locus of $X$. In many cases we will deal with the case of
hypersurfaces of nonsingular varieties, so we formalize the following definition.

\begin{definition}\label{def:JX}
\index{singularity subscheme of a hypersurface}
Let $X$ be a hypersurface in a nonsingular variety $M$, defined by the vanishing of
a section $s$ of $\cO(X)$. Then the {\em singularity subscheme\/} $JX$ of $X$ is 
defined as the zero-scheme of the section $ds$ of $\Omega^1_M\otimes\cO(X)$
determined by $s$. We will denote by $\iota$ the embedding 
$JX\hookrightarrow X$ or $JX\hookrightarrow M$, as context will dictate.
\qede\end{definition}

Thus, if $z_1,\dots, z_n$ are local parameters for $M$ at a point $p$, and $f$ is a local 
equation of $X$, the ideal of $JX$ at $p$ as a subscheme of $M$ is the jacobian/Tyurina 
ideal
\[
\left(\frac{\partial f}{\partial z_1},\dots, \frac{\partial f}{\partial z_n}, f\right)\saf.
\]
In characteristic~$0$, if $M=\Pbb^n$ and $F(x_0,\dots, x_n)$ is a homogeneous polynomial 
defining a hypersurface $X$, then $JX$ is globally defined by the ideal
\[
\left(\frac{\partial F}{\partial x_0},\dots, \frac{\partial F}{\partial x_n}\right)
\]
(in characteristic~$0$, a homogeneous polynomial belongs to the ideal of its partials).

In order to illustrate the type of information encoded by this subscheme, we present the
case of hyperplane arrangements.
\index{hyperplane arrangement}
Let $\cA$ denote a hyperplane arrangement in 
(complex) projective space $\Pbb^n$,
consisting of $d$
(not necessarily distinct) hyperplanes, and consider the hypersurface $A$ given by 
the union of these hyperplanes. More precisely, let $L_i(x_0,\dots, x_n)$, $i=1,\dots, d$ 
be linear forms whose vanishing defines the hyperplanes; then the hypersurface $A$ is 
defined by the polynomial
\[
F(x_0,\dots, x_n):= \prod_{i=1}^d L_i(x_0,\dots, x_n)\saf.
\]
Max Wakefield and Masahiko Yoshinaga prove (\cite{MR2424913}) that an essential 
arrangement of 
distinct hyperplanes in $\Pbb^n$, $n\ge 2$, may be reconstructed from the 
corresponding 
singularity subscheme. The following result proves that the {\em ranks of the cohomology\/}
of the complement are determined by the Segre class of the singularity subscheme of
the arrangement.

\begin{theorem}\label{thm:Segarr}
For an arrangement $\cA$ of $d$ hyperplanes, define integers $\sigma_i$, $i=0,\dots,n$, 
such that
\[
[\Pbb^n]-\iota_* s(JA,\Pbb^n) = \sum_{i\ge 0} \sigma_i [\Pbb^{n-i}]\saf.
\]
Then
\begin{equation}\label{eq:rkHsc}
\rk H^k(\Pbb^n\smallsetminus A,\Qbb) = \sum_{i=0}^k \binom ki (d-1)^{k-i}\sigma_i
\end{equation}
for $k=0,\dots, n$.
\end{theorem}

This statement is given in~\cite[Theorem~5.1]{MR3047491}; we will sketch a proof
in~\S\ref{ss:CSMhyp} (see Example~\ref{ex:hyparrCSM}).
In fact, in~{\em loc.~cit.,\/} the result is stated for hyperplane arrangements consisting of distinct 
hyperplanes. Remarkably, this hypothesis is not needed: if any of the hyperplanes appear 
with a multiplicity, the effect on the Segre class of the singularity subscheme precisely 
compensates for these multiplicities.

\begin{example}
Consider the arrangement in $\Pbb^3$ consisting of the planes $x_1=0$, $x_2=0$, 
$x_3=0$. The corresponding hypersurface has equation $x_1 x_2 x_3=0$; the
singularity subscheme is defined by the ideal
\[
(x_1 x_2, x_1 x_3, x_2 x_3)\saf. 
\]
We have computed the corresponding Segre class in Example~\ref{ex:3lines}:
\[
\iota_* s(JA,\Pbb^3) = 3[\Pbb^1]-10[\Pbb^0]\saf.
\]
We have $d=3$ and $(\sigma_0,\dots,\sigma_3)=(1,0,-3,10)$, therefore 
Theorem~\ref{thm:Segarr} gives
\[
\rk H^k(\Pbb^3\smallsetminus A,\Qbb) =
\begin{cases}
2^0\cdot 1 = {\bf 1} & k=0\\
2^1\cdot 1+2^0\cdot 0 = {\bf 2} & k=1\\
2^2\cdot 1+2\cdot 2^1\cdot 0+2^0\cdot (-3)={\bf 1} & k=2\\
2^3\cdot 1+3\cdot 2^2\cdot 0+3\cdot 2^1\cdot (-3)+2^0\cdot 10={\bf 0} & k=3
\end{cases}
\]
as it should.

Now assume the same planes appear with multiplicities $2,3,5$ respectively.
The ideal of $A$ is generated by $x_1^2 x_2^3 x_3^5$, therefore $JA$ is
defined by the ideal
\[
\left(x_1 x_2^3 x_3^5, x_1^2 x_2^2 x_3^4, x_1^2 x_2^3 x_3^4\right)
\]
and the package {\tt SegreClasses} evaluates its Segre class as
\[
\iota_* s(JA,\Pbb^3) = 7[\Pbb^2] - 46[\Pbb^1] +270[\Pbb^0]\saf.
\]
In this case $d=10$ and $(\sigma_0,\dots,\sigma_3)=(1,-7,46,-270)$, therefore 
\[
\rk H^k(\Pbb^3\smallsetminus A,\Qbb) =
\begin{cases}
9^0\cdot 1 = {\bf 1} & k=0\\
9^1\cdot 1+9^0\cdot (-7) = {\bf 2} & k=1\\
9^2\cdot 1+2\cdot 9^1\cdot (-7)+9^0\cdot 46={\bf 1} & k=2\\
9^3\cdot 1+3\cdot 9^2\cdot (-7)+3\cdot 9^1\cdot 46+9^0\cdot (-270)={\bf 0} & k=3
\end{cases}
\]
according to Theorem~\ref{thm:Segarr}, with the same result since the support
of the arrangement is the same as in the previous case.
\qede\end{example}

In general, the fact that multiplicities do not affect the right-hand side 
of~\eqref{eq:rkHsc} is a consequence of the residual formula of 
Proposition~\ref{prop:ressc}, as the reader may enjoy verifying.

Note that the Segre class appearing in Theorem~\ref{thm:Segarr} is the Segre
class $s(JA,\Pbb^n)$ of the singularity subscheme {\em in the ambient space\/} 
$\Pbb^n$. The singularity subscheme~$JA$ is also contained in the hypersurface $A$.
It is natural to ask what type of information the Segre class $s(JA,A)$ may encode;
a full answer to this question will be given in~\S\ref{ss:CMa}. Here we point out
that, in the case of reduced arrangements (that is, if the hyperplanes are all
different), this Segre class is in fact {\em determined by the number $d$ of
hyperplanes.\/} 

\begin{proposition}\label{prop:Schyar}
Let $\cA$ be a reduced arrangement of $d$ hyperplanes in $\Pbb^n$, and let 
$\iota: JA\hookrightarrow \Pbb^n$ be the corresponding singularity subscheme. Then
\begin{equation}\label{eq:Segfha}
\iota_* s(JA,A) = d \sum_{i=2}^n (-1)^i (d-1)^{i-1} [\Pbb^{n-i}]\saf.
\end{equation}
\end{proposition}

\begin{proof}
Let $H_1,\dots, H_d$ be the hyperplanes of the arrangement, and let 
$L_k(x_0,\dots, x_n)$ be a generator of the homogeneous ideal of $H_k$.
By~\eqref{eq:segunio}, 
\begin{equation}\label{eq:segunioha}
s(JA,A) = \sum_k s(JA\cap H_k,H_k)\saf.
\end{equation}
The ideal of $JA\cap H_k$ is given by
\[
\left(\sum_{j=1}^d \prod_{\ell\ne j} L_\ell\frac{\partial L_j}{\partial x_i}, L_k\right)_{i=0,\dots, n}
=\left(\prod_{\ell\ne k} L_\ell \frac{\partial L_k}{\partial x_i}, L_k\right)_{i=0, \dots, n}
\]
and this is the ideal
\[
\left(\prod_{\ell\ne k} L_\ell, L_k\right)
\]
since at least one of the derivatives of $L_k$ is nonzero. 

It follows that $JA\cap H_k$ is the subscheme of $H_k$ traced by the union of the
other hyperplanes; that is, it is a Cartier divisor of class $(d-1)H$ in $H_k$,
where $H$ denotes the hyperplane class. Therefore
\[
\iota_* s(JA\cap H_k,H_k) = (d-1) [\Pbb^{n-2}] -(d-1)^2 [\Pbb^{n-3}]+ (d-1)^3 [\Pbb^{n-4}]
+\cdots
\]
and the statement follows from~\eqref{eq:segunioha}.
\end{proof}

Therefore, while the Segre class $s(JA,\Pbb^n)$ detects nontrivial combinatorial
information about the arrangement (as Theorem~\ref{thm:Segarr} shows), the
Segre class $s(JA,A)$ is blind to any information but the degree of the arrangement
(assuming that the arrangement is reduced).

In particular, note that $s(JA,\Pbb^n)$ is {\em not\/} determined by $s(JA,A)$; we
will come back to this point in~\S\ref{ss:CF}, Example~\ref{ex:unf}.

In~\S\ref{sec:Charcla} we will learn that if $X$ is a hypersurface in a nonsingular variety,
then the two classes $s(JX,X)$ and $s(JX,M)$ are closely related to different `characteristic
classes' for $X$, and this will provide a further explanation for the behavior observed in 
this example (see Examples~\ref{ex:hyparr2} and~\ref{ex:hyparrCSM}).


\section{Numerical invariants}\label{sec:numinvs}

\subsection{Multiplicity}
\index{multiplicity}
The most basic numerical invariant of a singularity is its {\em multiplicity.\/}
Let $X$ be a hypersurface of $\Abb^n$, and let $p$ be the origin. Write the
equation $F$ of $X$ as a sum of homogeneous terms:
\[
F(x_1,\dots, x_n)=\sum_{i\ge 0} F_i(x_1,\dots, x_n)
\]
with $F_i(x_1,\dots, x_n)$ homogeneous of degree $i$. By definition,
the multiplicity $m_pX$ of $X$ at $p$ is the smallest $m$ such that 
$F_m(x_1,\dots, x_n)\ne 0$. Thus, $p\in X$ if and only if $m_pX\ge 1$.
The `initial' homogeneous polynomial $F_{m_pX}$ defines the
{\em tangent cone\/}\index{cone!tangent}
to~$X$ at $p$; therefore, $m_pX$ is the degree of the
tangent cone to $X$ at $p$.

There is a natural identification of the tangent cone to $X$ at $p$ defined
in the previous paragraph with the 
{\em normal cone\/}\index{cone!normal} 
$C_pX$ introduced
in~\S\ref{ss:Segredefs}:
\begin{equation}\label{eq:defTC}
C_pX = \Spec (\oplus_{k\ge 0} \mathfrak m^k/\mathfrak m^{k+1})
\end{equation}
where $\mathfrak m$ is the maximal ideal in the local ring of $X$ at $p$. 
We can projectivize this cone, or rather consider the projective completion
$\pi\colon \Pbo(C_pX\oplus \one) \to p$
(cf.~\eqref{eq:projcom}; this accounts for the possibility $X=p$, see
Remark~\ref{rem:addfac}), and 
observe that the degree $m_pX$ of this projective cone satisfies
\[
(m_p X) [p] = \pi_* \left(\sum_{i\ge 0} c_1(\cO(1))^i \cap [\Pbo(C_pX\oplus \one)]
\right)
\]
(cf.~\eqref{eq:toysegre}). In other words, we have verified that the multiplicity of
$X$ at $p$ is precisely the information carried by the Segre class of $p$ in $X$:
\begin{equation}\label{eq:muldef}
s(p,X) = (m_p X) [p]\saf.
\end{equation}

Of course these considerations are not limited to the case in which $X$ is an
affine hypersurface. The tangent cone to a point $p$ of any scheme $X$ is
defined to be the normal cone $C_pX$, that is, the spectrum of the corresponding 
associated graded ring, as in~\eqref{eq:defTC}. A standard computation shows
that if $U=\Abb^N$ is an affine space centered at $p$, and $X\cap U$ is defined
by an ideal $I$, then the ideal defining $\oplus_{k\ge 0} \mathfrak m^k/\mathfrak m^{k+1}$
is generated by the initial forms of the polynomials in $I$; so this is indeed a straightforward
generalization of the situation for hypersurfaces. We can {\em define\/} $m_p X$ to
be the degree of the projective completion $\Pbo(C_pX\oplus \one)$; and 
then~\eqref{eq:muldef} holds in 
this generality. To avoid certain pathologies, it is common to assume that $X$
be pure-dimensional.
For example, this hypothesis implies that the multiplicity of $X$ at $p$ equals the 
sum of the multiplicities of its irreducible components, by~\eqref{eq:segunio}.

The degree of the projective completion of $C_pX$ can also be computed 
by means of the Hilbert function defined for all integers $t>0$ by 
\begin{equation}\label{eq:hilse}
\mathfrak h(t):=\dim_k (\oplus_{i=0}^{t-1} \mathfrak m^i/\mathfrak m^{i+1})\saf\colon
\end{equation}
for $t\gg 0$, $\mathfrak h(t)$ agrees with
a polynomial with leading term $(m_pX) \frac{t^d}{d!}$, where
$d$ is the dimension of $X$. 

More generally, we can consider a subvariety $V$ of a 
(pure-dimensional) scheme~$X$. Samuel (\cite{MR0048103}) defines the multiplicity 
$m_VX$ of the local ring $\cO_{X,V}$ in terms of the leading term of~\eqref{eq:hilse}, 
where now $\mathfrak m$ is taken to be the maximal ideal of $\cO_{X,V}$. 
This amounts to taking the fiberwise degree of the projective completion 
$\Pbo(C_VX\oplus \one)\to V$ of the normal cone $C_VX$,
hence it determines the dominant term of the Segre class:
\begin{equation}\label{eq:multdeffs}
s(V,X)=(m_V X) [V] + \text{lower dimensional terms}\saf.
\end{equation}

\begin{example}
Let $V$ be a proper subvariety of codimension~$d$ of a variety $X$, and let 
$\pi: E\to V$ be the exceptional divisor in the blow-up $B\ell_VX$. Then
\[
\pi_*(E^{d-1})=(-1)^d (m_VX) [V]\saf.
\]
Indeed, this is the dominant term of the Segre class $s(V,X)$ by~\eqref{eq:SegfromE}.
\qede\end{example}

Summarizing, we can take~\eqref{eq:multdeffs} as the {\em definition\/} of multiplicity of
a variety along a subvariety, and this agrees with Samuel's algebraic notion of
multiplicity. The agreement can be extended by the additivity formula~\eqref{eq:segunio}
to arbitrary pure-dimensional schemes $X$. It can also be extended to the case in
which $V$ is an irreducible component of a subscheme $Z$ of $X$, leading to the 
following interpretation of Samuel's multiplicity.

\begin{definition}
\index{multiplicity!of a scheme along a subvariety}
The multiplicity of a pure-dimensional scheme $X$ along a subscheme $Z$ at
an irreducible component $V$ of $Z$ is the coefficient of $[V]$ in $s(Z,X)$.
\end{definition}

\begin{example}\label{ex:nons1}
If $X$ is nonsingular and $Z$ is reduced, then the multiplicity of $X$ along~$Z$ is
$1$ at every component of $Z$. For instance, each line in Example~\ref{ex:3lines}
appears with multiplicity $1$ in the Segre class $s(J,\Pbb^3)$, and this is the
reason why the dominant term in~\eqref{eq:3linesc} equals $3[\Pbb^1]$.
\qede\end{example}

\begin{example}\label{ex:gmSm}
If $Z$ is (locally) a complete intersection in $X$, and its support 
$V$ is irreducible, then
\[
s(Z,X) = m [V] + \text{lower dimensional terms}
\]
where $m$ is the {\em geometric 
multiplicity\/}\index{multiplicity!geometric} 
of $V$ in $Z$,
that is, the length of the local ring~$\cO_{Z,V}$. Indeed, in this case the Segre 
class is the inverse Chern class of the normal bundle (\S\ref{ss:Segpro}):
$s(Z,X) = c(N_ZX)^{-1}\cap [Z]=[Z]+\cdots$, and $[Z]=m [V]$ (\cite[\S1.5]{85k:14004}).
So the multiplicity of $X$ along $Z$ at $V$ equals the geometric multiplicity of $V$ 
in $Z$ for complete intersections.

This is not true in general, even if $X$ is nonsingular. 
For example, let $Z$ be the `triple point' defined by the ideal $(x^2,xy,y^2)$ in the plane.
Then $s(Z,\Abb^2)=4[p]$, where $V=p$ is the origin, while the geometric multiplicity is $3$.
Indeed, let $Z'$ be the scheme defined by $(x^2,y^2)$. Then $s(Z',\Abb^2)=4[p]$,
since $Z'$ is a complete intersection, and $s(Z,\Abb^2)=s(Z',\Abb^2)$ since
$(x^2, xy, y^2)$ is the integral closure of $(x^2,y^2)$ (see~\S\ref{ss:Segpro}).
\qede\end{example}

\begin{example}\label{ex:discri1}
Let $D$ be the 
{\em discriminant\/}\index{discriminant!multiplicity} 
of a line bundle $\cL$ on a nonsingular complete
variety~$M$, i.e., the subset of $\Pbb H^0(M,\cL)$ parametrizing singular sections of 
$\cL$. For $X\in D$, consider the integer
\begin{equation}\label{eq:muldis}
m_XD = \int c(\cL) c(T^\vee M\otimes \cL)\cap s(JX,M)
\end{equation}
where $JX$ is the singularity subscheme of $X$ (Definition~\ref{def:JX})
and $T^\vee M$ is the cotangent bundle of $M$.
Under reasonable hypotheses, if $D$ is a hypersurface, then $m_XD\ne 0$ and in this 
case $m_XD$ is the multiplicity of $D$ at $X$, as the notation suggests. 
(See~\cite{MR94d:14047} for the precise 
statement of a more general result. A different formula not using Segre classes may be
found in~\cite{MR1141011}.) 

To see this, one can realize the discriminant $D$ as the image of the correspondence
\[
\hat D:=\{(p,X)\in M\times \Pbb(M,\cL)\,|\, p\in \Sing(X)\}\saf.
\]
The fiber of $X$ in this correspondence is (isomorphic to) $JX$, and $\hat D$ maps 
birationally to $D$ under mild hypotheses. The birational invariance of Segre classes 
implies then that $s(X,D)$ is the push-forward of $s(JX,\hat D)$, and the latter is computed
by making use of Theorem~\ref{thm:Fulton}, which we will discuss later.

For instance, let $X$ consist of a $d$-fold hyperplane in $M=\Pbb^n$. Then $JX$
is a $(d-1)$-fold hyperplane, and consequently
\[
s(JX,\Pbb^n) = (1+(d-1)H)^{-1}\cap (d-1)[\Pbb^{n-1}]\saf,
\]
where $H$ denotes the hyperplane class. (Example~\ref{ex:Cardiv}.)
View $X$ as a point of the discriminant $D$ of $\cO(X)$ over $\Pbb^n$.
Then according
to~\eqref{eq:muldis} the multiplicity of $D$ at~$X$~is
\begin{align*}
m_DX &=\int (1+dH)\,\frac{(1+(d-1)H)^{n+1}}{1+dH} \cap s(JD,\Pbb^n) \\
&= \int (1+(d-1)H)^n \cap (d-1) [\Pbb^{n-1}] \\
&= n(d-1)^n\saf.
\end{align*}

At the opposite extreme, assume that $X$ has isolated singularities. Then
we will verify that $m_DX$ equals the sum of their 
{\em Milnor\index{Milnor!number}
numbers,\/}\index{multiplicity!of discriminants and Milnor numbers} 
see~\S\ref{ss:Milnor}.
\qede\end{example}

Several more refined notions of `multiplicity' may be defined by means of 
Segre classes; see~\cite{MR1282823} and \cite{MR3912658} for two particularly
well-developed instances.

\subsection{Local Euler obstruction}\label{ss:locEuob}
\index{local Euler obstruction}
The {\em local Euler obstruction\/} $\Eu_X(p)$ of a possibly singular variety $X$ 
(or more generally a reduced pure-dimensional scheme)
at a point $p\in X$ is another numerical invariant, in some ways analogous to the 
multiplicity; indeed, if $X$ is a curve, then $\Eu_X(p)$ {\em equals\/} the multiplicity
$m_pX$. We first summarize the original transcendental definition, due to
MacPherson (\cite[\S3]{MR0361141}). 

We will assume that $X$ is a subvariety of a nonsingular variety $M$. If $X$ has
dimension~$n$, there is a rational map
\[
X \dashrightarrow \Gr_n(TM)|_X
\]
associating with each nonsingular $x\in X$ the tangent space $T_xX\subseteq T_xM$,
viewed as a point in the Grassmann bundle $\Gr_n(TM)$. The closure of the image
of this rational map is the 
{\em Nash\index{Nash blow-up} 
blow-up\/} $\hat X$ of $X$; it comes equipped
with 
\begin{itemize}
\item a proper birational map $\nu: \hat X \to X$; and
\item a rank-$n$ vector bundle $\hat T$, the pull-back of the tautological subbundle on $\Gr_n(TM)$.
\end{itemize}
Over the nonsingular part $X^\circ$ of $X$, $\nu$ is an isomorphism and $\hat T$ 
agrees with the pull-back of~$TX^\circ$. Thus, the Nash blow-up is a modification of $X$ 
that admits a natural vector bundle~$\hat T$ restricting to $TX^\circ$ on the nonsingular part
of $X$. The fiber of $\hat X$ over a point $x\in X$ parametrizes `limits' of tangent
spaces to $X^\circ$ as one approaches $x$. At a point $\hat x\in \nu^{-1}(x)$,
the fiber of $\hat T$ over $\hat x$ is just this limit tangent space.

The Nash blow-up and the tautological bundle $\hat T$ are independent of the chosen
embedding of $X$ in a nonsingular variety.

Let $p\in X$. As we will work in a neighborhood of $p$, we may assume that $X$ is
affine, $M=\Abb^m$, and $p$ is the origin. Over $\Cbb$, MacPherson considers 
the differential form $d ||z||^2$, a section of the real dual bundle $TM^*$.
By construction $\hat T$ is a subbundle of $\nu^*(TM)$; we denote by $r$ the
pull-back of this form to the real dual~$\hat T^*$. 

Next, consider the ball $B_\epsilon$ and the sphere $S_\epsilon$ of radius $\epsilon$
centered at $p$.
For small enough $\epsilon$, $r$ is nonzero over $\nu^{-1}(z)$, $0< ||z||\le \epsilon$
(\cite[Lemma~1]{MR0361141}). By definition, {\em the local Euler obstruction 
$\Eu_X(p)$ is the obstruction to extending $r$ as a nonzero section of 
$\hat T^*$ from $\nu^{-1}(S_\epsilon)$ to $\nu^{-1}(B_\epsilon)$.\/}

For curves, the local Euler obstruction equals the multiplicity. The local Euler obstruction
of a cone over a plane curve of degree~$d$ at the vertex equals $2d-d^2$ 
(\cite[p.~426]{MR0361141}). In particular, note that (unlike the multiplicity) $\Eu_X(p)$ 
may be negative.

The following algebraic alternative to the transcendental definition is due to 
G.~Gonzalez-Sprinberg\index{local Euler obstruction!Gonzalez-Sprinberg-Verdier formula}
and J.-L.~Verdier.

\begin{theorem}[\cite{GS}]\label{thm:GS}
With notation as above,
\begin{equation}\label{eq:GSV}
\Eu_X(p) = \int c(\hat T|_{\nu^{-1}(p)})\cap s(\nu^{-1}(p), \hat X)\saf.
\end{equation}
\end{theorem}

The proof of this equality is quite delicate. The section $r$ may be replaced with 
a section $\sigma_s$ of $\hat T$ obtained by projecting the `radial' section of 
$\nu^* T\Abb^m$ by means of a hermitian form $s$. Viewing $\nu^{-1}(p)$
as a union of components of the zero-scheme of this section, the local Euler 
obstruction is then interpreted as its contribution of $\nu^{-1}(p)$ to the intersection
product of $\sigma_s(\hat X)$ with the zero-section of $\hat T$, that is, the 
localized contribution to the degree of the top Chern class 
$c_{\dim X}(\hat T)\cap [\hat X]$. This 
gives~\eqref{eq:GSV} as we have seen in Example~\ref{ex:loctcc},
particularly~\eqref{eq:contrize}. The main problem with this sketch is that
the section $\sigma_s$ is {\em not\/} algebraic. This is handled in~\cite{GS} by 
applying this argument to a variety dominating $\hat X$ and such that the pull-back 
of $\sigma_s$ {\em is\/} algebraic; \eqref{eq:GSV} then follows by the projection 
formula. 

Theorem~\ref{thm:GS} yields an interpretation of the local Euler obstruction that does
not depend on complex geometry, so may be adopted over arbitrary fields. The use
of the Nash blow-up is not necessary: any proper birational map $\nu: \hat X \to X$
such that $\nu^* \Omega^1_X$ surjects onto a locally free sheaf $\hat\Omega$ of rank
$n=\dim X$ will do, with $\hat T = \hat\Omega^\vee$. (This follows from the birational
invariance of Segre classes; see~\cite[Example~4.2.9]{85k:14004}.)

Claude Sabbah (\cite{MR804052}) recasts the algebraic definition of $\Eu_X(p)$
it in terms of the 
{\em conormal\index{conormal!space} 
space\/} of $X$. Recall that if $W\subsetneq M$ is an
embedding of nonsingular varieties, then the 
{\em conormal\index{conormal!bundle} 
bundle\/} $N^\vee_WM$ of $W$ in $M$ is the kernel of the natural morphism of 
cotangent bundles $T^\vee M|_W \to T^\vee W$:
\[
\xymatrix{
0 \ar[r] & N^\vee_WM \ar[r] & T^\vee M|_W \ar[r] & T^\vee W \ar[r] & 0\saf.
}
\]
The conormal {\em space\/} $N^\vee_XM$ of a possibly singular subvariety 
$X$ of $M$ is the closure of the conormal bundle of its nonsingular part
$X^\circ$:
\[
N^\vee_XM := \overline{N^\vee_{X^\circ} M}\saf.
\]
The projectivization 
$\Pbo(N^\vee_XM)\subseteq \Pbo(T^\vee M|_X)$ is equipped with
\begin{itemize}
\item a morphism $\kappa: \Pbo(N^\vee_XM)\to X$; and,
letting $m=\dim M$,
\item a rank-$(m-1)$ vector bundle $\overline T$, the pull-back of the 
tautological subbundle on $\Pbo(T^\vee M|_X)=\Gr_{m-1}(TM|_X)$.
\end{itemize}

\begin{proposition}[{\cite[Lemma~2]{MR1063344}}]\label{prop:Kennedy}
\begin{equation}\label{eq:Kenthm}
\Eu_X(p) =(-1)^{m-n-1} \int c(\overline T|_{\kappa^{-1}(p)})\cap 
s(\kappa^{-1}(p), \Pbo(N^\vee_XM))\saf.
\end{equation}
\end{proposition}

This result may be established as a corollary of Theorem~\ref{thm:GS}, by means
of a commutative diagram
\[
\xymatrix{
J \ar[r] \ar[d] & \Pbo(N^\vee_XM) \ar[d]^\kappa \\
\hat X \ar[r]_\nu & X
}
\]
where $J$ is the unique component of the fiber product dominating $X$; see
\cite{MR1063344} for details.

\begin{example}\label{ex:discri2}
\index{discriminant!local Euler obstruction}
Again let $D$ be the discriminant of a line bundle $\cL$ on a non singular complete
variety (Example~\ref{ex:discri1}), and let $X\in D$ be a singular section of $\cL$.
Then under mild hypotheses (implying that $D$ is a hypersurface) we have
\begin{equation}\label{eq:Eudis}
\Eu_D(X)=\int c(T^\vee M\otimes \cL)\cap s(JX,M)
\end{equation}
(\cite[Theorem~3]{MR1641591}). Indeed, one can verify that the correspondence 
$\hat D$ mentioned in~Example~\ref{ex:discri1} is the Nash blow-up of $D$, and
$JX$ is isomorphic to the fiber of the point $X\in D$ under $\hat D\to D$. 
Then~\eqref{eq:Eudis} follows from
the Gonzalez-Sprinberg--Verdier formula~\eqref{eq:GSV}, after manipulations
expressing $s(JX,\hat D)$ in terms of $s(JX,M)$ and a computation of the
Chern class of the tautological bundle.

For a concrete instance, consider (as in Example~\ref{ex:discri1}) the case of a 
$d$-fold hyperplane in $\Pbb^n$. According to~\eqref{eq:Eudis}, the local Euler
obstruction of the discriminant of $\cO(X)$ at the corresponding point is
\begin{align*}
\Eu_D(X) &=\int \frac{(1+(d-1)H)^n}{1+dH} \cap (d-1) [\Pbb^{n-1}] \\
&= (d-1)\cdot\frac{(d-1)^n-1}{d}\saf.
\end{align*}

The reader should compare the formulas for the multiplicity of a discriminant at
a singular hypersurface $X$, \eqref{eq:muldis}, and for the local Euler obstruction
at $X$, \eqref{eq:Eudis}. We do not know if the similarity between these formulas
can be extended to more general cases, e.g., discriminants of complete intersections. 
\qede\end{example}

A classical result of L\^e D{\~u}ng Tr{\'a}ng and Bernard Teissier expresses the local 
Euler obstruction as an alternating sum of multiplicities of polar varieties, 
\cite[Corollaire~5.1.2]{MR634426}.

\subsection{Milnor number}\label{ss:Milnor}
\index{Milnor!number}
Segre classes provide a natural algebraic framework to treat Milnor numbers.
Here we work over~$\Cbb$; the formulas we will obtain could be taken as alternative
algebraic definitions extending the notions to arbitrary algebraically closed fields 
of characteristic~$0$.

Let $X$ be a hypersurface in a nonsingular variety $M$, and let $p$ be an isolated
singularity of $X$. Again consider the singularity subscheme $JX$ of~$X$, 
Definition~\ref{def:JX}. In this case $p$ is the support of one component of $JX$,
which we denote $\hat p$. As a subscheme {\em of $M$,\/} the ideal of $\hat p$ at $p$ is
\[
\left(\frac{\partial f}{\partial z_1},\dots, \frac{\partial f}{\partial z_n}, f\right)
\]
where the ideal of $X$ is locally generated by $f$ and $z_1,\dots, z_n$ are local
parameters for $M$ at $p$.

\begin{proposition}
The Milnor number $\mu_X(p)$ of $X$ at $p$ equals the coefficient of $p$ in $s(\hat p,M)$:
\[
s(\hat p, M) = \mu_X(p) [p]\saf.
\]
\end{proposition}

From this observation and~\eqref{eq:muldis}, it follows that if $X$ only has isolated
singularities $p_1,\dots, p_r$, then (under mild hypotheses) the multiplicity of the 
discriminant of $\cO(X)$ at $X$ equals the sum of the Milnor numbers 
$\sum_i \mu_X(p_i)$. For an earlier proof of this fact, at least in the context of dual 
varieties, cf.~\cite{MR856206}.

\begin{proof}
In characteristic $0$, $f$ is integral over the ideal generated by its partials
(see e.g., \cite[Example~1.43]{MR2153889}), therefore $s(\hat p,M)=s(p', M)$, where 
$p'$ is the scheme defined by $(\partial f/\partial z_1,\dots, \partial f/ \partial z_n)$.
Now (Example~\ref{ex:gmSm}) $s(p', M)=m[p]$, where $m$ is the geometric multiplicity 
of $p$ in $p'$. By definition,
\[
m=\dim \cO_{M,p}/(\partial f/\partial z_1,\dots, \partial f/ \partial z_n)\saf,
\]
and this equals the Milnor number $\mu$ (\cite[p.~115]{MR0239612}).
\end{proof}

As an alternative, one can verify that $s(\hat p,M)$ evaluates the effect on the Euler 
characteristic of $X$ due to the presence of the singularity $p$, 
cf.~\cite[Example~14.1.5(b)]{85k:14004}.

Adam Parusi\'nski (\cite{MR949831}) defines a generalization of the 
Milnor\index{Milnor!number!Parusi\'nski's} 
number to hypersurfaces
with arbitrary (compact) singular locus. A section $s$ of $\cO(X)$ defining $X$ determines 
a section $ds$ of $T^\vee M\otimes \cO(X)$ in a neighborhood of $X$, of which $JX$ is the 
zero-scheme (Definition~\ref{def:JX}). By definition, Parusi\'nski's generalized Milnor number 
$\mu(X)$ equals the contribution of the singular locus to the intersection number of the image 
of this section and the zero section of $T^\vee M\otimes \cO(X)$. 

\begin{proposition}[{\cite[Proposition~2.1]{MR97b:14057}}]
With notation as above,
\begin{equation}\label{eq:parnum}
\mu(X) = \int c(T^\vee M\otimes \cO(X))\cap s(JX,M)\saf.
\end{equation}
\end{proposition}

\begin{proof}
Let $U$ be a neighborhood of $JX$ where $ds$ is defined, and consider the fiber
diagram
\[
\xymatrix{
JX \ar[r] \ar[d] & U \ar[d]^{ds} \\
U \ar[r]^-0 & T^\vee M\otimes \cO(X)|_U
}
\]
The normal bundle of the zero-section equals $T^\vee M\otimes \cO(X)|_{JX}$,
and $s(JX,U)=s(JX,M)$ as open embeddings are flat, cf.~Proposition~\ref{prop:flapro}. 
The stated formula follows then from~\eqref{eq:FMdef}.
\end{proof}

In other words, Parusi\'nski's Milnor number equals the localized contribution of~$JX$ to the 
degree of the top Chern class of $T^\vee M\otimes \cO(X)$. 
(But note that in general the section $s$ does not extend to an algebraic section defined on 
the whole of $M$, so this number does {\em not\/} equal the degree of the top Chern class.)

If $M$ is compact and $X_\text{gen}$ is a nonsingular hypersurface linearly equivalent 
to~$X$, then
\begin{equation}\label{eq:parmil}
\mu(X)=(-1)^{\dim X} (\chi(X_\text{gen})-\chi(X))
\end{equation}
where $\chi$ denotes the topological Euler characteristic
(\cite[Corollary~1.7]{MR949831}, and cf.~\cite[Example~14.1.5(b)]{85k:14004}). 
Thus this generalization of the Milnor number can also
be interpreted as the effect on the Euler characteristic of $X$ due to its singular locus.
This observation is at the root of the definition of the `Milnor class', see~\S\ref{ss:Mil}.

Comparing~\eqref{eq:Eudis} and~\eqref{eq:parnum}, we see that, under reasonable
hypotheses, this generalized Milnor number equals the local Euler obstruction of the 
discriminant of~$\cO(X)$ at $X$. 
The class $c(T^\vee M\otimes \cO(X))\cap s(JX,M)$ appearing in these formulas is the 
`$\mu$-class'\index{$\mu$-class} 
studied in~\cite{MR97b:14057}. Even when $JX$ or $M$ are not compact, 
this class carries interesting information. 


\section{Characteristic classes}\label{sec:Charcla}
\index{characteristic class}
The formalism of Segre classes provides a unifying point of view on several 
`characteristic classes' for singular varieties. We refer here to various generalizations
to (possibly) singular varieties of the basic notion of total Chern class of the tangent
bundle of a nonsingular variety:
\[
c(TV)\cap [V]\in A_*(V)\saf.
\]
This is class of evident importance in the nonsingular case. The codimension-$1$ term
$c_1(TV)\cap [V]$ is the canonical class of $V$, up to a sign. For compact complex varieties, 
the degree of the dimension~$0$ term equals the topological Euler characteristic, as a
consequence of the classical Poincar\'e-Hopf theorem. The total Chern class is effective
if the tangent bundle is suitably ample. For nonsingular toric varieties, the class has a 
compelling combinatorial interpretation: it is the sum of the classes of the torus orbit closures,
which are determined by the cones of the corresponding fan. In any case, the sheaf
of differentials is in a sense the `only' canonically determined sheaf on a scheme, and
the total Chern class of the (co)tangent bundle is correspondingly the `only' canonically
defined class in the Chow group of a nonsingular variety. 

It is natural to explore generalizations of this notion to singular varieties, and in this 
section we will review
different alternatives for such an extension, as they relate to Segre classes. 
We remark that there are several other notions of `characteristic class' associated to a 
variety (for example the Todd and L classes), and modern unifications of these notions, 
such as the {\em Hirzebruch\/} and {\em motivic Chern class\/} of 
Brasselet-Sch\"urmann-Yokura (\cite{MR2646988}). While analogues of Segre classes
may be defined in these different contexts, we will limit ourselves to the characteristic
classes defined in the Chow group and having a direct relation with the classical 
notion of Segre classes discussed in~\S\ref{sec:SegreClasses}.
We will also not deal with germane notions such as Johnson's or Yokura's Segre
classes (see~\cite{MR463161}, \cite{MR857438}).

\subsection{Chern-Fulton and Chern-Fulton-Johnson classes}\label{ss:CF}
Let $X$ be a scheme that can be embedded as a closed subscheme of a nonsingular
variety $M$. Here no restrictions on the characteristic of the ground field are needed.
We let
\begin{equation}\label{eq:FFJ}
\begin{aligned}
\cfu(X)&:=c(TM|_X)\cap s(X,M) \\
\cfj(X)&:=c(TM|_X)\cap s(\cN_XM)\saf.
\end{aligned}
\end{equation}
Here $\cN_XM=\cI/\cI^2$, where $\cI$ is the ideal sheaf of $X$ in $M$; the
Segre class $s(\cN_XM)$ is obtained by applying the basic construction of Segre
classes to the cone $\Sym^*_{\cO_X}(\cN_XM)$ (see~\S\ref{ss:Segredefs}).
We call $\cfu(X)$ the 
`Chern-Fulton\index{Chern!-Fulton class} 
class' of $X$, and $\cfj(X)$ the 
`Chern-Fulton-Johnson'\index{Chern!-Fulton-Johnson class}
class of $X$. The following result shows that these classes are canonically determined
by $X$.

\begin{theorem}[{\cite[Example~4.2.6]{85k:14004}, \cite{MR595428}}]\label{thm:Fulton}
The classes $\cfu(X)$, $\cfj(X)$ defined above are independent of the ambient
nonsingular variety $M$.
\end{theorem}

This is proved by relating classes determined by different embeddings by means of
`exact sequences of cones' (cf.~\cite[Examples~4.1.6, 4.1.7]{85k:14004}). 
For instance,
if $X\hookrightarrow M$ and $X\hookrightarrow N$ are distinct embeddings in nonsingular 
varieties, then we have a diagonal embedding $X\subseteq M\times N$, and there is a 
corresponding exact sequence of cones
\[
\xymatrix{
0 \ar[r] & TN|_X \ar[r] & C_X(M\times N) \ar[r] & C_XM \ar[r] & 0
}
\]
implying
\[
s(X,M\times N) = s(TN|_X)\cap s(X,M)\saf.
\]
The independence of $\cfu(X)$ follows. 

Theorem~\ref{thm:Fulton} is useful in the computation of Segre classes; it is employed in
the computations leading to the formulas presented in Example~\ref{ex:discri1} 
and~\ref{ex:discri2}. It has the following consequence.

\begin{corollary}
Let $X=V$ be a nonsingular variety. Then
\[
\cfu(V)=\cfj(V)=c(TV)\cap [V]\saf.
\]
\end{corollary}

\begin{proof}
By Theorem~\ref{thm:Fulton}, we can use $X=M=V$ to compute $\cfu(V)$ and $\cfj(V)$;
and $s(V,V) = s(\cN_VV) = [V]$.
\end{proof}

Therefore these two classes are generalizations of the notion of total Chern class from 
the nonsingular case. They both satisfy formal properties analogous
to the nonsingular case. For example, both classes satisfy expected 
adjunction formulas for sufficiently transversal intersections with smooth subvarieties:
for the Chern-Fulton class, this follows from~\cite[Theorem~3.2]{MR3608229}; and
see~\cite[\S3]{MR595428} for the Chern-Fulton-Johnson class.
On the other hand, some simple relations in the nonsingular case do {\em not\/}
hold for these classes.

\begin{example}\label{ex:unf}
If $W\subseteq X\subseteq M$, with both $X$ and $M$ nonsingular, we may 
compute~$\cfu(W)$ using either embedding, and Theorem~\ref{thm:Fulton} implies that
\begin{equation}\label{eq:compa}
s(W,X) = c(N_XM|_W)\cap s(W,M)\saf.
\end{equation}
For instance, if $X$ is a {\em nonsingular\/} hypersurface, then
\[
s(W,X) = c(\cO(X))\cap s(W,M)\saf.
\]
Such appealingly simple formulas do {\em not\/} hold in general if $X$ is singular, 
even if it is regularly embedded in $M$ (so that $N_XM$ is defined, and the terms in
the formulas make sense). Indeed, \eqref{eq:compa} fails already for $W=$ a singular 
point of a curve $X$ in~$M=\Pbb^2$. 

Without additional hypotheses on $W$ and $X$ guaranteeing that the corresponding sequence of 
cones is exact, the Segre class of $W$ in $X$ is not determined by the class of $W$ in~$M$. 
Sean Keel (\cite{MR1040263}) proved that~\eqref{eq:compa} does hold if $X$ is regularly 
embedded, provided that the embedding $W\subseteq X$ is 
`linear'.\index{linear embedding of schemes}

It would be useful to have precise comparison results relating the difference between
the two sides of~\eqref{eq:compa} to the singularities of $X$. We will encounter below 
(Remark~\ref{rem:sigdif}) one case in which this difference has a clear significance.
\qede\end{example} 

The classes $\cfu(X)$ and $\cfj(X)$ differ in general. The discrepancy is a manifestation
of the difference between the associated graded ring of an ideal $I$ of a commutative
ring $R$, that is, $\oplus_k I^k/I^{k+1}$, and the symmetric algebra of $I/I^2$ over $R/I$. 
The former is in a sense closer to the ring $R$: for example, in the geometric context and 
if $R$~is an integral domain, the Krull dimension of the associated graded ring 
equals the dimension of $R$, while the dimension of the symmetric algebra of a module 
is bounded below by the number of generators of the module (\cite[Corollary~2.8]{MR825143}). 
The difference is analogous to the difference between the tangent {\em cone\/} of a scheme
at a point and the tangent {\em space\/} of the same: the former may be viewed as an
analytic approximation of the scheme at the point, while the latter only records the 
minimal embedding dimension.
Accordingly, $\cfu(X)$ is perhaps a more natural object of study than $\cfj(X)$. The triple 
planar point $X$ defined by the ideal $(x^2,xy,y^2)$ in the affine plane $\Abb^2$
gives a concrete example for which $\cfu(X)\ne \cfj(X)$ (see~\cite[\S2.1]{MR2143071}).

One class of ideals for which the associated graded ring is isomorphic to the symmetric 
algebra is given by ideals generated by regular sequences (cf.~\cite{MR31:1275},
\cite[Theorem~1.3]{MR83b:14017}).
Thus, $\cfu(X)=\cfj(X)$
{\em if $X$ is a local complete intersection.\/} In this case the embedding $X\subseteq M$
is regular, $s(X,M)=c(N_XM)^{-1}\cap [X]$ (\S\ref{ss:Segpro}), and therefore
\begin{equation}\label{eq:fufjvir}
\cfu(X)=\cfj(X) = c(T_\text{vir} X)\cap [X]\saf,
\end{equation}
where $T_\text{vir} X$ is the class $TM|_X-N_XM$ in the Grothendieck group of
vector bundles on $X$ (so $c(T_\text{vir})=c(TM|_X) c(N_XM)^{-1}$).
We can view $T_\text{vir} X$ as a 
`virtual\index{virtual!tangent bundle} 
tangent bundle' for $X$; it is well-defined for
local complete intersections, i.e., independent of 
the ambient nonsingular variety $M$. We note that, more generally,
\[
\cfu(Z) = c(T_\text{vir} X)\cap s(Z,X)
\]
if $Z$ is 
{\em linearly\/}\index{linear embedding of schemes} 
embedded in a local complete intersection $X$, 
cf.~Example~\ref{ex:unf}.

If $X$ is a local complete intersection, we will denote the class~\eqref{eq:fufjvir} by 
$\cvir(X)$, the 
`virtual'\index{Chern!class!virtual}\index{virtual!Chern class} 
Chern class of $X$. For instance, if $X=D$ is a hypersurface 
in a nonsingular variety $M$, 
then
\[
\cvir(D)=\cfu(D)=\cfj(D) = c(TM|_D)\cap \frac{[D]}{1+D}\saf.
\]
This implies the following useful interpretation of the Chern-Fulton / Fulton-Johnson
class of a hypersurface.

\begin{proposition}\label{prop:cvirgen}
Let $i:D\hookrightarrow M$ be a hypersurface in a nonsingular variety $M$, and let 
$i':D_\text{gen}\to M$ be a nonsingular hypersurface such that $[D]=[D_\text{gen}]$. 
Then
\[
i_* \cvir(D) = i'_* (c(TD_\text{gen})\cap [D_\text{gen}])\saf.
\]
\end{proposition}

In particular, over $\Cbb$ and if $M$ is compact, then $\int \cfu(D)=\chi(D_\text{gen})$
is the topological Euler characteristic of a smoothing of $D$, when a smoothing is
available. Barbara Fantechi and Lothar G\"ottsche prove that in fact $\int \cfu(X)$
is constant along lci deformations if $X$ is a local complete 
intersection~\cite[Proposition~4.15]{MR2578301}.

These results may be seen as indicating that a class such as $\cfu(X)$ is {\em not\/}
useful in the study of singularities, precisely because (at least in the lci case) it is blind
to the singularities of $X$. This feature is balanced by the sensitivity of $\cfu(X)$
to the scheme structure of $X$; as we will see below (Proposition~\ref{prop:csmcf}) 
this can be used to encode in a Chern-Fulton class substantial information on the 
singularities of $X$.

Bernd Siebert obtains a formula for the 
`virtual fundamental class'\index{virtual!fundamental class} 
in Gromov-Witten
theory in terms of the Chern-Fulton class, \cite[Theorem~4.6]{MR2115776}. Siebert
also argues that $\cfu(X)$ could be considered as the Segre class of the 
Behrend-Fantechi {\em intrinsic 
normal cone\/}\index{normal cone!intrinsic}
of $X$ (\cite{MR1437495}).

\subsection{The Deligne-Grothendieck conjecture and MacPherson's theorem}\label{ss:DDMp}
A functorial theory of Chern classes arose in work of Alexander Grothendieck and 
Pierre Deligne. Here we will assume that the ground field is algebraically closed, of 
characteristic~$0$.

For an algebraic variety $X$, we denote by $\Ff(X)$ the group of integer-valued
constructible\index{constructible function} 
functions on $X$. These are integer linear combinations of 
indicator functions for constructible subsets of $X$; equivalently, every constructible
function $\varphi\in \Ff(X)$ may be written
\[
\varphi = \sum_W m_W \one_W
\]
where $W$ ranges over 
subvarieties of $X$, $\one_W(p)=1$ or $0$ according
to whether $p\in W$ or $p\not\in W$, and $m_W\in \Zbb$ is nonzero for only finitely
many subvarieties $W$. 

For a proper morphism $f: X\to Y$, we can define a push-forward of constructible 
functions $f_*\colon \Ff(X) \to \Ff(Y)$. By linearity, this is determined by the push-forward
$f_*(\one_W)$ of the indicator function of a 
subvariety $W$ of $X$; we set
\[
f_*(\one_W)(p) := \chi(f^{-1}(p)\cap W)
\]
for $p\in Y$. Here, $\chi$ is the topological Euler characteristic if $k=\Cbb$, and a suitable
generalization for more general algebraically closed fields of characteristic~$0$
(see e.g.,~\cite[\S2.1]{MR3031565}).

With this push-forward, the assignment $X \mapsto \Ff(X)$ is a covariant functor from the 
category of algebraic $k$-varieties, with proper morphisms, to the category of abelian groups
(\cite[Proposition~1]{MR0361141} for the complex case; the argument generalizes to more
general fields).

The Chow group is {\em also\/} a covariant functor between the same categories. 
The following statement, whose conjectural formulation is attributed to Deligne and 
Grothendieck,\index{Deligne-Grothendieck conjecture} 
gives a precise relationship between these two functors. It was proved
by MacPherson~\cite{MR0361141}.

\begin{theorem}\label{thm:MacP}
\index{MacPherson natural transformation}
There exists a natural transformation $c_*: \Ff \Rightarrow A_*$ which, on a nonsingular 
variety $V$, assigns to the constant function $\one_V$ the total Chern class $c(TV)\cap [V]$.
\end{theorem}

MacPherson's statement and proof was for complex varieties, in homology; Fulton 
(\cite[Example~19.1.7]{85k:14004}) places the target in the Chow group. 
Gary Kennedy (\cite{MR1063344}) extended the result to arbitrary algebraically closed 
fields of characteristic~$0$. An alternative argument in this generality (and an alternative
construction of $c_*$) is given in~\cite{MR2282409}.

The natural transformation $c_*$ is easily seen to be unique if it exists, as its 
value is determined by the normalization requirement by resolution of singularities. 
MacPherson provides a different construction, not relying on resolutions; and then proves 
that this construction satisfies the covariance requirement. The ingredients in MacPherson's
construction are the {\em local Euler obstruction\/} $\Eu_X$, reviewed above 
in~\S\ref{ss:locEuob}, and the 
{\em Chern-Mather\index{Chern!-Mather class}
class\/} $\cma(X)$, which will be discussed 
below in~\S\ref{ss:CMa}. MacPherson defines $c_*$ by prescribing that
\[
c_*(\Eu_X) = \cma(X)\saf,
\]
and is able to prove that this assignment determines a natural transformation. 
Since $\Eu_V=\one_V$ if $V$ is nonsingular, this definition satisfies the normalization 
requirement in Theorem~\ref{thm:MacP}. Any choice of a constructible function on
varieties $X$ which takes the constant value $\one_V$ for nonsingular varieties $V$ will
then provide us with a `characteristic class' in the Chow group $A_*(X)$ agreeing with
the total Chern class of the tangent bundle when $X=V$ is a nonsingular variety, as
prospected in the leader to this section. 

\begin{example}
Let $D$ be a hypersurface in a nonsingular complex variety. Assume that $D$ may be realized 
as the central fiber of a flat family over a disk, such that the general fiber $D_\text{gen}$
is nonsingular. 
Verdier~\cite{MR629126}\index{Verdier specialization} 
defines a `specialization' of constructible functions 
from the general fiber to $D$, and proves that this specialization operation is compatible 
with MacPherson's natural transformation and specialization of Chow classes. 
As a consequence, if $\sigma(\one)$ denotes the specialization of the constant 
function~$\one$, we have 
\[
\cvir(D) = c_*(\sigma(\one))
\]
(cf.~Proposition~\ref{prop:cvirgen}). If $D$ is itself nonsingular, then $\sigma(\one)=\one_D$,
and $\cvir(D)=c(TD)\cap [D]$.
We do not know whether the Chern-Fulton or Chern-Fulton-Johnson classes admit a similar 
description for more general varieties.
\qede\end{example}

A formula for the Chern-Mather class due to Sabbah, \cite[Lemma~1.2.1]{MR804052},
leads to a useful alternative description of the image of a constructible function $\varphi$ 
via MacPherson's natural transformation $c_*$. 
In recalling this description, we essentially follow the lucid account given 
in~\cite[\S1]{MR2002g:14005}. 

Let $X$ be a proper subvariety of a nonsingular variety $M$. Every constructible
function $\varphi\in \Ff(X)$ may be written uniquely as a finite linear combination of local 
Euler obstructions of subvarieties of $X$:
\[
\varphi=\sum_{W\subseteq X} n_W \Eu_W
\]
(\cite[Lemma~2]{MR0361141}). Now recall (\ref{ss:locEuob}) that the conormal 
{\em space\/}\index{conormal!space} 
$N^\vee_WM$ of a possibly singular subvariety 
$W$ of $M$ is the closure of the conormal bundle of its nonsingular part
$W^\circ$: $N^\vee_WM := \overline{N^\vee_{W^\circ} M}$.
We associate with the local Euler obstruction of a subvariety $W$ of $M$ the 
cycle of the projectivization of its conormal space, up to a sign recording the parity 
of the dimension of $W$:
\begin{equation}\label{eq:EuT}
\Eu_W \mapsto (-1)^{\dim W} [\Pbo(N^\vee_W M)]\saf.
\end{equation}
By linearity, every constructible function on $X$ is then associated with a cycle in
the projectivized cotangent bundle of the ambient nonsingular variety $M$,
$\Pbo(T^\vee M)$, and in fact of the restriction $\Pbo(T^\vee M|_X)$ to $X$. 

\begin{definition}\label{def:charcyc}
\index{characteristic cycle}
The {\em characteristic cycle\/} of the constructible function 
$\varphi$ is the linear combination
\[
\Ch(\varphi):=\sum_{W\subseteq X} n_W (-1)^{\dim W} [\Pbo(N^\vee_W M)]\saf,
\]
where $\varphi=\sum_{W\subseteq X} n_W \Eu_W$.
\qede\end{definition}

(We have chosen to view $\Ch(\varphi)$ as a cycle in $\Pbo(T^\vee M|_X)$. It is also 
common in the literature to avoid the projectivization, and consider characteristic 
cycles as cycles in $T^\vee M|_X$.)

In keeping with the theme of this paper, we will formulate the alternative description
of $c_*$ stemming from Sabbah's work in terms of a Segre operator
(cf.~\cite[Lemma~4.3]{MR2097164}). For this, it is convenient to adopt the following
notation. If $A=\sum_i a_i$ is a rational equivalence class, where $a_i$ is the component
of dimension $i$, we will let
\[
A_\vee:= \sum_i (-1)^i a_i
\]
be the class obtained by changing the sign of all odd-dimensional components of $A$.
Note that if $E$ is a vector bundle, then
\[
(c(E)\cap A)_\vee = c(E^\vee) \cap A_\vee\saf.
\]
Later on, it will also be convenient to use the notation
\begin{equation}\label{eq:upvee}
A^\vee:= (-1)^{\dim M} A_\vee = \sum_i (-1)^{\dim M-i} a_i \saf,
\end{equation}
where $M$ is the fixed ambient nonsingular variety.

\begin{theorem}\label{thm:Sabthm}
The class $c_*(\varphi)_\vee$ is the shadow of the characteristic cycle $\Ch(\varphi)$.
That~is,
\begin{equation}\label{eq:Sabsha}
c_*(\varphi)=c(TM|_X)\cap \Segre_{T^\vee M|_X}(\Ch(\varphi))_\vee\saf.
\end{equation}
\end{theorem}

Indeed, \eqref{eq:Sabsha} is equivalent to
\begin{equation}\label{eq:sabbah}
c_*(\varphi) = (-1)^{\dim M-1} c(TM|_X)\cap \pi_*(c(\cO(1))^{-1}\cap \Ch(\varphi))\saf,
\end{equation}
where $\pi_*: \Pbo(T^\vee M|_X) \to X$ is the projection; this is~\cite[(12)]{MR2002g:14005},
and the right-hand side is the shadow of $\Ch(\varphi)$ by Lemma~\ref{lem:shad}, up
to changing the sign of every other component. By linearity,~\eqref{eq:sabbah} follows from
\index{Chern!-Mather class}
\[
\cma(W) = c_*(\Eu_W) = (-1)^{\dim M-\dim W-1} c(TM|_W)\cap 
\pi_*(c(\cO(1))^{-1}\cap [\Pbo(N^\vee_WM)])
\]
(where $\pi$ is now the projection to $W$). This formula is (equivalent to)
\cite[Lemma~1.2.1]{MR804052}; also cf.~\cite[Lemma~1]{MR1063344}.

Formula~\eqref{eq:Sabsha} should be compared with the formulas~\eqref{eq:FFJ} 
defining the Chern-Fulton and Chern-Fulton-Johnson classes. The Segre term
\index{Segre!-Schwartz-MacPherson class}
\[
\Segre_{T^\vee M|_X}(\Ch(\varphi))_\vee
\]
plays for MacPherson's natural transformation precisely the same r\^ole played by 
the `ordinary' Segre classes $s(X,M)$, resp., $s(\cN_XM)$ for $\cfu(X)$, resp., $\cfj(X)$.
We will come back to this term below, see~\eqref{eq:SSM}.

The strength of Theorem~\ref{thm:Sabthm} is that (as Sabbah puts it, \cite[p.~162]{MR804052})
{\em `\dots cela montre que la th\'eorie des classes de Chern de~\cite{MR0361141}
se ram\`ene \`a une th\'eorie de Chow sur $T^\vee M$, qui ne fait intervenir que des
classes fondamentales.'\/} Indeed, the Segre term is determined by the characteristic
cycle $\Ch(\varphi)$; this is a linear combination of $(\dim M-1)$-dimensional
fundamental classes of projectivized conormal spaces. These characteristic cycles 
(and the local
Euler obstruction itself) arise naturally in the theory of holonomic D-modules;
this aspect is also treated in~\cite{MR804052}, as well as in work of Masaki 
Kashiwara, Victor Ginzburg, and others (see e.g.,~\cite{MR647684}, \cite{MR833194}).
The characteristic cycles $\Ch(\varphi)$ are projectivizations of 
{\em Lagrangian\/}\index{Lagrangian cycle}
cycles in $T^\vee M$, and various functoriality properties admit a compelling geometric
description in terms of Lagrangian cycles. Thus, the functor $\Ff$ of constructible 
functions may be replaced by
a `Lagrangian functor' associating with $X$ the group of integer linear combinations
of conormal cycles. See~\cite{MR804052} and~\cite{MR1063344} for more information.

From this point of view, defining a characteristic class for arbitrary varieties that
generalizes the total Chern class of the tangent bundle from the nonsingular case
amounts to identifying ways to define Lagrangian cycles which, in the nonsingular
case, associate a variety with the cycle of its conormal bundle (up to sign). 
We will focus on two specific choices:
\begin{itemize}
\item The conormal {\em space\/} of a (possibly singular) variety $X$, corresponding to
the Chern-Mather\index{Chern!-Mather class} 
class $\cma(X)=c_*(\Eu_X)$ (\S\ref{ss:CMa}); and
\item The `characteristic cycle' of $X$, that is, $\Ch(\one_X)$, corresponding to the
`Chern-Schwartz-MacPherson class' of $X$ (\S\ref{ss:CSMhyp}).
\end{itemize}
One of the challenges will be to find (more) explicit expressions for the corresponding
Segre terms $\Segre_{T^\vee M|_X}(\Ch(\Eu_X))_\vee$,
$\Segre_{T^\vee M|_X}(\Ch(\one_X))_\vee$.

\subsection{Chern-Mather classes}\label{ss:CMa}
A key ingredient in MacPherson's construction of the natural transformation
$c_*$ is the {\em Chern-Mather\/}
class of a variety $X$, $\cma(X)$. MacPherson gives a definition of this class 
in~\cite[\S2]{MR0361141}, attributing it to Mather. We note that the definition of an equivalent notion
was given earlier by Wu Wen-Ts\"un (\cite{MR234962}); the equivalence was proved 
later by Zhou Jianyi (\cite{MR1288398})\footnote{We also note that in his review 
of~\cite{MR116023}, Raoul Bott credits Wu with an approach to the algebraic construction
of characteristic classes similar to and preceding Grothendieck's.}.

Let $X$ be a reduced subscheme of a nonsingular variety $M$ of pure dimension $n$, 
and let $X^\circ$ be the nonsingular part of $X$. Recall
(\S\ref{ss:locEuob}) that the {\em Nash blow-up\/} $\hat X$ of $X$ is the
closure of the image of the natural rational map $X \dashrightarrow \Gr_n(TM)|_X$
associating with a nonsingular $x\in X^\circ$ the tangent space $T_xX^\circ\subseteq T_xM$.
The projection from the Grassmannian restricts to a proper birational map $\nu: \hat X 
\to X$, and the tautological subbundle restricts to a rank-$n$ vector bundle $\hat T$
on $\hat X$ extending the pull-back of $TX^\circ$. The local Euler obstruction $\Eu_X(p)$ 
equals
\[
\int c(\hat T|_{\nu^{-1}(p)})\cap s(\nu^{-1}(p), \hat X)
\]
(Theorem~\ref{thm:GS}). 
Following MacPherson, we define the Chern-Mather class of $X$ to be the 
push-forward of the Chern class of $\hat T$.

\begin{definition}\label{def:cMa}
\index{Chern!-Mather class}
With notation as above, the {\em Chern-Mather class\/} of $X$ is
\begin{equation}\label{eq:cMa}
\cma(X) = \nu_*\left( c(\hat T)\cap [\hat X]\right)\saf,
\end{equation}
an element of $A_*(X)$.
\qede\end{definition}

As we discussed in~\S\ref{ss:DDMp}, we have the following alternative expression for the
Chern-Mather class:
\begin{equation}\label{eq:cmaex}
\cma(X)=c(TM|_X)\cap (-1)^{\dim X} \Segre_{T^\vee M|_X}([\Pbo(N^\vee_XM)])_\vee\saf.
\end{equation}
This is~\eqref{eq:Sabsha} for $\varphi=\Eu_X$, as
$\Ch(\Eu_X)=(-1)^{\dim X} [\Pbo(N^\vee_X M)]$ (see~\eqref{eq:EuT}).
The equivalence of~\eqref{eq:cMa} and~\eqref{eq:cmaex}, due to Sabbah,
may be verified by the same techniques proving Proposition~\ref{prop:Kennedy};
cf.~\cite[Lemma~1]{MR1063344}.

If $X$ is a hypersurface, the Segre term can be expressed directly in terms of 
ordinary Segre classes. Recall that for a rational equivalence class~$A$ of
a subvariety of a fixed ambient variety $M$, we let
\[
A^\vee:=(-1)^{\dim M} A_\vee\saf.
\]

\begin{theorem}\label{thm:MaSe}
Let $X$ be a hypersurface of a nonsingular variety $M$. Then
\[
(-1)^{\dim X}\Segre_{T^\vee M|_X}([\Pbo(N^\vee_XM)])_\vee = 
\left([X]+\iota_* s(JX,X)^\vee\right)\otimes_M \cO(X)\saf.
\] 
\end{theorem}

(Cf.~\cite[Lemma~I.2]{MR2001i:14009} and~\cite[Proposition~2.2]{MR2020555}.) 
In this statement, $JX$ is the singularity subscheme of $X$ (Definition~\ref{def:JX}),
$\iota\colon JX\to X$ is the embedding, and we use the notation $\otimes_M$ 
recalled in~\S\ref{ss:resanot}.

\begin{proof}
The left-hand side of the stated formula equals
\begin{equation}\label{eq:SegCma}
\pi_*\left(c(\cO_{T^\vee M|_X}(1))^{-1}\cap [\Pbo(N^\vee_X M)]\right)\saf,
\end{equation}
where $\pi\colon \Pbo(T^\vee M|_X)\to X$ is the projection. As $X$ is a hypersurface, 
the projectivized conormal space $\Pbo(N^\vee_X M)$ may be realized as the 
closure of the image of the rational map
\[
X \dashrightarrow \Pbo((T^\vee M\otimes \cO(X))|_X)\cong \Pbo(T^\vee M|_X)
\]
associating with every $x\in X^\circ$ the hyperplane $T_x X^\circ$ of $T_xM$,
viewed as a point of~$\Pbo(N^\vee_x M)$. This
closure is isomorphic to the blow-up of $X$ along the base scheme of the rational
map, and the base scheme is $JX$ by definition. For another point of 
view on this observation, recall that the Nash blow-up of a hypersurface $X$ is 
isomorphic to its blow-up along $JX$, see e.g., \cite[Remark~2]{MR409462}; 
for hypersurfaces, the conormal space is isomorphic to the Nash blow-up. 
Now we have
\begin{equation}\label{eq:O1}
\cO_{T^\vee M|_X}(1) \cong \cO_{(T^\vee M\otimes \cO(X))|_X}(1)\otimes \pi^*\cO(X)\saf;
\end{equation}
therefore~\eqref{eq:SegCma} may be rewritten
\[
\pi_*\left(c(\cO_{(T^\vee M\otimes\cO(X))|_X}(1)\otimes\pi^*\cO(X))^{-1}\cap
 [B\ell_{JX}X]\right)\saf,
\]
and as $\cO_{(T^\vee M\otimes\cO(X))|_X}(1)$ restricts to $c(\cO(-E))$ on the
blow-up, where $E$ denotes the exceptional divisor, this class equals
\[
\pi_*\left(c(\pi^*\cO(X)\otimes \cO(-E))^{-1}\cap [B\ell_{JX}X]\right)
=\pi_*\left(c(\cO(-E))^{-1}\cap [B\ell_{JX}X]\right)\otimes_M\cO(X)
\saf,
\]
where now $\pi$ denotes the projection from the blow-up and we made use
of~\eqref{eq:notpr2} and of the projection formula. 
This last expression equals the
right-hand side of the formula given in the statement, by~\eqref{eq:SegfromE}.
\end{proof}

\begin{remark}\label{rem:pertuMa}
To parse the expression obtained in Theorem~\ref{thm:MaSe}, note that as $X$~is
a hypersurface, 
\begin{align*}
\cfu(X) &=c(TM|_X)\cap s(X,M) = c(TM|_X)\cap (1+X)^{-1}\cap [X] \\
&= c(TM|_X)\cap \left([X] \otimes_M \cO(X)\right)\saf,
\end{align*}
while \eqref{eq:cmaex} and Theorem~\ref{thm:MaSe} imply that
\begin{equation}\label{eq:cmaexp}
\cma(X) =c(TM|_X)\cap \left(
\left([X]+\iota_* s(JX,X)^\vee\right)
\otimes_M \cO(X)\right)\saf.
\end{equation}

What this is saying is that the Chern-Mather class of a hypersurface 
$X$ is the Chern-Fulton class of a virtual object whose fundamental class is
\begin{equation}\label{eq:fdperMa}
[X]+\iota_* s(JX,X)^\vee\saf,
\end{equation}
a perturbation of the fundamental class of $X$, determined by the Segre class of
the singularity subscheme of $X$ {\em in $X$.\/}

Enforcing the analogy with the Chern-Fulton class, we could formally write
\[
\cma(X) = c(TM|_X)\cap s_\text{Ma}(X,M)\saf,
\]
for a `Segre-Mather\index{Segre!-Mather class} 
class' $s_\text{Ma}(X,M)$. Thus $s_\text{Ma}(X,M)=s(X,M)$ if 
both $X$ and~$M$ are nonsingular, and Theorem~\ref{thm:MaSe} gives
an explicit expression for the Segre-Mather class if $X$ is a hypersurface
in a nonsingular variety $M$.

We do not know a similarly explicit expression of the Segre-Mather class for more
general varieties~$X$.
\qede\end{remark}

As there are implementations for the computation of Segre classes (see~\S\ref{ss:Segpro}),
Chern-Mather classes of hypersurfaces in e.g., nonsingular projective varieties can also 
be computed by making use of~\eqref{eq:cmaexp}. See~\cite{MR3608229} for concrete 
examples.

\begin{remark}\label{rem:Pieneremark}
We note that the relation between the Segre class of the singularity subscheme of a 
hypersurface~$X$ {\em of projective space\/} and the Chern-Mather class of~$X$ may 
also be obtained as a corollary of results of Piene: the 
{\em polar\index{polar!cycles and classes} 
classes\/} of a 
hypersurface $X\subseteq \Pbb^n$ can be computed in terms of the Segre class $s(JX,X)$ 
(\cite[Theorem~2.3]{Pienepolarclasses}) and the Chern-Mather class may be expressed
in terms of polar classes (\cite[Th\'eor\`eme~3]{MR1074588}).

In fact, for projective varieties, the fact that~\eqref{eq:cmaexp} only holds for hypersurfaces 
is tempered by another result of Piene, \cite[Corollaire, p.~20]{MR1074588}, showing that 
Chern-Mather classes are preserved by general projections. Thus, the computation of the
degrees of the components of the Chern-Mather class of a projective variety may be
reduced to the hypersurface case. 
\qede\end{remark}

In any case, it would be interesting to extend Theorem~\ref{thm:MaSe} beyond the 
hypersurface case. It is conceivable that even if $X$ is not a hypersurface, the Segre 
term in~\eqref{eq:cmaex} may admit an equally transparent expression in terms of the 
Segre class of a scheme naturally associated with the singularities of $X$. 

\begin{example}\label{ex:hyparr2}
It follows easily from the definition that if $X=X_1\cup X_2$ is the union of two closed
reduced
subschemes of the same pure dimension
and with no irreducible components in common,
then $\cma(X)=\cma(X_1)+\cma(X_2)$ (where the classes on the right-hand side are
viewed as classes in $A_*(X)$). Indeed, the Nash blow-up of $X$ is simply the union of 
the Nash blow-ups of $X_1$ and $X_2$. (This also implies that $\Eu_X=\Eu_{X_1}
+\Eu_{X_2}$; cf.~\cite[p.~426]{MR0361141}.)

For a hyperplane 
arrangement\index{hyperplane arrangement}
$\cA$ consisting of $d$ distinct hyperplanes $H_i$ in 
$\Pbb^n$, this implies that the Chern-Mather class of the corresponding hypersurface 
$A$ is
\[
\cma(A)=\sum_i \cma(H_i) = \sum_i c(TH_i)\cap [H_i]
\]
and therefore if $i:A\to \Pbb^n$ is the embedding, and $H$ denotes the hyperplane class,
\[
i_*\cma(A)=d\cdot (1+H)^n\cap [\Pbb^{n-1}]\saf.
\]
Let's verify that this is compatible with the formula~\eqref{eq:Segfha} for $s(JA,A)$ 
obtained in~\S\ref{ss:hyparr}:
\begin{align*}
\iota_* s(JA,A) & = d \sum_{i=2}^n (-1)^i (d-1)^{i-1} [\Pbb^{n-i}] \\
&= d \left((d-1) [\Pbb^{n-2}]\otimes_A \cO((d-1) H)\right) \\
&= d (d-1) (1+(d-1)H)\cap\left([\Pbb^{n-2}]\otimes_{\Pbb^n} \cO((d-1) H)\right)\saf,
\end{align*}
where we have used the notation in~\S\ref{ss:resanot}. It follows that the
`perturbed fundamental class' \eqref{eq:fdperMa} is
\begin{multline*}
[A]+\iota_* s(JA,A)^\vee \\
= d\left([\Pbb^{n-1}]+(d-1)(1-(d-1)H)\cap \left([\Pbb^{n-2}]\otimes_{\Pbb^n} \cO(-(d-1) H)\right)\right)
\end{multline*}
and therefore the push-forward of the Segre-Mather class to $\Pbb^n$ equals
(using~\eqref{eq:notpr1} and~\eqref{eq:notpr2})
\begin{align*}
&i_*\left(([A]+\iota_*  s(JA,A)^\vee)\otimes_{\Pbb^n} \cO(A)\right) \\
&= d\left([\Pbb^{n-1}]+(d-1) (1-(d-1)H)\cap \left([\Pbb^{n-2}]\otimes_{\Pbb^n} \cO(-(d-1) H)
\right)\right)\otimes_{\Pbb^n} \cO(dH) \\
&=d\left([\Pbb^{n-1}]\otimes \cO(dH) + (d-1)\frac{1+H}{1+dH}\cap ([\Pbb^{n-2}]
\otimes_{\Pbb^n} \cO(H))\right) \\
&=d\left( \frac 1{1+dH} + (d-1)\frac {1+H}{1+dH} \frac{H}{(1+H)^2}\right) \cap [\Pbb^{n-1}] \\
&= d\cdot (1+H)^{-1}\cap [\Pbb^{n-1}]\saf.
\end{align*}
In conclusion,
\begin{align*}
i_* \left(c(T\Pbb^n|_A)\cap s_\text{Ma}(A,\Pbb^n)\right)
&=d\cdot (1+H)^{n+1} (1+H)^{-1} \cap [\Pbb^{n-1}] \\
&=d\cdot (1+H)^n\cap [\Pbb^{n-1}]
\end{align*}
as it should.\smallskip

More generally, let $X = \cup_{i=1}^r X_i$ be the union of $r$ distinct irreducible 
(possibly singular) hypersurfaces in a nonsingular variety $M$. Denote by $X_{- i}$ 
the union of the hypersurfaces other than~$X_i$. Then, omitting evident push-forwards:
\begin{equation}\label{eq:JXXunion}
s(JX,X) = \sum_i X_i\cdot s(X_{-i},M) + s(JX_i, X_i)\otimes_M \cO(X_{-i})\saf.
\end{equation}
This may be proved by the same technique used in the proof of Proposition~\ref{prop:Schyar},
using~\eqref{eq:ressc} (that is, `residual intersection') to account for the singularity
subschemes of the individual components $X_i$. The reader should have no difficulty
verifying that~\eqref{eq:JXXunion} is compatible with the fact that $\cma(X) = 
\sum_i \cma(X_i)$.
\qede\end{example}

\subsection{Chern-Schwartz-MacPherson classes of hypersurfaces}\label{ss:CSMhyp}
\index{Chern!-Schwartz-MacPherson (CSM) class}
Again all our schemes will be subschemes of a fixed nonsingular variety $M$, and we
work in characteristic~$0$. We do not need to assume that schemes are reduced or 
pure-dimensional.

Choosing the function $\one_X$ for every scheme is trivially the simplest way to 
define a constructible function generalizing $\one_V$ for nonsingular varieties $V$.
Thus, this defines a characteristic class trivially generalizing $c(TV)\cap [V]$.

\begin{definition}\label{def:CSMclass}
Let $X$ be a scheme as above. The {\em Chern-Schwartz-MacPherson (CSM) class\/} of $X$ 
is the class
\[
\csm(X):= c_*(\one_X)\in A_*(X)\saf.
\]
\end{definition}
More generally (abusing language) we let
\[
\csm(W):= c_*(\one_W)\in A_*(X)
\]
for any constructible subset $W$ of $X$; the context will determine the Chow group where 
$\csm(W)$ is meant to be taken. Note that as $\one_W$ only depends on the support $W_\text{red}$
of $W$, we have $\csm(W)=\csm(W_\text{red})$. (Cf.~Remark~\ref{rem:nonred} below for relevant
comments on this point.)

Definition~\ref{def:CSMclass} is given in~\cite{MR0361141} (for compact complex varieties,
and in homology); MacPherson attributes it to Deligne. 
In~\cite{MR83h:32011}, Brasselet and Marie-H\'el\`ene Schwartz proved that
the class agrees via Alexander duality with the classes defined earlier by Schwartz in
relative cohomology (\cite{Schwartz1,Schwartz2}). 

One way to compute $\csm(X)$ is to express the constant function $\one_X$ as a
linear combination of local Euler obstructions:
\[
\one_X = \sum_i m_i \Eu_{W_i}
\]
for a choice of finitely many subvarieties $W_i$ of $X$. It then follows that
\[
\csm(X) = c_*(\one_X) = \sum_i m_i c_*(\Eu_{W_i}) = \sum_i m_i \cma(W_i)\saf.
\]
The proof in~\cite{MR83h:32011} relies on establishing precise relations between
indices of radial vector fields and local Euler obstructions, and hence between
Schwartz's classes and Chern-Mather classes. It is also possible to prove that
the classes defined by Schwartz satisfy enough of the functoriality properties of
the classes defined by MacPherson to guarantee that they must agree
(\cite{MR2415339}); this approach avoids the use of local Euler obstructions or
Chern-Mather classes.

One motivation in Schwartz's work was to obtain a class generalizing the classical
Poincar\'e-Hopf theorem to singular varieties. This incorporated in MacPherson's
approach as an implication of the naturality of $c_*$. Assume that $X$ is complete, 
so that the constant map 
$\kappa\colon X\to pt=\Spec k$ is proper. The fact that $c_*$ is a natural transformation
implies that the following diagram is commutative:
\[
\xymatrix@C=10pt{
\Ff(X) \ar[d]_{\kappa_*} \ar[rr]^-{c_*} & & A_*(X) \ar[d]^{\kappa_*} \\
\Ff(pt) \ar@{=}[r] & \Zbb \ar@{=}[r] & A_*(pt)
}
\]
If $W\subseteq X$ is any constructible subset, the commutativity of the
diagram
\[
\xymatrix{
\one_W \ar@{|->}[d]_{\kappa_*} \ar@{|->}[r] & \csm(W) \ar@{|->}[d]^{\kappa_*} \\
\chi(W) \ar@{|->}[r] & \int \csm(W)
}
\]
amounts to the equality
\begin{equation}\label{eq:PH}
\int \csm(W) = \chi(W)\saf\colon
\end{equation}
the degree of the CSM class
equals the topological Euler characteristic (or a suitable generalization over
fields other than~$\Cbb$). This can be viewed as an extension to
possibly singular, possibly noncompact varieties of the Poincar\'e-Hopf
theorem, holding over arbitrary algebraically closed fields of characteristic~$0$.

By Theorem~\ref{thm:Sabthm}, we have
\[
\csm(X)=c(TM|_X)\cap \Segre_{T^\vee M|_X}(\Ch(\one_X))_\vee\saf.
\]
Just as in~\S\ref{ss:CMa}, it is natural to ask for a more explicit and computable
expression for the Segre term
\begin{equation}\label{eq:SSM}
s_\text{SM}(X,M):=\Segre_{T^\vee M|_X}(\Ch(\one_X))_\vee\saf,
\end{equation}
which we view as a `Segre-Schwartz-MacPherson' class.
\index{Segre!-Schwartz-MacPherson class}
In~\S\ref{ss:CSMgen} we will argue that this task can be reduced to the case of
hypersurfaces; in this section we focus on the hypersurface case. The following 
result is the CSM version of Theorem~\ref{thm:MaSe}.

\begin{theorem}\label{thm:CSMSe}
Let $X$ be a hypersurface in a nonsingular variety $M$. Then
\[
\Segre_{T^\vee M|_X}(\Ch(\one_X))_\vee =
\left([X]+\iota_* (c(\cO(X)\cap s(JX,M)))^\vee\right)\otimes_M \cO(X)\saf.
\]
\end{theorem}

This is~\cite[Lemma~I.3]{MR2001i:14009}; cf.~\cite[Proposition~2.2]{MR2020555}.
It can be interpreted as stating that if $X$ is a hypersurface of a nonsingular variety
$M$, then the Chern-Schwartz-MacPherson class of $X$ is the Chern-Fulton class
of an object whose `fundamental class' is
\begin{equation}\label{eq:fdperCSM}
[X]+ \iota_* (c(\cO(X))\cap s(JX,M))^\vee\saf.
\end{equation}

\begin{remark}\label{rem:sigdif}
The reader should compare \eqref{eq:fdperMa} and~\eqref{eq:fdperCSM}, that is,
the perturbations of the fundamental class corresponding to the different
characteristic classes we have encountered, in the case of hypersurfaces:
\begin{alignat*}{2}
&\text{Chern-Fulton:} &\qquad &[X] \\
&\text{Chern-Mather:} &\qquad &[X] + \iota_* s(JX,X)^\vee \\
&\text{Chern-Schwartz-MacPherson:} &\qquad 
&[X] +\iota_* (c(\cO(X))\cap s(JX,M))^\vee\saf.
\end{alignat*}
The difference between the Chern-Mather class and the Chern-Schwartz-MacPherson
class is captured precisely by the difference between
\[
s(JX,X)\quad\text{and}\quad c(\cO(X))\cap s(JX,M)\saf.
\]
As we have observed in Example~\ref{ex:unf}, it is natural to compare the classes
$s(W,X)$ and $c(\cO(X))\cap s(W,M)$, for any subscheme $W$ of a hypersurface $X$. 
The case $W=JX$ provides one instance in which the difference has a transparent 
and interesting interpretation.
\qede\end{remark}

Different proofs are known for Theorem~\ref{thm:CSMSe}. One approach consists of
proving that the class
\begin{equation}\label{eq:CSMpur}
c(TM|_X)\cap \left(
\left([X]+ \iota_* (c(\cO(X)\cap s(JX,M)))^\vee\right)\otimes_M \cO(X)
\right)
\end{equation}
has the same behavior under blow-ups along nonsingular subvarieties of $JX$
as the class $\csm(X)$. By resolution of singularities, we may then
reduce to the case in which $X$ is a divisor with normal crossings and 
nonsingular components, and in this case one can verify that~\eqref{eq:CSMpur}
does equal $\csm(X)$. It follows that~\eqref{eq:CSMpur} must equal
$\csm(X)$ in general. This approach is carried out in~\cite{MR2001i:14009}.

A perhaps more insightful argument consists of a concrete realization of the characteristic
cycle $\Ch(\one_X)$. For this, view the singularity subscheme $JX$ of $X$ as a
subscheme {\em of~$M$.\/} Consider the blow-up
\[
\pi\colon B\ell_{JX}M \to M
\]
of $M$ along $JX$. This is naturally embedded as a subscheme of $\Pbo(\cP^1_M(\cO(X)))$, 
the projectivization of the {\em bundle of principal parts\/} of $\cO(X)$. The inverse image
$\cX:=\pi^{-1}(X)$ is contained in $\Pbo((T^\vee M\otimes \cO(X))|_X) \subseteq
\Pbo((\cP^1_M(\cO(X)))|_X)$, and contains the exceptional divisor $\cE=\pi^{-1}(JX)$ of the 
blow-up. Thus, we have $(\dim M-1)$-dimensional cycles $[\cX]$, $[\cE]$ of
$\Pbo((T^\vee M\otimes \cO(X))|_X)\cong \Pbo(T^\vee M|_X)$.

The reader may find it helpful to recall that $B\ell_{JX}X$ may also be realized as
a subscheme of $\Pbo(T^\vee M|_X)$; the proof of Theorem~\ref{thm:MaSe} relies on the 
identification of this subscheme with the projectivized conormal space 
$\Pbo(N^\vee_XM)$, whose cycle is $(-1)^{\dim X} \Ch(\Eu_X)$. 

\begin{lemma}\label{lem:charcy}
The characteristic cycle $\Ch(\one_X)$ equals $(-1)^{\dim X} ([\cX]-[\cE])$.
\index{characteristic cycle!of a hypersurface}
\end{lemma}

This statement implies Theorem~\ref{thm:CSMSe}, by an argument similar to the
proof of Theorem~\ref{thm:MaSe}. (Cf.~e.g., \cite[Theorem~I.3]{MR2001i:14009}.)
Lemma~\ref{lem:charcy} is proved in~\cite[Corollary~2.4]{MR2002g:14005}, along 
with a thorough discussion of characteristic cycles of other constructible functions
naturally associated with a hypersurface.
An earlier description of the characteristic {\em variety\/} of a hypersurface is given 
in~\cite[Theorem~3.3]{MR84g:32018}.\smallskip

Theorem~\ref{thm:CSMSe} is equivalent to the following formula, which we state
as a separate result for ease of reference.

\begin{theorem}\label{thm:CSMpur2}
Let $X$ be a hypersurface in a nonsingular variety $M$. Then
\[
\csm(X)=c(TM|_X)\cap \left(
\left([X]+ \iota_* (c(\cO(X))\cap s(JX,M))^\vee\right)\otimes_M \cO(X)
\right)\saf.
\]
\end{theorem}

\begin{remark}
Xiping Zhang has generalized this result to the equivariant setting, \cite{Zhangsthesis}.
\qede\end{remark}

We have already observed that the formula in Theorem~\ref{thm:CSMpur2} may be viewed 
as expressing $\csm(X)$ as the Chern-Fulton class of a virtual object with a similar
behavior to a hypersurface, but with a fundamental class modified to include
lower dimensional terms. There is a
perhaps more compelling intepretation of this object as a Chern-Fulton class, 
obtained by applying residual intersection as follows.

Recall that the Chern-Fulton class of a scheme is not just determined by its support;
the specific scheme structure affects the class. For a hypersurface $X$ of a nonsingular
variety $M$, we consider the Chern-Fulton class of the scheme obtained by `thickening'
$X$ along its singularity subscheme $JX$: that is, for $k\ge 0$ we consider the scheme 
$X^{(k)}$ whose ideal sheaf in $M$ is
\[
\cI_{X,M} \cdot (\cI_{JX,M})^k\saf.
\]
Thus $X=X^{(0)}$. The residual formula in Proposition~\ref{prop:ressc} yields an
expression for the Segre class of this scheme in $M$. According to~\eqref{eq:ressc},
\[
s(X^{(k)},M)=\big([X]+ c(\cO(-X))\cap s((JX)_k,M)\big)\otimes_M\cO(X)\saf,
\]
where $(JX)_k$ is the subscheme of $M$ defined by the ideal~$(\cI_{JX,M})^k$.
(Thus $(JX)_0=\emptyset$, $(JX)_1=JX$, etc.)
Accordingly, we have an expression for the Chern-Fulton class of~$X^{(k)}$:
\index{Chern!-Fulton class!of a thickening.}
\begin{equation}\label{eq:CFthik}
\cfu(X^{(k)}) = c(TM)\cap \big([X]+ c(\cO(-X))\cap s((JX)_k,M)\big)\otimes_M\cO(X)\saf.
\end{equation}
This expression makes sense for all nonnegative integers $k$, and by definition
\[
\cvir(X) = \cfu(X^{(0)})\saf.
\]
Now we observe that $s((JX)_k,M)$ is determined by $s(JX,M)$ for all $k\ge 0$:
indeed, the component of dimension $\ell$ of this class is given by
\[
s((JX )_k,M)_\ell = k^{\dim M-\ell} s(JX ,M)_\ell\saf.
\]
Indeed, if $\cE$ denotes the exceptional divisor of the blow-up $B\ell_{JX}M$,
then the inverse image of $(JX)_k$ in the blow-up is $k\cE$, so the assertion
follows from~\eqref{eq:SegfromE}.

As a consequence,~\eqref{eq:CFthik} expresses $\cfu(X^{(k)})$ as a {\em polynomial\/}
in $k$, and as such this class can be given a meaning for every integer $k$.

\begin{proposition}\label{prop:csmcf}
Let $X$ be a hypersurface in a nonsingular variety $M$. With notation as above,
\[
\csm(X) = \cfu(X^{(-1)})\saf.
\]
\end{proposition}

This is of course just a reformulation of Theorem~\eqref{thm:CSMpur2}. It identifies the
Chern-Schwartz-MacPherson class of $X$ with the Chern-Fulton class of a virtual 
(fractional?) scheme obtained from $X$ by simply `removing' its singular locus.
The Segre-Schwartz-MacPherson\index{Segre!-Schwartz-MacPherson class!of a hypersurface} 
class of a hypersurface $X$ in a nonsingular variety $M$ is simply
\[
s_\text{SM}(X,M) = s(X^{(-1)},M)\saf.
\]

\begin{remark}\label{rem:nonred}
There is one case in which the virtual scheme $X^{(-1)}$ is {\em not\/} virtual.
Let $V$ be a {\em nonsingular\/} hypersurface of a nonsingular variety $M$,
and let $X$ be the non-reduced hypersurface whose ideal is the $r$-th power
of the ideal of $V$:
\[
\cI_{X,M} = {\cI_{V,M}}^r\saf.
\]
Then (as the characteristic is $0$), $JX$ has ideal ${\cI_{V,M}}^{r-1}$, hence
$X^{(k)}$ has ideal ${\cI_{V,M}}^{r+k(r-1)}$ for $k\ge 0$. This ideal makes sense
for $k=-1$, giving $X^{(-1)}=V$. Therefore
\[
\csm(X)=\csm(V)=c(TV)\cap [V]=\cfu(V)=\cfu(X^{(-1)})
\]
as it should. 

Using Proposition~\ref{prop:ressc}, it is not hard to verify that if $X$ is a possibly 
non-reduced effective Cartier divisor in a nonsingular variety $M$, then
\[
\cfu(X^{(-1)}) = \cfu({X_\text{red}}^{(-1)})\saf,
\]
even if the support $X_\text{red}$ is singular. This is compatible with our
definition of the Chern-Schwartz-MacPherson class of a possibly non-reduced 
scheme $X$, which guarantees that it only depends on the support of $X$.
\qede\end{remark}

\begin{example}
The {\em polar degree\/}\index{polar!degree} 
of a hypersurface $X$ of $\Pbb^n$ defined by a homogeneous
polynomial $F$ is the degree of the gradient map $\Pbb^n \to \Pbb^n$,
\begin{equation}\label{eq:gradmap}
p\mapsto \left(\frac{\partial F}{\partial x_0}: \cdots : \frac{\partial F}{\partial x_n}\right)\saf.
\end{equation}
A hypersurface is 
`homaloidal'\index{hypersurface!homaloidal} 
if this map is birational, that is, if its polar degree is~$1$.
Igor Dolgachev (\cite[p.~199]{MR1786486}) 
conjectured\index{Dolgachev's conjecture}
that a hypersurface $X$ 
is homaloidal if and only if $X_\text{red}$ is homaloidal. 

Now, the graph of the map~\eqref{eq:gradmap} is isomorphic to the blow-up of the 
zero-scheme of the partials, that is, to $B\ell_{JX}\Pbb^n$. Therefore, it is straightforward
to express the polar degree in terms of the degrees of the components of the Segre
class of $JX$ in $\Pbb^n$, and therefore in terms of the degrees of the components of
the Chern-Schwartz-MacPherson class of $X$. The result of this computation is the
following (see~\cite[\S3.1]{MR3031565} for more details).

\begin{proposition}
Let $X\subseteq \Pbb^n$ be a hypersurface. Denote by $\deg c_i(X)$ the degree of
the dimension-$i$ component of $\csm(X)$. Then the polar 
degree of $X$ equals
\[
(-1)^n-\sum_{i=0}^n (-1)^{n-i} \deg c_i(X)\saf.
\]
\end{proposition}

Since $\csm(X)=\csm(X_\text{red})$, it follows that the polar degree of $X$ equals
the polar degree of $X_\text{red}$, verifying Dolgachev's conjecture.
(To our knowledge, the first proof of the conjecture appeared 
in~\cite[Corollary~2]{MR2018927}, over $\Cbb$. The argument sketched above
holds over any algebraically closed field of characteristic~$0$.)
\qede\end{example}

\begin{example}\label{ex:hyparrCSM}
We return once more to a hyperplane 
arrangement\index{hyperplane arrangement} 
$\cA$ in $\Pbb^n$ and its
corresponding hypersurface $A$. We will sketch a proof of Theorem~\ref{thm:Segarr},
which relies on the computation of $\csm(A)$. We will assume that $A$ is reduced, 
but as we just observed, $\csm(A_\text{red})=\cfu(A_\text{red}^{(-1)})=\cfu(A^{(-1)})$, 
and it follows that the result holds without changes for non-reduced arrangements
(as stated in~\S\ref{ss:hyparr}).

The arrangement $\cA$ corresponds to a 
central arrangement $\widehat\cA$ in $\Abb^{n+1}$. 
We let $\chi_{\widehat\cA}(t)$ be
the characteristic\index{characteristic polynomial of an arrangement} 
polynomial of $\widehat\cA$; see 
e.g., \cite[Definition~2.5.2]{MR1217488}. (For arrangements corresponding
to graphs, this is essentially the same as the 
{\em chromatic\index{chromatic polynomial} 
polynomial\/}
of the graph.)
We define $\chi_\cA(t)$ to be the quotient $\chi_{\widehat\cA}(t)/(t-1)$; this is
also a polynomial in $\Zbb[t]$, of degree~$n$. 

Now consider the Chern-Schwartz-MacPherson class of the {\em complement\/} of $A$:
\[
\csm(\Pbb^n\smallsetminus A) = c_*(\one_{\Pbb^n} - \one_A)\in A_*(\Pbb^n)\saf.
\]
As an element of $A_*(\Pbb^n)$, this class may be written as an integer linear
combination of the classes $[\Pbb^i]$ for $i=0,\dots, n$.

\begin{theorem}[{\cite[Theorem~1.2]{MR3047491}}]\label{thm:ha1}
The class $\csm(\Pbb^n\smallsetminus A)$ equals
the class obtained by replacing $t^i$ with $[\Pbb^i]$ in $\chi_\cA(t+1)$.
\end{theorem}

This may be proved by a combinatorial argument, using `M\"obius inversion'. 
Alternately, one may use the deletion-contraction property of the characteristic
polynomial and the fact that Chern-Schwartz-MacPherson classes satisfy an 
inclusion-exclusion\index{inclusion-exclusion property of CSM classes} 
property: if $W_1$ and $W_2$ are locally closed subsets 
of a variety $V$, then
\begin{equation}\label{eq:inclexcl}
\begin{aligned}
\csm(W_1\cap W_2) &= c_*(\one_{W_1\cap W_2}) 
= c_*(\one_{W_1}+\one_{W_2}-\one_{W_1\cup W_2}) \\
&=\csm(W_1)+\csm(W_2)-\csm(W_1\cup W_2)
\end{aligned}
\end{equation}
in $A_*(V)$. June Huh extracts an expression of the characteristic polynomial
from these considerations, see~\cite[Remark~26]{MR2904577}.

The information carried by the characteristic polynomial of an arrangement
is equivalent to the information in its 
{\em Poincar\'e\index{Poincar\'e polynomial of an arrangement} 
polynomial\/}
\[ 
\pi_\cA(t):=(-t)^n \cdot \chi_A(-t^{-1})\saf.
\]
As the reader can verify, Theorem~\ref{thm:ha1} is equivalent to the following 
formula:
\[
i_*\csm(A)=c(T\Pbb^n)\cap \left( 1 - \frac 1{1+H} \pi_\cA\left(\frac{-H}{1+H}\right)\right) \cap [\Pbb^n]
\]
where $H$ is the hyperplane section and $i:A\to \Pbb^n$ is the inclusion. Therefore
\[
i_* s_\text{SM}(A,\Pbb^n) = \left( 1 - \frac 1{1+H} \pi_\cA\left(\frac{-H}{1+H}\right)\right) \cap [\Pbb^n]
\]
Using Theorem~\ref{thm:CSMSe} and simple manipulations, it follows that
\[
\pi_\cA\left(\frac{-H}{1+H}\right)\cap [\Pbb^n] = \frac{1+H}{1+dH}\left(1 - 
\iota_* s(JA,\Pbb^n)^\vee \otimes_{\Pbb^n} \cO(dH)\right)\cap [\Pbb^n] 
\]
where $d$ is the number of hyperplanes in the arrangement. Letting
\[
\iota_* s(JA,\Pbb^n) = \sum_i s_i [\Pbb^i] = \sum_i s_i H^{n-i}\cap [\Pbb^n]
\]
we get an equality of power series in $h$ modulo $h^{n+1}$:
\[
\pi_\cA\left(\frac{-h}{1+h}\right) \equiv \frac{1+h}{1+dh}
\left(1 - \sum_{i=0}^n \frac{ s_i\cdot (-h)^{n-i}}{(1+dh)^{n-i}}\right)
\quad \mod h^{n+1}
\]
or equivalently
\begin{equation}\label{eq:interm2}
\pi_\cA(t) \equiv \frac 1{1-(d-1)t} \left(1-\sum_{i=0}^n s_i \cdot 
\left( \frac t{1-(d-1)t}\right)^{n-i}\right)\quad \mod t^{n+1}\saf.
\end{equation}
By a classical result of Peter Orlik and Louis Solomon 
(\cite[Theorem~5.93]{MR1217488}),
\[
\pi_\cA(t) = \sum_{i=0}^n \rk H^k(\Pbb^n\smallsetminus A,\Qbb) t^i\saf.
\]
Reading off the coefficients of $t^i$, $i=0,\dots,n$ in~\eqref{eq:interm2} 
yields Theorem~\ref{thm:Segarr}.
\qede\end{example}

\subsection{Chern-Schwartz-MacPherson classes, general case}\label{ss:CSMgen}
Formulas in the style of Theorem~\ref{thm:CSMpur2} are useful: they have been applied 
to concrete computations of Chern-Schwartz-MacPherson classes, and they are amenable to
implementation in systems such as Macaulay2 since Segre classes are (\S\ref{ss:Segpro}).
One may expect that there should be a straightforward generalization of 
Theorem~\ref{thm:CSMSe} to higher codimension subschemes $X$ of a nonsingular
variety, based on the Segre class of a subscheme defined by a suitable Fitting ideal, 
generalizing the singularity subscheme~$JX$. One could also expect a generalization
of the interpretation of the Chern-Schwartz-MacPherson class as the Chern-Fulton class
of a suitable virtual scheme, along the lines of Proposition~\ref{prop:csmcf}.
With the exception of results for certain types of complete intersections 
(\cite{MR3165594}, \cite{FullwoodWang}), we do not know of explicit results
along these lines.

However, a formula for the Chern-Schwartz-MacPherson class of an arbitrary
subscheme of a nonsingular variety in terms of the Segre class of a related scheme
{\em can\/} be given. This is the most direct extension of~Theorem~\ref{thm:CSMSe}
currently available, and it will be presented below (Theorem~\ref{thm:CSMemb}). 
Before discussing this result, we note that, for computational purposes, the case
of arbitrary subschemes can already be treated by organizing a potentially large
number of applications of the hypersurface case.

\begin{proposition}\label{prop:CSMincexc}
Let $X$ be a subscheme of a nonsingular variety $M$, and assume that $X$ is 
the intersection of $r$ hypersurfaces $X_1,\dots, X_r$. Then
\[
\csm(X)= \sum_{s=1}^r (-1)^{s-1} \sum_{i_1<\cdots<i_s} 
\csm(X_{i_1}\cup\cdots\cup X_{i_s})\saf.
\]
\end{proposition}

This is clear from 
inclusion-exclusion,\index{inclusion-exclusion property of CSM classes} 
which holds for CSM
classes since it holds for constructible functions, cf.~\eqref{eq:inclexcl}.
Since the classes appearing in the right-hand side are all CSM classes of
hypersurfaces, they can be computed by applying Theorem~\ref{thm:CSMpur2}.
This approach yields an algorithm for computing Chern-Schwartz-MacPherson 
classes of subschemes of 
$\Pbb^n$ and more general varieties, based on the computation of Segre classes 
(cf.~\cite{MR1956868}, \cite{MR3484270}, \cite{MR3385954}, 
\cite{MR3608229}, \cite{MR3608162}). The current Macaulay2 distribution 
includes the package~{\tt CharacteristicClasses} \cite{CharacteristicClassesSource}, 
by Helmer and Christine Jost, which implements this observation.

\begin{example}
Let $X$ be the scheme defined by the ideal $(xz^2 - y^2w , xw^2 - yz^2 , x^2w - y^3 , 
z^4 -yw^3)$ in $\Pbb^3$. 
The following Macaulay2 commands compute the push-forward to $\Pbb^3$
of its Chern-Schwartz-MacPherson class.

{\begin{verbatim}
i1 : load("CharacteristicClasses.m2")

i2 : R=QQ[x,y,z,w]

i3 : I=ideal(x*z^2-y^2*w, x*w^2-y*z^2, x^2*w-y^3, z^4 -y*w^3)

i4 : CSM I

       3     2
o4 = 2h  + 6h
       1     1
\end{verbatim}
}
\noindent
(The package uses $h_1$ to denote the hyperplane class.)
This shows that the locus is a sextic curve with topological Euler characteristic equal 
to~$2$. (It is in fact an irreducible rational sextic with one singular point.)
As the ideal has four generators, the computation requires $15$ separate applications 
of Theorem~\ref{thm:CSMpur2}, including one for a degree-$13$ hypersurface.
\qede\end{example}

One intriguing aspect of this approach via inclusion-exclusion is that the same subscheme
may be represented as an intersection of hypersurfaces in many different ways; and
extra features such as embedded or multiple components do not affect the result, 
since the Chern-Schwartz-MacPherson class only depends on the support of the scheme.
Massive cancellations involving the Segre classes underlying such computations must 
be at work. To our knowledge, more direct proofs of such cancellations are not available.

One obvious drawback of Proposition~\ref{prop:CSMincexc} is the large number 
of computations needed to apply it: $2^r-1$ distinct Segre class computations
for the intersection of $r$ hypersurfaces. As we will see next,
the same input---for example, a set of generators for the homogeneous ideal of a
projective scheme $X\subseteq \Pbb^n$---may be used to obtain an expression
that is a more direct generalization of Theorem~\ref{thm:CSMpur2}, in the sense that 
it gives an expression for $\csm(X)$ in terms of a single Segre class of a related
scheme. The price to pay is an increase in dimension, and the fact that (at this time) 
the result only yields the push-forward of $\csm(X)$ to the Chow group $A_*(M)$
of the ambient nonsingular variety.

Let $X$ be a subscheme of a nonsingular variety $M$. We may assume that 
$X$~is the zero scheme of a section of a vector bundle $E$ on $M$; in fact, we may 
choose $E=\Spec(\Sym \cE)$, where $\cE$ is any locally free sheaf surjecting onto 
the ideal sheaf $\cI_{X,M}$ of $X$ in $M$. Note that we can assume that the rank of
$E$ is as high as we please: for example, we can replace $\cE$ with 
$\cE\oplus \cO_M^{\oplus a}$ for any $a\ge 0$. The surjection 
$\cE \twoheadrightarrow \cI_{X,M}$ induces a morphism $\phi:\cE|_X \to \Omega_M|_X$
whose cokernel is the sheaf of differentials~$\Omega_X$. We view this as a
morphism of vector bundles over~$X$, $\phi: E^\vee|_X \to T^\vee M|_X$.
The kernel of $\phi$ determines a subscheme $J_E(X)$ of the projectivization
$\Pbo(E^\vee|_X)\overset\pi\to X$.
\begin{definition}\label{def:JEX}
With notation as above, we will denote by $J_E(X)$ the subscheme of $\Pbo(E^\vee|_X)$
defined by the vanishing of the composition of the pull-back of $\phi$ with the 
tautological inclusion $\cO_{E^\vee}(-1)\to \pi^* E^\vee|_X$.
\qede\end{definition}

It may be helpful to describe $J_E(X)$ in analytic coordinates $(x_1,\dots, x_n)$
for $M$, over an open set $U$ where $\Omega_M$ and $E$ are trivial. 
If $X$ is defined by $f_0(\ux)=\cdots=f_r(\ux)=0$ (so $\rk E= r+1$), 
$\phi:\cE|_X \to \Omega_U|_X$ has matrix
\[
\begin{pmatrix}
\frac{\partial f_0}{\partial x_1} & \cdots & \frac{\partial f_r}{\partial x_1} \\
\vdots & \ddots & \vdots \\
\frac{\partial f_0}{\partial x_n} & \cdots & \frac{\partial f_r}{\partial x_n} \\
\end{pmatrix}
\]
and $J_E(X)$ is defined by the ideal
\[
(y_0 df_0+\cdots+y_r df_r) = \left(\sum_{i=0}^r y_i \frac{\partial f_i}{\partial x_j}\right)_{j=1,\dots,n}
\]
in $\Pbo(E^\vee|_{X\cap U}) = \Pbb^r \times (X\cap U)$.
In other words, $J_E(X)$ records linear relations between the differentials
of the generators of the ideal of $X$. These may be due to relations between
the generators themselves (note that nothing prevents us from choosing e.g., 
$f_0=f_1$), or to singularities of $X$.

\begin{example}
If $X$ is a hypersurface, defined by the vanishing of a section $s$ of $E=\cO(X)$,
then $J_E(X)$ is the subscheme of $X\cong \Pbb^0\times X$ defined by 
the vanishing of~$ds$. 
That is, $J_E(X)=JX$ in this case: in this sense, the definition of $J_E(X)$ 
generalizes the notion of `singularity subscheme' of a hypersurface.
\qede\end{example}

We view $J_E(X)$ as a subscheme of the nonsingular variety $\Pbo(E^\vee)$,
and denote by $\iota:J_E(X)\hookrightarrow \Pbo(E^\vee)$ the inclusion and
$\pi\colon \Pbo(E^\vee)\twoheadrightarrow M$ the projection. 
The claim is now that the Segre class of $J_E(X)$ in $\Pbo(E^\vee)$ 
determines the Segre term for the Chern-Schwartz-MacPherson class of $X$,
at least after push-forward to $M$.

\begin{theorem}[\cite{MR4010562}]\label{thm:CSMemb}
Let $i:X\hookrightarrow M$ be a closed embedding of a scheme $X$ in a 
nonsingular variety $M$, defined by a section of a vector bundle $E$ of 
rank $>\dim M$. Then with notation as above, $i_*\csm(X)$ equals 
\begin{equation}\label{eq:CSMemb}
c(TM)\cap 
\pi_* \left(\frac{c(\pi^*E^\vee\otimes \cO_{E^\vee}(1))}{c(\cO_{E^\vee}(1))}\cap
\left(s(J_E(X),\Pbo(E^\vee))^\vee \otimes_{\Pbo(E^\vee)} \cO_{E^\vee}(1)\right)
\right)\saf.
\end{equation}
\end{theorem}

Despite its rather complicated shape,~\eqref{eq:CSMemb} is straightforward to 
implement in a system capable of computing Segre classes; for example, 
Macaulay2 enhanced with the package 
{\tt SegreClasses}~(\cite{SegreClassesSource}) for computations in products of
projective space. Concrete examples may be found in~\cite[\S1]{MR4010562}.

Theorem~\ref{thm:CSMemb} is proved by
realizing $J_E(X)$ as the singularity subscheme of a hypersurface in $\Pbo(E^\vee)$,
applying Theorem~\ref{ss:CSMhyp}, and computing the push-forward by using
standard intersection-theoretic calculus and the naturality of $c_*$. The result
is that if $X$ is given by a section of a vector bundle~$E$, then~\eqref{eq:CSMemb}
computes
\begin{equation}\label{eq:allr}
i_*\csm(X) - \frac{c(TM)}{c(E)} c_\text{top}(E)\cap [M]
\end{equation}
(\cite[Theorem~2.5]{MR4010562}).
If the rank of $E$ exceeds the dimension of $M$ (as required in 
Theorem~\ref{thm:CSMemb}), then the second term vanishes, and the theorem
follows. We will come back to the more general case in~\S\ref{ss:Mil}.
To our knowledge, the auxiliary hypersurface used in this argument was first 
introduced by Callejas-Bedregal, Morgado, and Seade in~\cite{miclalci}, in the case 
of local complete intersections. The construction was also considered independently
by Ohmoto (\cite{Ohm}) and Xia Liao (\cite{Liao}).

The class~\eqref{eq:CSMemb} may be interpreted unambiguously as a class in $A_*(X)$, and it
is likely that it simply equals $\csm(X)$, but the argument we just sketched only 
shows the equality in $A_*(M)$. 

The reader will certainly notice similarities between the statement of 
Theorem~\ref{thm:CSMemb} and the case of hypersurfaces treated
in~\S\ref{ss:CSMhyp}. The new statement does recover Theorem~\ref{thm:CSMpur2}
(after push-forward to $M$) in the hypersurface case, as we see in the example that
follows.

\begin{example}\label{ex:hypemb}
Let $X$ be the hypersurface defined by a section $s$ of a line bundle $\cL\cong \cO(X)$
on a nonsingular variety $M$.
We may view $X$ as the zero scheme of the section $(s,s,\dots, s)$ of 
$E=\cO(X)^{\oplus r+1}$, for any $r\ge 0$. Then
\[
\Pbo(E^\vee) = \Pbo(\cO(-X)^{\oplus r+1}) \cong \Pbb^r \times M\saf;
\]
via this identification, $\cO_{E^\vee}(1)\cong \cO_{\Pbb^r\times M}(1)\otimes \pi^*\cO(X)$.
Therefore
\[
\frac{c(\pi^*E^\vee\otimes \cO_{E^\vee}(1))}{c(\cO_{E^\vee}(1))}
=\frac{(1+h)^{r+1}}{1+h+\pi^*X}\saf,
\]
where $h$ is the hyperplane class in $\Pbb^r\times M$.
The scheme $J_E(X)$ is locally defined by the ideal
\[
((y_0+\cdots+y_r) ds, s)
\]
in $\Pbb^r\times M$, where $y_i$ are homogeneous coordinates in $\Pbb^r$.
Note that $(ds,s)$ is the ideal of the singularity subscheme $JX$.
A generalization of the residual formula for Segre classes \eqref{eq:ressc} shows
that
\begin{multline*}
s(J_E(X),\Pbo(E^\vee))^\vee \otimes_{\Pbo(E^\vee)} \cO_{E^\vee}(1) \\
=\frac{h }{(1+h)(1+\pi^*X)}\cap \pi^*[X]
+\frac{1+h+\pi^* X}{(1+h)(1+\pi^*X)}\cap \pi^*\left(s(JX,M)^\vee\otimes_M \cO(X)\right)\saf.
\end{multline*}
Therefore, the term to push forward in~\eqref{eq:CSMemb} evaluates to
\[
\frac{(1+h)^r \cdot h}{(1+h+\pi^*X)(1+\pi^*X)}\cap \pi^*[X]
+\frac{(1+h)^r}{1+\pi^*X}\cap \pi^*\left(s(JX,M)^\vee\otimes_M \cO(X)\right)
\saf.
\]
The push-forward is carried out by the projection formula and reading off the
coefficient of $h^r$. The second summand pushes forward to 
\[
\frac 1{1+X}\cap \left(s(JX,M)^\vee\otimes_M \cO(X)\right)
=\left( c(\cO(X)\cap s(JX,M)\right)^\vee \otimes_M\cO(X)
\saf.
\]
The first summand pushes forward to
\[
\left(\text{coefficient of $h^r$ in } 
\frac{(1+h)^r \cdot h}{1+h+\pi^*X}\right)\cap \frac {[X]}{1+X}
\]
and elementary manipulations
 evaluate the coefficient, giving
\[
\left(1-\frac{X^r}{(1+X)^r}\right)\cap \frac {[X]}{1+X}\saf.
\]
In conclusion, \eqref{eq:CSMemb} equals
\[
c(TM)\cap \left(\frac {[X]}{1+X} +\left( c(\cO(X)\cap s(JX,M)\right)^\vee \otimes_M\cO(X)
-\frac{X^r}{(1+X)^{r+1}}\cap [X]\right)\saf.
\]
Theorem~\ref{thm:CSMemb} asserts that for $r+1>\dim M$, this expression
equals $i_*\csm(X)$. And indeed, if $r+1>\dim M$, the last term vanishes and 
we recover the expression in Theorem~\ref{thm:CSMpur2}.
\qede\end{example}

As with Proposition~\ref{prop:CSMincexc}, one intriguing feature of
Theorem~\ref{thm:CSMemb} is the vast degree of freedom in the choice of the
data needed to apply it---here, the vector bundle $E$ and the section of $E$ 
whose zero-scheme defines $X$. The fact that different choices of 
bundles or of defining sections lead to the same result reflects sophisticated identities 
involving the relevant Segre classes, for which we do not know a more direct proof. 

\subsection{Milnor classes}\label{ss:Mil}
\index{Milnor!class}
We have seen that Parusi\'nski's\index{Milnor!number!Parusi\'nski's} 
generalization of the Milnor number to complex hypersurfaces with
arbitrary singularities satisfies~\eqref{eq:parmil}:
\[
\mu(X)=(-1)^{\dim X} (\chi(X_\text{gen})-\chi(X))\saf,
\]
where $X_\text{gen}$ is a nonsingular hypersurface linearly equivalent 
to~$X$. Also, we have seen that $\chi(X_\text{gen})=\int\cvir(X)$
\index{Chern!class!virtual} 
(Proposition~\ref{prop:cvirgen}) and $\chi(X)=\int\csm(X)$ \eqref{eq:PH}. 
Therefore,
\[
\mu(X) = (-1)^{\dim X} \int \cvir(X)-\csm(X) \saf.
\]
This equality motivates the following definition, which makes sense over any algebraically
closed field of characteristic~$0$.

\begin{definition}\label{def:Milnorclass}
Let $X$ be a local complete intersection. The {\em Milnor class\/} of $X$ is
the class
\[
\cM(X):=(-1)^{\dim X} \left(\cvir(X)-\csm(X)\right)
\]
where $\cvir(X)$ is the class of the virtual tangent bundle of $X$.
\qede\end{definition}

(Recall that being a local complete intersection in a nonsingular variety is an
intrinsic notion, cf.~\cite[Remark~{II.8.22.2, p.185}]{MR0463157}, and that the virtual
tangent bundle of a local complete intersection is well-defined as a class in the
Grothendieck group of vector bundles on $X$.)

Definition~\ref{def:Milnorclass} would place the class in $A_*(X)$. The class
is clearly supported on the singular locus $X^\text{sing}$ of $X$, and in 
the case of a hypersurface $X$ we will produce below a well-defined class
in $A_*(JX)$ whose image in $A_*(X)$ is the class of 
Definition~\ref{def:Milnorclass}. Formulas explicitly localizing the class to
the singular locus are also given in the local complete intersection case
in~\cite{MR1873009} (over $\Cbb$, and in homology).

One could extend the definition of the Milnor class to more general schemes~$X$, 
as measuring the difference between $\csm(X)$ and $\cfu(X)$ or $\cfj(X)$
(cf.~\eqref{eq:fufjvir}). However,
recall that in general $\cfu(X)\ne \cfj(X)$ for schemes that are not local complete
intersections, so this would require a choice that seems arbitrary.
For this reason, we prefer to only consider the Milnor class for local complete 
intersections.

The geometry associated to Milnor classes of hypersurfaces and more generally
local complete intersections has been studied very thoroughly. We mention
\cite{MR2002g:14005}, \cite{MR1873009}, \cite{MR3053711}, \cite{miclalci}
among many others, as well as~\cite{MR1695362}, \cite{MR1720876}, where
(to our knowledge) the notion was first introduced and studied. 
The contribution~\cite{milnorhandbook} to this Handbook includes a thorough
survey of Milnor classes. Here we focus specifically on the relation between 
Milnor classes and Segre classes, and on consequences of this relation.\smallskip

First, we note that the Milnor class of a hypersurface $X$ of a
nonsingular variety~$M$ admits an expression in terms of a Segre 
operator~\eqref{eq:Sabsha}:
\begin{equation}\label{eq:MilSe}
\cM(X) = c(T_\text{vir} X)\cap \Segre_{T^\vee M}([\cE])_\vee\saf,
\end{equation}
where $[\cE]$ is the class of the exceptional divisor of the blow-up
$\pi\colon B\ell_{JX}M\to M$; as pointed out in~\S\ref{ss:CSMhyp},
$\cE$ may be viewed as a cycle in $\Pbo(T^\vee M)$,
so $\Segre_{T^\vee M}([\cE])$ is defined.
To verify~\eqref{eq:MilSe}, let $\cX=\pi^{-1}(X)$;
then $s(X,M) = \pi_* s(\cX, B\ell_{JX}M)$,
by the birational invariance of Segre classes,
and this implies the expression
\begin{align*}
\cvir(X) &= c(TM|_X)\cap \pi_*\left(\frac{[\cX]}{1+\cX}\right) \\
\intertext{for the virtual Chern class of $X$.
Also, note that $\cO_{T^\vee M}(1)|_{\cX}\cong \cO(\cX-\cE)|_{\cX}$
(this follows from~\eqref{eq:O1}); by Lemma~\ref{lem:charcy},
\eqref{eq:sabbah} implies}
\csm(X) &=c_*(\one_X)=c(TM|_X)\cap \pi_* \left(\frac{[\cX]-[\cE]}{1+\cX-\cE}\right)\saf.
\end{align*}
Therefore
\begin{align*}
(-1)^{\dim X}(\cvir(X)-\csm(X))
&=(-1)^{\dim X} c(TM|_X) \cap \pi_*\left(
\frac{[\cX]}{1+\cX} - \frac{[\cX]-[\cE]}{1+\cX-\cE}\right) \\
&=(-1)^{\dim X} c(TM|_X) \cap \pi_*\left(
\frac 1{1+\cX} \cdot \frac {[\cE]}{1+\cX-\cE} \right) \\
&=\frac{c(TM|_X)}{1+X} \cap \pi_*\left(
\frac {[\cE]}{1-\cX+\cE} \right)_\vee \\
&=\frac{c(TM|_X)}{1+X} \cap \pi_*\left(
c(\cO_{T^\vee M}(-1))^{-1}\cap [\cE] \right)_\vee \\
&=c(T_\text{vir} X)\cap \Segre_{T^\vee M}([\cE])_\vee
\end{align*}
as claimed.
By Theorem~\ref{thm:Sabthm}, identity~\eqref{eq:MilSe} may be written
\[
\cM(X) = c(\cO(X))^{-1}\cap c_*(\nu_{JX})
\]
for the constructible function $\nu_{JX}$ whose characteristic cycle is the exceptional
divisor~$\cE$. As a Lagrangian cycle, $[\cE]$ is a linear combination of cycles of
conormal spaces of subvarieties of $JX$:
$[\cE] = \sum_W n_W [N^\vee_WM]$;
then, as prescribed by Definition~\ref{def:charcyc}:
\[
\nu_{JX} = \sum_W (-1)^{\dim W} n_W \one_W\saf.
\]
Over $\Cbb$, and if $X$ is reduced, Parusi\'nski and Pragacz 
(\cite[Corollary~2.4]{MR2002g:14005}) prove that
\[
\nu_{JX}=(-1)^{\dim X}(\chi_X-\one_X)\saf,
\]
where for $p\in X$, $\chi_X(p)$ denotes the Euler characteristic of the Milnor fiber
of $X$ at~$p$. (In~\cite{MR2002g:14005}, $\nu_{JX}$ is denoted $\mu$.)

In general, note that $\cE$ is the projectivized normal cone of $JX$. 
If $Y$ is any subscheme of $M$, then we can associate to $Y$ a constructible
function $\nu_Y$ by letting $\nu_Y = \sum_W (-1)^{\dim W} n_W \one_W$, where
the subvarieties $W$ are the supports of the components of the normal cone
$C_YM$ and $n_W$ is the multiplicity of the component supported on $W$.
Then the class $c_*(\nu_Y)$ generalizes the class $c_*(\nu_{JX})=c(\cO(X))\cap \cM(X)$.
Kai Behrend (\cite[Proposition~4.16]{MR2600874}) proves that if $Y$ is endowed with
a {\em symmetric obstruction theory\/} (the singularity subscheme of a hypersurface gives
an example), then the $0$-dimensional component of
$c_*(\nu_Y)$ equals the corresponding 
`virtual fundamental class';\index{virtual!fundamental class} 
its degree is a Donaldson-Thomas type invariant. 

Expression~\eqref{eq:MilSe} for the Milnor class may be recast in terms of the Segre 
class~$s(JX,M)$.

\begin{proposition}\label{prop:Milhyp}
Let $X$ be a hypersurface in a nonsingular variety $M$. Then
\[
\cM(X)=(-1)^{\dim M} c(TM|_{JX})\cap \left((c(\cO(X))\cap s(JX,M))^\vee\otimes_M
\cO(X)\right)\saf.
\]
\end{proposition}

This is an immediate consequence of Theorem~\ref{thm:CSMpur2}. Indeed, 
\begin{align*}
\cvir(X) &=c(T_\text{vir} X)\cap [X] = c(TM|_X) c(N_XM)^{-1}\cap [X]
=c(TM|_X) c(\cO(X))^{-1}\cap [X] \\
&=c(TM|_X)\cap ([X]\otimes_M \cO(X))\saf.
\end{align*}

Note that we have written the right-hand side in Proposition~\ref{prop:Milhyp} as a class 
in $A_*(JX)$. The statement means that this class pushes forward to the difference defining 
the Milnor class in Definition~\ref{def:Milnorclass}. 
The formula also implies that every connected component of $JX$ has a well-defined 
contribution to the Milnor class of~$X$. Of course if a component is supported on an
isolated point $p$, and $\hat p$ denotes the part of $JX$ supported on $p$, then
the contribution of $p$ to the Milnor class is
\[
(-1)^{\dim M} c(TM|_{JX})\cap \left((c(\cO(X))\cap s(\hat p,M)^\vee)\otimes_M
\cO(X)\right)=s(\hat p,M)\saf,
\]
a class whose degree equals (in the complex setting) the ordinary Milnor number, 
cf.~\S\ref{ss:Milnor}.

Proposition~\ref{prop:Milhyp} may be formulated in terms of the 
`$\mu$-class'\index{$\mu$-class}
of~\cite{MR97b:14057}, already mentioned in~\S\ref{ss:Milnor}:
\[
\mu_{\cO(X)}(JX):=c(T^\vee M\otimes \cO(X))\cap s(JX,M)\saf.
\]
Indeed, simple manipulations using~\eqref{eq:notpr1} and~\eqref{eq:notpr2} show that
\begin{align*}
\cM(X)&=(-1)^{\dim M} c(\cO(X))^{\dim X} \left(\mu_{\cO(X)}(JX)^\vee \otimes_M \cO(X)\right)\saf,
\intertext{or, equivalently,}
\mu_{\cO(X)}(JX) &= (-1)^{\dim M} c(\cO(X))^{\dim X}\left(\cM(X)^\vee\otimes_M \cO(X)\right)\saf.
\end{align*}
It is somewhat remarkable that $\cM(X)$ and $\mu_{\cO(X)}(JX)$ are exchanged
by the `same' operation. Such involutions are not uncommon in the theory,
see~\cite{MR3232012}, \cite{MR3554655}. 

The $\mu$-class has applications to e.g., duality, and such
applications can be formulated in terms of the Milnor class. We give one explicit example.

\begin{example}
Let $M$ be a nonsingular projective variety, 
and let $H$ be a hyperplane tangent
to $M$, that is, a point of the 
{\em dual variety\/}\index{dual variety} 
$M^\vee$ of $M$; so $X=M\cap H$ 
is a singular hypersurface of $M$. Rewriting~\cite[Proposition~2.2]{MR97b:14057}
in terms of the Milnor class, we obtain that {\em the codimension of $M^\vee$ in the dual
projective space is the smallest 
integer $r\ge 1$ such that the component of dimension $r-1$ in the class
\[
(1+X)^{\dim M}\left(\cM(X)^\vee\otimes_M \cO(X)\right)
\]
does not vanish.\/} Further, the projective degree of this component (viewed as a class
in the dual projective space) equals the multiplicity of $M^\vee$ at $H$, up to
sign. (This result generalizes~\eqref{eq:muldis}.)
We do not know a `Segre class-free' proof of these facts.

For a concrete example, consider $M=\Pbb^2\times \Pbb^1$, embedded in $\Pbb^5$
by the Segre embedding. Using coordinates $(x_0:x_1:x_2)$ for the first factor, 
and $(y_0:y_1)$ for the second factor, let $X$ be the hypersurface with equation
$x_0 y_1=0$: Thus, $X$ is a hyperplane section via the Segre embedding, and $X$
is the transversal union of two surfaces isomorphic to $\Pbb^1\times \Pbb^1$, resp.~,
$\Pbb^2$, meeting along a $\Pbb^1$.
If $h_1$, resp., $h_2$ denote the pull-back of the hyperplane class from the first, resp.~second 
factor, then the reader can verify that
\begin{align*}
\cvir(X) &= \left((h_1 + h_2)+ (2h_1^2 + 3h_1h_2) + 4h_1^2h_2\right)\cap [\Pbb^2\times \Pbb^1]\saf, \\
\csm(X) &=\left((h_1 + h_1)+ (2h_1^2 + 4h_1h_2) + 5h_1^2h_2\right)\cap [\Pbb^2\times \Pbb^1]\saf.
\end{align*}
It is easy to obtain these expressions `by hand'; in any case, the following application
of~\cite{CharacteristicClassesSource} will confirm the second assertion.
{\begin{verbatim}
i1 : load("CharacteristicClasses.m2")

i2 : R=MultiProjCoordRing {2,1}

i3 : CSM ideal(R_0*R_4)

       2       2
o3 = 5h h  + 2h  + 4h h  + h  + h
       1 2     1     1 2    1    2
\end{verbatim}
}
Therefore
\begin{gather*}
\cM(X) = (-h_1h_2-h_1^2h_2)\cap [\Pbb^2\times \Pbb^1]\saf, \\
(1+X)^{\dim M}\left(\cM(X)^\vee\otimes \cO(X)\right) = -h_1h_2 
\cap [\Pbb^2\times \Pbb^1]\saf.
\end{gather*}
In fact, it is easy to verify (by hand!) that for the corresponding hypersurface in 
$M=\Pbb^n\times \Pbb^1$, we have
\begin{gather*}
\cM(X) = (-1)^{n+1} (1+h_1)^{n-1}h_1h_2\cap [\Pbb^n\times \Pbb^1]\saf, \\
(1+X)^{\dim M}\left(\cM(X)^\vee\otimes \cO(X)\right) = (-1)^{n+1} h_1h_2 
\cap [\Pbb^n\times \Pbb^1]\saf.
\end{gather*}
The conclusion is that $M^\vee$ has codimension~$n$ in the dual $\Pbb^{2n+1}$,
and is nonsingular at the point corresponding to this hyperplane section.
(In fact, it is well known that the Segre embedding of $\Pbb^n\times\Pbb^1$ 
in $\Pbb^{2n+1}$ is isomorphic to its dual variety for all~$n\ge 1$
\cite[Example~9.1]{MR2113135}.)
\qede\end{example}

It is natural to ask about extensions of Proposition~\ref{prop:Milhyp} to more general local complete
intersections. For us, $X\subseteq M$ is a 
{\em local complete intersection\/}\index{local complete intersection} 
if $X$ is the
zero-scheme of a regular section of a vector bundle $E$ defined on some neighborhood of $X$.
For notational convenience, we will restrict $M$ if necessary and assume that $E$~is
defined over the whole of $M$. Recall that the bundle $E$ and the section defining~$X$ 
determine a closed subscheme $J_E(X)$ of $\Pbo(E^\vee|_X)$ (Definition~\ref{def:JEX}).
We view~$J_E(X)$ as a subscheme of $\Pbo(E^\vee)$, and denote by $\pi\colon
\Pbo(E^\vee)\to M$ the projection.

\begin{theorem}\label{thm:Millci}
Let $i:X\hookrightarrow M$ be a local complete intersection in 
a nonsingular variety $M$, obtained as the zero-scheme of a regular section of a
vector bundle $E$ of rank $\codim_XM$. Then $(-1)^{\dim X+1} i_*\cM(X)$ equals
\begin{equation}\label{eq:Millci}
c(TM)\cap 
\pi_* \left(\frac{c(\pi^*E^\vee\otimes \cO_{E^\vee}(1))}{c(\cO_{E^\vee}(1))}\cap
\left(s(J_E(X),\Pbo(E^\vee))^\vee \otimes_{\Pbo(E^\vee)} \cO_{E^\vee}(1)\right)
\right)
\end{equation}
in $A_*(M)$.
\end{theorem}

This statement may seem puzzling at first, since~\eqref{eq:CSMemb} and~\eqref{eq:Millci}
are {\em the same formula,\/} yet the first is stated to equal $i_*\csm(X)$ (for arbitrary $X$)
and the second equals $i_*\cM(X)$ (for local complete intersections). The difference
is in the ranks of the bundle $E$: in Theorem~\ref{thm:CSMemb} the rank is required to
exceed the dimension of the ambient variety $M$, while in Theorem~\ref{thm:Millci}
the rank is {\em equal to the codimension of~$X$.\/} Both statements are consequences of the
more general result~\eqref{eq:allr}: the formula evaluates the CSM class up to a correction
term, which is $0$ if $\rk E\gg 0$, and it is precisely $i_*(\cvir(X))$ if $X$ is a local 
complete intersection and $\rk E = \codim_XM$. 

\begin{example}
Let $X\subseteq M$ be a hypersurface defined by a section $s$ of $\cO(X)$.
In Example~\ref{ex:hypemb} we viewed $X$ as the zero scheme of the section 
$(s,\dots, s)$ of~$\cO(X)^{\oplus r+1}$, and showed that~\eqref{eq:CSMemb}
evaluates to
\[
c(TM)\cap \left(\frac {[X]}{1+X} + \left( c(\cO(X)\cap s(JX,M)\right)^\vee \otimes_M\cO(X)
-\frac{X^r}{(1+X)^{r+1}}\cap [X]\right)\saf.
\]
The case considered in Theorem~\ref{thm:Millci} corresponds to $r=0$, for which
the formula gives
\[
c(TM)\cap \left(\left( c(\cO(X)\cap s(JX,M)\right)^\vee \otimes_M\cO(X)\right)\saf,
\]
agreeing with $(-1)^{\dim X+1} i_*\cM(X)$ by Proposition~\ref{prop:Milhyp}.
In this sense, Theorem~\ref{thm:Millci} generalizes Proposition~\ref{prop:Milhyp}.
\qede\end{example}

Expression~\eqref{eq:Millci} shows that, as in the case of the `characteristic'
classes reviewed in this section, the Milnor class of a local complete intersection
is determined by a Segre class, $s(J_E(X),\Pbo(E^\vee))$ in this case. If $M=\Pbb^n$,
this class can be computed using e.g., the Macaulay2 package~\cite{SegreClassesSource};
the other ingredients in~\eqref{eq:Millci} are straightforward. For explicit formulas
and examples, see~\cite{MR4010562}.


\section{L\^e cycles}\label{sec:Le}

\subsection{St\"uckrad-Vogel intersection theory and van Gastel's result}\label{ss:SVvG}
\index{St\"uckrad-Vogel intersection theory}\index{intersection theory!St\"uckrad-Vogel}
An `excess intersection'\index{excess intersection} 
situation occurs when loci intersect in higher than expected
dimension. For example, $r$ hypersurfaces in a nonsingular variety $M$ are expected
to intersect in a codimension-$r$ subscheme; if they intersect along a subscheme of
higher dimension, `excess' intersection occurs. 

The ability to deal with excess intersection is one the successes of 
Fulton-MacPherson's intersection theory. If $X_1,\dots, X_r$ are hypersurfaces, and 
$Z$ is a connected component of $X_1\cap \cdots \cap X_r$, then 
the contribution of $Z$ to the intersection product of the classes of the hypersurfaces
may be written as
\begin{equation}\label{eq:FMag}
\left\{\prod_{i=1}^r (1+X_i)\cap s(Z,M)\right\}_{\dim M-r}\saf.
\end{equation}
For this, view $X_1\cdots X_r$ as $(X_1\times\cdots\times X_r)\cdot \Delta$,
where $\Delta$ is the diagonal in $M\times \cdots\times M$:
we have $(X_1\times\cdots\times X_r)\cap \Delta\cong X_1\cap\cdots \cap X_r$, 
$\Delta\cong M$, and we consider the fiber diagram
\[
\xymatrix{
X_1\cap\cdots\cap X_r \ar[r] \ar[d] & \Delta\cong M \ar[d] \\
X_1\times\cdots \times X_r \ar[r] & M\times\cdots \times M\saf.
}
\]
We can view $Z$ as a connected component of $(X_1\times\cdots\times X_r)
\cap \Delta$.
The restriction of the normal bundle $N_{X_1\times\cdots\times X_r}
(M\times\cdots\times M)$
to $Z$ is then isomorphic to $\oplus_i \cO(X_i)|_Z$, so that its Chern class
is (the restriction of) $\prod_{i=1}^r (1+X_i)$. Then~\eqref{eq:FMag} follows
from~\eqref{eq:FMcont}.
The fact that $Z$ may be of dimension higher than $\dim M-r$ is precisely accounted for
by the Segre class of $Z$ in $M$.

An alternative approach to intersection theory in projective space, dealing differently with 
excess intersection,
was developed by J\"{u}rgen St\"uckrad and Wolfgang Vogel (\cite{MR783085}, and see
\cite{MR1724388} for a comprehensive account). In excess intersection situations, this 
approach produces a {\em cycle\/} after a transcendental extension of the base field;
the intersection product can be computed from this cycle, and agrees with the
Fulton-MacPherson intersection product.

We review the construction of the St\"uckrad-Vogel `$v$-cycle', essentially following the 
`geometric' account 
given in~\cite{MR1079843}, where it is also extended to the setting of more general schemes.
However, we only present the construction in the somewhat limited scope needed for our
application, and we make a substantial simplification, at the price of only obtaining a
cycle depending on general choices. (The St\"uckrad-Vogel construction produces a
well-defined cycle independent of such choices, after a transcendental extension of the 
base field.) 

Let $V$ be a variety, $\cL$ a line bundle on $V$, $s_1,\dots, s_r$ nonzero sections of 
$\cL$, and $\cD$ the collection of the corresponding Cartier divisors $D_1,\dots, D_r$.
The sections $s_1,\dots, s_r$ span a subspace of $H^0(V,\cL)$; by a `$\cD$-divisor' 
we will mean a divisor defined by a section of this subspace. 
Let $Z=D_1\cap\cdots\cap D_r$.

The following inductive procedure constructs a cycle on $Z$, depending on general
choices of $\cD$-divisors. The procedure
only involves proper intersections with Cartier divisors, which is defined at the level
of cycles: if $W$ is a variety, and a Cartier divisor~$D$ intersects it properly, i.e., it
does not contain it, then $D\cap W$ is a Cartier divisor in $W$ (or empty), and we
denote by $D*W$ the corresponding cycle (or $0$). The class of this cycle is the
intersection product of $[W]$ by $D$ in the Chow group.
By linearity, this operation is extended to cycles
$\rho$ such that $D$ does not contain any component of $\rho$: then $D*\rho$
denotes the corresponding `proper intersection' product.

The algorithm may be described as follows.
\begin{itemize}
\item Let $\alpha^0=0$, $\rho^0=V$;
\item For $j>0$: if $\rho^{j-1}\ne 0$, then a general $\cD$ divisor $D'_j$ intersects
$\rho^{j-1}$ properly; let $D'_j * \rho^{j-1}=\alpha^j+\rho^j$, where $\alpha^j$ 
collects the components of the intersection product that are contained in 
$Z=D_1\cap\cdots\cap D_r$;
\item This procedure stops when $\rho^j=0$.
\end{itemize}

It is easy to see that a general $D'_j$ does intersect $\rho^{j-1}$ properly, so it is always 
possible to make the choice needed in the second point.
Also, let $s'_j$ be the section defining $D'_j$. The construction implies that if $\rho^{j-1}\ne 0$, 
then $s'_j$ is not in the span of $s'_1,\dots, s'_{j-1}$. In particular, the procedure must
stop at some $j\le r$. We set $\alpha^i=\rho^i=0$ for $j< i\le r$.

\begin{definition}\label{def:vcycle}
We denote by $\cD\capdot V$ the sum $\sum_{i=0}^r \alpha^{i}$. This is a cycle on
$Z=D_1\cap\cdots\cap D_r$.
\qede\end{definition}

\begin{remark}
We chose the notation~$\cD\capdot V$ to align with the notation used by van~Gastel
(in a more general context). This is the 
`$v$-cycle'\index{$v$-cycle} 
determined by $\cD$.
The definition presented above only depends on the linear system spanned by the
sections defining the divisors $D_i$ in the collection~$\cD$.
\qede\end{remark}

According to our definition, the cycle $\cD\capdot V$ depends on the choice of the 
divisors~$D'_j$. One of the advantages of the more sophisticated St\"uckrad-Vogel construction 
is that it yields a well-defined cycle independent of any choice, albeit after extending the ground 
field. However, we are only interested in the rational equivalence class of~$\cD\capdot V$,
and this is independent of the choices. In fact, the following holds.

\begin{theorem}\label{thm:vG}
With notation as above,
\[
[\cD\capdot V] = s(Z,V)\otimes_V \cL^\vee
\]
in $A_*(Z)$.
\end{theorem}

In the context of St\"uckrad-Vogel intersection theory, this is~\cite[Corollary~3.6]{MR1079843}.
Theorem~\ref{thm:vG} can also be proved by interpreting $\cD\capdot V$ in terms of
the blow-up of~$V$ along $Z$; this naturally identifies its rational equivalence class as a 
`tensored Segre class'\index{Segre class!tensored} 
in the sense of~\cite{MR3599436}, up to a product by $c(\cL)$.

By~\eqref{eq:notpr2}, Theorem~\ref{thm:vG} is equivalent to
\begin{equation}\label{eq:SegSV}
s(Z,V)=[\cD\capdot V]\otimes_V \cL\saf.
\end{equation}
Using~\eqref{eq:FMag}, we see that
\[
D_1\cdots D_r\cap [V] = \left\{c(\cL)^r \cap \left([\cD\capdot V]\otimes_V \cL\right)\right\}_{\dim V-r}
\]
in $A_{\dim V-r} Z$. This is equivalent to the formula
\[
D_1\cdots D_r\cap [V] = \sum_{j=0}^r c_1(\cL)^{r-j}\cap \alpha^j\saf,
\]
cf.~\cite[Proposition~1.2 (c)]{MR1079843}. 

In conclusion, the St\"uckrad-Vogel construction offers an alternative to the treatment
of excess intersection of linearly equivalent divisors. By~\eqref{eq:SegSV}, the relevant
Segre class may be computed in terms of the $v$-cycle. Among other pleasant features,
this approach leads to `positivity' statements for Segre classes: by construction, the
$v$-cycle is effective; by~\eqref{eq:SegSV}, the non-effective parts of the Segre class
of the intersection of sections of a line bundle $\cL$ are due to the `tensor' operation
$\_ \otimes_V\cL$. (Cf.~\cite[Corollary~1.3]{MR3599436}.)

\subsection{L\^e cycles and numbers}
Broadly speaking, one can view singularities as arising because of an excess intersection.
For example, if $X$ is a hypersurface of $\Pbb^n$, with equation $F(x_0,\dots, x_n)=0$,
the singular locus of $X$ is the intersection of the $n+1$ hypersurfaces with equations
$\partial F/\partial x_i=0$, $i=0,\dots,n$. Then $X$ is singular precisely when these
hypersurfaces meet with excess intersection. The scheme they define is the singularity
subscheme~$JX$ of Definition~\ref{def:JX}; and the Segre class that is relevant to the 
Fulton-MacPherson approach is precisely, and not surprisingly, the class $s(JX,M)$ 
that appears in most results concerning hypersurfaces reviewed in~\S\ref{sec:numinvs} 
and~\ref{sec:Charcla}. Taking the point of view of~\S\ref{ss:SVvG}, we could express
these results in terms of the $v$-cycle corresponding to the linear system spanned by
the partials. 

A closely related construction was provided (independently from St\"uckrad and Vogel)
by Massey in~1986, 
leading to his definition of 
{\em L\^e cycles\/}\index{L\^e!cycles and classes} 
(\cite{MR1031905}, \cite{MR1094048},
\cite{MR1441075}). The theory and applications of L\^e cycles are surveyed 
in~\cite{masseyhandbook}. Massey's definition may be given for analytic functions
defined for a nonempty open subset of $\Cbb^{n+1}$. We are going to consider the
case of a homogeneous polynomial, and view it as the generator of the ideal of a
hypersurface in $\Pbb^n$. We will follow~\cite[\S7.7]{masseyhandbook} for the resulting
{\em projective\/} L\^e cycles. The considerations that follow would hold over any 
algebraically closed field of characteristic~$0$.

Let $F(x_0,\dots, x_n)$ be a homogeneous polynomial, defining a projective
hypersurface $X\subseteq \Pbb^n$. 
Massey's definition can be phrased in terms very close to the inductive definition
given in~\S\ref{ss:SVvG}, applied to the linear system spanned by the derivative
$\partial F/\partial x_i$ of $F$. We give the affine definition of the cycles first.
\begin{itemize}
\item Let $\Gamma^{n+1}=\Cbb^{n+1}$, $\Lambda^{n+1}=0$;
\item For $1\le k\le n+1$, define $\Gamma^{k-1}$ and $\Lambda^{k-1}$ by downward
induction by
\[
\Gamma^k * V\left(\frac{\partial F}{\partial x_{k-1}}\right) = \Lambda^{k-1} + 
\Gamma^{k-1}\saf,
\]
where the (cycle-theoretic) intersection is assumed to be proper, and $\Lambda^{k-1}$
consists of the components contained in $JX$, $\Gamma^{k-1}$ of the other components.
\end{itemize}
Following~\cite[\S7.7]{masseyhandbook}:

\begin{definition}
The {\em projective L\^e cycles\/} of $X$ are the cycles $\Llambda^k_X:=\Pbb(\Lambda^{k+1})$.
\qede\end{definition}

The projectivization of the cycles $\Gamma^j$ are the {\em projective relative 
polar\index{polar!cycles and classes} 
cycles\/} of $X$. 

The L\^e cycles of $X$ evidently depend on the chosen coordinates, and may not be
defined for certain choices as the cycles appearing in the definition may fail to meet properly.
Massey proves that a general choice of coordinates guarantees that the intersections are
proper, so that the corresponding L\^e cycles exist. In the following, the L\^e cycles we
consider will be assumed to be obtained from a general choice of coordinates.

Comparing Massey's definition with 
Definition~\ref{def:vcycle}, we recognize that the sum $\sum_{k=0}^n \Llambda^k$
of L\^e cycles may be viewed as an instance of the $v$-cycle $\cD\capdot \Pbb^n$,
where $\cD$~is the collection of partial derivatives of $F$. The dependence on the
choices (e.g., the choice of coordinates in Massey's definition, or the choice of $D'_j$
in Definition~\ref{def:vcycle}) is eliminated once one passes to rational equivalence,
so that
\[
[\cD\capdot \Pbb^n] = \sum_k [\Llambda^k_X]
\]
in $A_*(JX)$ if all choices are general. (Note however that the indexing conventions 
differ, so that with notation as in~\S\ref{ss:SVvG}, $[\Llambda^k_X] = [\alpha^{n-k}]$.)

With this understood, the next result follows immediately from Theorem~\ref{thm:vG}.

\begin{proposition}\label{prop:MasSeg}
Let $X$ be a degree-$d$ hypersurface in $\Pbb^n$, with projective L\^e 
cycles~$\Llambda^k_X$. Then
\begin{equation}\label{eq:MasSeg}
\sum_k [\Llambda^k_X] = s(JX,\Pbb^n)\otimes_{\Pbb^n} \cO(-(d-1))
\end{equation}
in $A_*(JX)$.
\end{proposition}

\begin{remark}\label{rem:Inv}
For $M=\Cbb^n$, Gaffney and Gassler (\cite{MR1703611}) propose
a generalization of classes of L\^e cycles based on more general ideals, which in
the case of the Jacobian ideal of a polynomial defining a hypersurface $X$ is
closely related with the Segre class of $JX$ (cf.~the definition of the {\em Segre
cycle\/} $\Lambda^{\bf g}_k(I,Y)$ in~\cite[(2.1)]{MR1703611}). Partly motivated
by this work, Callejas-Bedregal, Morgado, and Seade
gave a definition of L\^e cycles for a hypersurface $X$ of a compact complex 
manifold $M$, which amounts essentially to a cycle representing the Segre class 
$s(JX,M)$ 
(\cite[Definition~3.2]{MR3232012}). This definition is {\em not\/} compatible with 
Massey's L\^e cycles for $M=\Pbb^n$, as the authors opted to omit the extra tensor
appearing in~\eqref{eq:MasSeg}. Since the `hyperplane' defined in~\cite{MR1703611} 
differs from the tautological class used in \cite{MR3232012}, this causes a 
discrepancy amounting to a twist of the line bundle of the hypersurface.
This twist is accounted for in Proposition~\ref{prop:MasSeg}, which is 
compatible with the construction in~\cite{MR1703611}. 

See~\cite{MR3232013} and~\cite{MR3554655} for further 
discussions of~\cite[Definition~3.2]{MR3232012}. In particular, Callejas-Bedregal, 
Morgado, and Seade propose an alternative `geometric' definition in~\cite{MR3232013}
(Definition~1.3), which {\em does\/} agree with Massey's for $M=\Pbb^n$.
Also see~\cite[\S4]{milnorhandbook} (particularly Definition~4.4) for a comprehensive
account. We will come back to this definition in~\S\ref{ss:LMS}.
\qede\end{remark}

The fact that the L\^e cycles are {\em cycles\/} is important for geometric applications.
Proposition~\ref{prop:MasSeg} only computes their {\em classes\/} up to rational
equivalence, in the Chow group $A_*(JX)$ of the singularity subscheme of the
hypersurface. These classes still carry useful information, even after a push-forward
by the inclusion $\iota: JX\to \Pbb^n$. We consider the class
\[
\iota_*([\Llambda^k_X]) = \lambda^k_X [\Pbb^k]\saf,
\]
where the integers $\lambda^k_X$ are (still following Massey) called the 
{\em L\^e\index{L\^e!numbers}
numbers\/} of the hypersurface. (Massey's L\^e numbers also depend on the choice
of coordinates; again, we will assume that the choice of coordinates is sufficiently
general.) Proposition~\ref{prop:MasSeg} implies as an
immediate corollary a formula for the L\^e numbers in terms of the degrees of
the components of the Segre class (and conversely). 

\begin{corollary}\label{cor:LeSe}
Let $X\subseteq \Pbb^n$ be a hypersurface, and denote by $s_i$ the degree of 
the $i$-th dimensional component of the Segre class $s(JX,\Pbb^n)$.
Then for $k=0,\dots,n$:
\begin{align}
\lambda^k_X &= \sum_{j=k}^n \binom{n-k-1}{j-k} (d-1)^{j-k} s_j \label{eq:setola}\\
s_k &= \sum_{j=k}^n \binom{n-k-1}{j-k} (-(d-1))^{j-k} \lambda^j_X\saf.
\end{align}
\end{corollary}

\begin{proof}
Denote the hyperplane class by $H$. By Proposition~\ref{prop:MasSeg} and
the definition of~$\otimes_{\Pbb^n}$~\eqref{eq:defmytens}:
\begin{align*}
(\lambda_X^n + \lambda_X^{n-1} H + \cdots &+ \lambda_X^0 H^n)\cap [\Pbb^n] \\
&=((s_n + s_{n-1} H + \cdots + s_0 H^n)\cap [\Pbb^n])\otimes_{\Pbb^n}
\cO(-(d-1)) \\
&=\left(s_n + \frac{s_{n-1}H}{(1-(d-1)H)} + \cdots + \frac{s_0H^n}{(1-(d-1)H)^n}\right)
\cap [\Pbb^n]
\end{align*}
and the first formula follows by matching terms of equal degrees in the two
expressions. 
`Solving for $s(JX,\Pbb^n)$' in Proposition~\ref{prop:MasSeg} gives
\[
s(JX,\Pbb^n) = \sum_k [\Llambda^k_X] \otimes_{\Pbb^n} \cO(d-1)
\]
(apply~\eqref{eq:notpr2}), and the second formula follows by the same token.
\end{proof}

\begin{remark}
Formula~\eqref{eq:setola} in Corollary~\ref{cor:LeSe}:
\[
\lambda_X^k = s_k + (n-k-1)(d-1) s_{k+1}
+\binom{n-k-1}2 (d-1)^2 s_{k+2} + \cdots\saf.
\]
can be viewed as the degree
of the ordinary Segre class, `corrected' by a term
determined by the degree $d$ of the hypersurface.

In the introduction to~\cite{MR1703611}, Gaffney and Gassler state:
{\em ``\dots In fact, the Segre numbers {\rm (of the Jacobian ideal)} are just
the L\^e numbers of David Massey.''} 
Corollary~\ref{cor:LeSe} is compatible with this assertion: it is easy to
verify that the 
`Segre numbers'\index{Segre!numbers} 
of~\cite{MR1703611} agree with the right-hand 
side of~\eqref{eq:setola}. 
\qede\end{remark}

\begin{example}\label{ex:Mass}
Consider the hypersurface $X$ of $\Pbb^5$ defined by the polynomial
\[
F= x_0^7- x_1^7 - (x_2^3+x_3^3+x_4^3+x_5^3)\, x_0^4\saf.
\]
The singularity subscheme $JX$ is a non-reduced $3$-dimensional 
subscheme of $\Pbb^5$ supported on the linear subspace $x_0=x_1=0$.
We can use the package~\cite{SegreClassesSource} to compute its Segre 
class:
{\begin{verbatim}
i1 : load("SegreClasses.m2")

i2 : R=ZZ/32749[x0,x1,x2,x3,x4,x5]

i3 : X=ideal(x1^7- x0^7 - (x2^3+x3^3+x4^3+x5^3)*x0^4)

i4 : JX=ideal jacobian X

i5 : segre(JX,ideal(0_R))

            5       4       3      2
o5 = - 3168H  + 792H  - 144H  + 18H
            1       1       1      1
\end{verbatim}
}
\noindent
(Working over a finite field of large characteristic does not affect the result, and
often leads to faster computations.)
Thus,
\[
\iota_* s(JX,\Pbb^5) = 18 [\Pbb^3] -144 [\Pbb^2] +792 [\Pbb^1]-3168 [\Pbb^0]\saf,
\]
and Corollary~\ref{cor:LeSe} yields
\[
\begin{cases}
\lambda^4_X &= {\bf 0} \\
\lambda^3_X &= {\bf 18}\\
\lambda^2_X &= -144+2\cdot 6\cdot 18={\bf 72}\\
\lambda^1_X &= 792+3\cdot6\cdot (-144)+\binom 32 \cdot 36\cdot 18={\bf 144}\\
\lambda^0_X &= -3168+4\cdot 6\cdot 792+\binom 42 \cdot 36\cdot (-144)
+\binom 43 \cdot 216\cdot 18={\bf 288}\saf.
\end{cases}
\]
These L\^e numbers agree with those obtained by applying Massey's inductive 
definition with coordinates $(x_0,\dots, x_5)$; the L\^e cycles are complete intersections
in this case, and computing their degrees is straightforward. (Using $(x_5,\dots, x_0)$
leads to a different list; this latter choice is not sufficiently general.)
\qede\end{example}

We can also projectivize the cycles $\Gamma^k$ appearing in Massey's definition 
(corresponding to the $\rho$-cycles in the St\"uckrad-Vogel algorithm). Again (loosely)
following Massey, we call $\Ggamma^k_X:=\Pbb(\Gamma^{k+1})$ the `projective polar 
cycles' of $X$, and their degrees~$\gamma^k_X$ the 
`polar\index{polar!numbers} 
numbers' of $X$.
We assume these are computed for a general choice of coordinates.

At the level of rational equivalence classes, Massey's algorithm implies easily the relation
\[
\sum_k [\Llambda^k_X] = [\Pbb^n] - (1-(d-1) H) \sum_k [\Ggamma^k_X]
\]
from which $\lambda^n_X=0$ and 
\[
\lambda^k_X = (d-1) \gamma^{k+1}_X-\gamma^k_X 
\]
for $0\le k<n$. Equivalently, 
\[
\gamma^k_X = (d-1)^{n-k} - \sum_{j=k}^{n-1} (d-1)^{j-k} \lambda^j_X
\]
for $0\le k\le n$. (Also see~\cite[Corollary 7.7.3]{masseyhandbook}.)

\begin{corollary}\label{cor:polSe}
With notation as in Corollary~\ref{cor:LeSe}, and for $k=0,\dots,n$:
\begin{align*}
\gamma^k_X &= (d-1)^{n-k}-\sum_{j=k}^n \binom{n-k}{j-k} (d-1)^{j-k} s_j \\
s_k &= \delta_k^n-\sum_{j=k}^n \binom{n-k}{j-k} (-(d-1))^{j-k} \gamma^j_X
\end{align*}
where $\delta_k^n=1$ if $k=n$, $0$ otherwise.
\end{corollary}

\begin{proof}
The first formula is obtained by reading off the coefficient of $[\Pbb^k]$ in the
identity
\begin{equation}\label{eq:gafromse}
\sum_k [\Ggamma^k_X] = (1-(d-1)H)^{-1}\cap 
\left([\Pbb^n]-s(JX,\Pbb^n)\otimes_{\Pbb^n} \cO(-(d-1))\right)\saf,
\end{equation}
which follows from the above discussion and Proposition~\ref{prop:MasSeg}.
Solving for $s(JX,\Pbb^n)$ in~\eqref{eq:gafromse} gives
\[
s(JX,\Pbb^n) = [\Pbb^n] - (1+(d-1)H)^{-1}\cap\sum_k 
([\Ggamma^k_X]\otimes_{\Pbb^n} \cO(d-1))
\]
(use~\eqref{eq:notpr1} and~\eqref{eq:notpr2}) with the stated implication on degrees.
\end{proof}

\begin{example}
For the hypersurface in Example~\ref{ex:Mass}, the computation of the polar numbers
runs as follows.
\[
\begin{cases}
\gamma^5_X &= {\bf 1} \\
\gamma^4_X &= {\bf 6} \\
\gamma^3_X &= 36-18={\bf 18}\\
\gamma^2_X &= 216-(-144)-\binom 31\cdot 6\cdot 18={\bf 36}\\
\gamma^1_X &= 1296-792-\binom 41\cdot6\cdot (-144)-\binom 42 \cdot 36\cdot 18={\bf 72}\\
\gamma^0_X &= 7776-(-3168)-\binom 51\cdot 6\cdot 792-\binom 52 \cdot 36\cdot (-144)
-\binom 53 \cdot 216\cdot 18={\bf 144}\saf.
\end{cases}
\]
Again, it is straightforward to verify that these agree with the result of Massey's algorithm,
applied with coordinates $(x_0,\dots, x_5)$.
\qede\end{example}

\begin{remark}
We already mentioned (Remark~\ref{rem:Pieneremark}) Piene's seminal 1978 
paper~\cite{Pienepolarclasses}, including formulas for polar classes of hypersurfaces 
in terms of Segre classes. The reader is warned that these two uses of the term `polar' 
differ: Piene's polar classes of a hypersurface $X$ are classes in $A_*(X)$, while Massey's
polar cycles are not supported on $X$. Therefore, the degrees of Piene's polar classes
are not the polar numbers $\gamma^k$ computed above.
However, we note that the formula in Corollary~\ref{cor:polSe} is very similar
to the formula in~\cite[Theorem~2.3]{Pienepolarclasses}; the main difference is in the
use of $s(JX,M)$ rather than $s(JX,X)$.
\qede\end{remark}

\subsection{L\^e, Milnor, Segre}\label{ss:LMS}
One moral to be drawn from the preceding considerations is that the information carried
by the L\^e classes
of a hypersurface $X$ of projective space, 
its Milnor class,
and the Segre class
of its singularity subscheme $JX$, is essentially the same. 
The relation between Segre classes and Milnor classes goes back to~\cite{MR96d:14004},
while the relation between Milnor classes and L\^e classes was first studied 
in~\cite{MR3232012, MR3232013}.
As far as hypersurfaces of 
projective space are concerned, many of the results covered in this review could be written in 
terms of any of these notions. Note however that 
extending L\^e cycles/classes to the setting of a hypersurface of a more general nonsingular
variety is nontrivial (this is one of the main goals of~\cite{MR3232012}; and see below); 
localizing Milnor 
classes to the components of the singular locus also requires nontrivial considerations
(see e.g.,~\cite{MR1873009}); while the Segre class of the singularity subscheme $JX$
is naturally defined as a class in the Chow group of $JX$, does not require a projective 
embedding, and may be considered over arbitrary fields. For these reasons, it would
seem that the language of Segre classes is preferable over these alternatives. 

For the convenience of the reader, we collect here the formulas translating these notions 
into one another. For notational economy we will let
\[
\Llambda:=\sum_k [\Llambda^k_X]\quad,\quad
\cM:=\cM(X)\quad,\quad
S:=s(JX,\Pbb^n)
\]
for a degree-$d$ hypersurface $X$ of $\Pbb^n$, and omit evident push-forwards. Then, 
denoting by $H$ the hyperplane class:
\begin{align}
&\label{eq:LS} \left\{\begin{aligned}
\Llambda &= S\otimes_{\Pbb^n} \cO(-(d-1)H) \\
S &= \Llambda\otimes_{\Pbb^n} \cO((d-1)H)
\end{aligned}\right. \\[5pt]
&\label{eq:MS} \left\{\begin{aligned}
\cM &= (-1)^n \dfrac{(1+H)^{n+1}}{1+dH}\cap \left(S^\vee \otimes_{\Pbb^n} \cO(dH)\right) \\
S &= (-1)^n \dfrac{(1+dH)^n}{(1+(d-1)H)^{n+1}}\cap \left(\cM^\vee \otimes_{\Pbb^n} \cO(dH)\right)
\end{aligned}\right. \\[5pt] 
&\label{eq:LM} \left\{\begin{aligned}
\Llambda &= (-1)^n (1+H)^n (1-(d-1)H)\cap  (\cM^\vee\otimes_{\Pbb^n} \cO(H)) \\
\cM &= (-1)^n\dfrac{(1+H)^{n+1}}{1+dH}\cap (\Llambda^\vee\otimes_{\Pbb^n} \cO(H))
\end{aligned}\right. 
\end{align}
Indeed, \eqref{eq:LS} follows from~Proposition~\ref{prop:MasSeg}; 
\eqref{eq:MS} from~Proposition~\ref{prop:Milhyp};
and \eqref{eq:LM} is then an immediate consequence, using~\eqref{eq:notpr1} 
and~\eqref{eq:notpr2}.

This dictionary suggests possible extensions of the notion of L\^e classes to hypersurfaces
of more general varieties. Let $M$ be a nonsingular compact complex variety endowed 
with a very ample line 
bundle $\cO(H)$. For a hypersurface $X$ of $M$, 
Callejas-Bedregal, Morgado, and Seade have constructed 
{\em global L\^e cycles,\/}\index{L\^e!cycles and classes!global}
determined by the choice of linear subspaces
of~$\Pbb^n$, generalizing the case $M=\Pbb^n$; 
see~\cite[Definition~1.3]{MR3232013} 
and~\cite[\S4.3]{milnorhandbook}. Denoting the corresponding 
class $\Llambda_\text{CBMS}(X)$, 
and letting $\cL=\cO(X)$, they prove the following result (which we state using our 
notation).

\begin{theorem}[{\cite[Theorem~4.6]{milnorhandbook}}]
\begin{align*}
\Llambda_\text{CBMS}(X) &= (-1)^{\dim M} c(\cO(H))^{\dim M} \, c(\cO(H)\otimes \cL^\vee)
\cap (\cM(X)^\vee\otimes_M \cO(H)) 
\\
\cM(X) &= (-1)^{\dim M} c(\cO(H))^{\dim M+1}\, c(\cL)^{-1}\cap 
(\Llambda_\text{CBMS}(X)^\vee\otimes_M \cO(H))\saf.
\end{align*}
\end{theorem}
That is, the natural generalization of~\eqref{eq:LM} holds for this class;
the class $\Llambda_\text{CBMS}$ agrees with the class
of Massey's L\^e cycle for $M=\Pbb^n$.

It is straightforward (using Proposition~\ref{prop:Milhyp} and~\eqref{eq:notpr1} 
and~\eqref{eq:notpr2}) to write $\Llambda_\text{CBMS}(X)$ in terms of a
Segre class:
\[
\Llambda_\text{CBMS}(X) = c(\cO(H))\, c(T^\vee M\otimes\cO(H))\cap
\left( s(JX,M)\otimes_M (\cO(H)\otimes \cL^\vee)\right)\saf.
\]
This expression reduces to~\eqref{eq:LS} for $M=\Pbb^n$, and it could
be used to extend the definition of $\Llambda_\text{CBMS}(X)$ to arbitrary fields 
and possibly noncomplete varieties.

There are other possible extensions of Massey's L\^e class to more general 
projective varieties; \eqref{eq:LS} suggests alternative generalizations.
Exploring such alternatives is the subject of current research.


\bibliographystyle{abbrv}
\bibliography{Scaiosvbib}

\end{document}